\documentclass{amsart}

\usepackage{comment}

\usepackage[mathscr]{eucal}
\usepackage{amssymb}
\usepackage{graphicx}
\usepackage{psfrag} 
\usepackage{epstopdf} 
\usepackage[usenames,dvipsnames]{color}
\usepackage[normalem]{ulem}
\usepackage{amsthm}
\usepackage{bbm}
\usepackage{enumerate}
\usepackage{array}
\usepackage{amsmath}
\usepackage{leftindex}

\usepackage{hyperref}
\usepackage[T1]{fontenc}

\hypersetup{colorlinks=true,linkcolor={Brown},citecolor={Brown},urlcolor={Brown}}

\usepackage{cleveref}

\numberwithin{equation}{section}
	\makeatletter
	\def\l@subsection{\@tocline{2}{0pt}{2.5pc}{5pc}{}}

\makeindex

\usepackage[all]{xy}
\CompileMatrices

\newdir{ >}{{}*!/-10pt/\dir{>}}

\swapnumbers 

\newtheorem{Thm}[equation]{Theorem}
\newtheorem*{Thm*}{Theorem}
\newtheorem{Prop}[equation]{Proposition}
\newtheorem{Lem}[equation]{Lemma}
\newtheorem{Cor}[equation]{Corollary}
\newtheorem{Conj}[equation]{Conjecture}
\theoremstyle{remark}
\newtheorem{Def}[equation]{Definition}
\newtheorem{Ter}[equation]{Terminology}
\newtheorem{War}[equation]{Warning}
\newtheorem{Not}[equation]{Notation}
\newtheorem{Exa}[equation]{Example}

\newtheorem{Cons}[equation]{Construction}

\newtheorem*{Conv*}{Conventions}
\newtheorem{Hyp}[equation]{Hypotheses}
\newtheorem{Rec}[equation]{Recollection}

\newtheorem{Rem}[equation]{Remark}

\theoremstyle{definition}
\newtheorem*{Ack*}{Acknowledgements}

\newcommand{\nc}{\newcommand}
\nc{\dmo}{\DeclareMathOperator}

\dmo{\Ab}{\mathsf{Ab}}
\dmo{\AbMon}{AbMon}
\dmo{\Abelem}{Abelem}
\dmo{\Aut}{Aut}
\dmo{\Bi}{bi}
\dmo{\Bisets}{Bisets}
\dmo{\DER}{\mathsf{DER}}
\dmo{\ADDER}{\mathsf{ADDER}}
\dmo{\coev}{coev}
\dmo{\Coloc}{Coloc}
\dmo{\ev}{ev}
\dmo{\Fib}{Fib}
\dmo{\Free}{Free}
\dmo{\Id}{Id}
\dmo{\loc}{Loc}
\dmo{\rmI}{I}
\dmo{\rmL}{L}
\dmo{\rmR}{R}
\dmo{\Spc}{Spc}
\dmo{\thick}{Thick}
\dmo{\chara}{char}%
\dmo{\coh}{coh} 
\dmo{\Coind}{Coind}
\dmo{\coker}{coker}
\newcommand{\colim}{\mathop{\mathrm{colim}}}
\dmo{\cone}{Cone}
\dmo{\Cone}{Cone}
\dmo{\Der}{D}
\dmo{\Ch}{Ch}
\nc{\Rder}{\mathrm{R}} 
\nc{\Lder}{\mathrm{L}} 
\dmo{\Khocat}{K}
\dmo{\End}{End}
\dmo{\Ext}{Ext}
\dmo{\rmH}{H}
\dmo{\Ho}{Ho}
\dmo{\Hom}{Hom}
\dmo{\id}{id}
\dmo{\Img}{Im}
\dmo{\incl}{incl}
\dmo{\Ind}{Ind}
\dmo{\ind}{ind}
\dmo{\Indk}{Ind_{\kappa}}
\dmo{\CoInd}{CoInd}

\dmo{\PSh}{PSh}
\dmo{\Top}{Spt} 
\dmo{\Spt}{Spt}

\dmo{\Ker}{Ker}
\dmo{\Les}{Les}
\dmo{\Map}{Map}%
\dmo{\Mod}{\mathsf{Mod}}
\dmo{\GrMod}{GrMod}
\dmo{\lax}{lax}
\dmo{\modname}{mod}%
\dmo{\grmod}{grmod}
\dmo{\Mor}{Mor}%
\dmo{\Obj}{Obj}
\dmo{\Or}{Or}
\dmo{\Oldorbit}{\mathcal{O}} 
\dmo{\Fix}{Fix} 
\dmo{\Ev}{Ev} 
\dmo{\pr}{pr}
\dmo{\canin}{in} 
\dmo{\Proj}{Proj} 
\dmo{\Inj}{Inj} 
\dmo{\proj}{proj}
\dmo{\Qcoh}{Qcoh}
\dmo{\rank}{rank}
\dmo{\Res}{Res}
\dmo{\Con}{Conj}
\dmo{\res}{res}
\dmo{\Defl}{Def}
\dmo{\Infl}{Inf}
\dmo{\Iso}{Iso}
\dmo{\Rname}{R}
\newcommand{\Sp}{\mathsf{Sp}} 
\nc{\SHp}{\SH_{(p)}}
\dmo{\smallb}{b}
\dmo{\smallperf}{perf}
\dmo{\Spec}{Spec}
\dmo{\Spech}{Spec^h}
\dmo{\Stab}{Stab}
\dmo{\stab}{stab}
\dmo{\supp}{Supp}
\dmo{\switch}{switch}
\dmo{\TTR}{TTR}
\dmo{\Spanname}{{\sf Span}}
\dmo{\map}{map}
\dmo{\Rel}{Rel}

\nc{\Ivo}[1]{{\color{OliveGreen}#1}}
\nc{\Ruben}[1]{{\color{Blue}#1}}
\nc{\Rout}[1]{\Ruben{\sout{#1}}}
\nc{\Iout}[1]{\Ivo{\sout{#1}}}

\nc{\SEcell}{\rotatebox[origin=c]{45}{$\Downarrow$}} 
\nc{\NEcell}{\rotatebox[origin=c]{135}{$\Downarrow$}} 
\nc{\SWcell}{\rotatebox[origin=c]{-45}{$\Downarrow$}} 
\nc{\NWcell}{\rotatebox[origin=c]{-135}{$\Downarrow$}} 
\nc{\Scell}{\rotatebox[origin=c]{0}{$\Downarrow$}} 
\nc{\Ncell}{\rotatebox[origin=c]{0}{$\Uparrow$}} 
\nc{\Ecell}{\Rightarrow}
\nc{\oEcell}[1]{\overset{\scriptstyle #1}{\Ecell}}
\nc{\isoEcell}{\overset{\sim}{\,\Ecell\,}}
\nc{\isocell}[1]{\undersett{ #1}\isoEcell}
\nc{\Isocell}[1]{\undersett{ #1}{\overset{\sim}{\Longrightarrow}}}
\nc{\Wcell}{\rotatebox[origin=c]{90}{$\Uparrow$}} 
\nc{\Span}{\Spanname}
\nc{\Spanhat}{\textrm{\sf S}\widehat{\textrm{\sf pan}}} 
\nc{\tSpan}{\pih{\Spanname}}
\nc{\IFF}{$\Leftrightarrow$}
\nc{\ass}{\mathrm{ass}} 
\nc{\lun}{\mathrm{lun}} 
\nc{\run}{\mathrm{run}} 
\nc{\fun}{\mathrm{fun}} 
\nc{\un}{\mathrm{un}} 
\nc{\Crich}{\underline{\cat{C}}}
\nc{\uA}{\underline{A}}
\nc{\doublequot}[3]{#1\backslash #2/#3}
\nc{\HGK}{\doublequot HGK}
\nc{\quadtext}[1]{\quad\textrm{#1}\quad}
\nc{\qquadtext}[1]{\qquad\textrm{#1}\qquad}
\nc{\PZG}{\cat{C}_{\bbZ}(\bbZ G)}
\nc{\TTRK}{\TTR(\cat K)}
\nc{\psets}{\mathsf{-sets}_\sbull}
\nc{\Gsets}{G\mathsf{-sets}}
\nc{\Hsets}{H\mathsf{-sets}}
\nc{\AddK}{\Add^{\Sigma}(\cat K)}
\nc{\adj}{\dashv\,}
\nc{\adjto}{\rightleftarrows}
\nc{\AK}{A\MModcat{K}}
\nc{\BK}{B\MModcat{K}}
\nc{\bbA}{\mathbb{A}}
\nc{\bbB}{\mathbb{B}}
\nc{\bbC}{\mathbb{C}}
\nc{\bbD}{\mathbb{D}}
\nc{\bbF}{\mathbb{F}}
\nc{\bbI}{\mathbb{I}}
\nc{\bbM}{\mathbb{M}}
\nc{\bbN}{\mathbb{N}}
\nc{\bbP}{\mathbb{P}}
\nc{\bbQ}{\mathbb{Q}}
\nc{\bbR}{\mathbb{R}}
\nc{\bbZ}{\mathbb{Z}}
\nc{\bbZp}{\mathbb{Z}_{(p)}}
\nc{\Sphere}{\mathbb{S}} 
\nc{\cat}[1]{\mathscr{#1}}
\nc{\Displ}{\displaystyle}
\nc{\ie}{{\sl i.e.}\ }
\nc{\cf}{{\sl cf.}\ }
\nc{\into}{\mathop{\rightarrowtail}}
\nc{\inv}{^{-1}}
\nc{\isoto}{\buildrel \sim\over\to}
\nc{\isotoo}{\mathop{\buildrel \sim\over\too}}
\nc{\onto}{\mathop{\twoheadrightarrow}}
\nc{\too}{\mathop{\longrightarrow}\limits}
\nc{\xytriangle}[7]{\xymatrix@C=#7em{#1\ar[r]^-{\Displ #4} & #2 \ar[r]^-{\Displ #5}&#3\ar[r]^-{\Displ #6}&T #1}}
\nc{\ababs}{{\sl ab absurdo}}
\nc{\adh}[1]{\overline{#1}}
\nc{\adhoc}{{\sl ad hoc}}
\nc{\adhpt}[1]{\adh{\{#1\}}}
\nc{\afortiori}{{\sl a fortiori}}
\nc{\aka}{{a.\,k.\,a.}\ }
\nc{\ala}{{\sl \`a la}\ }
\nc{\apriori}{{\sl a priori}}
\nc{\Autcat}[1]{\Aut_{\cat #1}}
\nc{\cO}{\mathcal{O}}
\nc{\calO}{\mathcal{O}}
\nc{\eg}{{\sl e.g.}}
\nc{\eps}{\varepsilon}
\nc{\equalby}[1]{\overset{\textrm{#1}}{=}}
\nc{\gm}{\mathfrak{m}}
\nc{\Homcat}[1]{\Hom_{\cat #1}}
\nc{\Morcat}[1]{\Mor_{\cat #1}}
\nc{\hook}{\hookrightarrow}

\nc{\Idcat}[1]{\Id_{\cat{#1}}}
\nc{\ideal}[1]{\langle #1\rangle}
\nc{\ihom}{{\mathsf{hom}}} 
\nc{\ihomcat}[1]{\ihom_{\cat #1}}
\nc{\Kcat}[1]{#1\MModcat{K}}
\nc{\KP}{\cat{K}_{\cat P}}
\nc{\loccit}{{\sl loc.\ cit.}}
\nc{\lind}{\rmL\!}
\nc{\RR}{\rmR\!}
\nc{\Lotimes}{\otimes^{\rmL}}
\nc{\Mid}{\,\bigm|\,}
\nc{\MMod}{\,\textsf{-}\Mod}%
\nc{\Exact}{\mathfrak K\,\text{-}\mathrm{Exa}^\mathbb Z/2_\infty} 
\nc{\MModcat}[1]{\MMod_{\cat #1}}%
\nc{\mmod}{\,\text{--}\modname}%
\nc{\mmodb}{\mmod^\sbull}%
\nc{\op}{{\mathrm{op}}}
\nc{\co}{{\mathrm{co}}}
\nc{\costar}{**}
\nc{\oto}[1]{\overset{#1}\to}
\nc{\ointo}[1]{\overset{#1}{\rightarrowtail}}
\nc{\loto}[1]{\overset{#1}{\leftarrow}}
\nc{\otoo}[1]{\overset{#1}{\,\longrightarrow\,}}
\nc{\lotoo}[1]{\overset{#1}{\,\longleftarrow\,}}
\nc{\ourfrac}[2]{\genfrac{}{}{0pt}{}{\Displ #1}{\scriptstyle #2}}
\nc{\ouriff}{\Leftrightarrow}
\nc{\oursetminus}{\!\smallsetminus\!}
\nc{\potimes}[1]{^{\otimes #1}}
\nc{\pproj}{\,\text{-}\proj}
\nc{\ptimes}[1]{^{\times #1}}
\nc{\dd}[1]{_{{\scriptscriptstyle(#1)}}}
\nc{\uu}[1]{^{{\scriptscriptstyle(#1)}}}
\nc{\pushout}{\textrm{\rm p.o.}}
\nc{\qp}{q_{_{\scriptstyle \cat P}}\!}%
\nc{\Rcat}[1]{\Rname_{\cat #1}^\sbull}
\nc{\rdto}{}
\nc{\restr}[1]{{|_{\scriptstyle #1}}}
\nc{\RK}{\Rcat{K}}
\nc{\sbull}{{\scriptscriptstyle\bullet}}
\nc{\SET}[2]{\bigl\{\,#1\Mid#2\,\bigr\}}
\nc{\SHA}{\SH{}^{\bbA^{1}}}
\nc{\SHfin}{\SH^{\text{\rm fin}}}
\nc{\smat}[1]{\left(\begin{smallmatrix} #1 \end{smallmatrix}\right)}
\nc{\SpcAK}{\Spc(A\MModcat{K})}
\nc{\SpcK}{\Spc(\cat K)}
\nc{\suppcat}[1]{\supp(\cat #1)}
\nc{\then}{\Rightarrow}
\nc{\tideal}[1]{\ideal{#1}}
\nc{\unit}{\mathbbm{1}}
\nc{\unitcat}[1]{\unit_{\cat #1}}
\nc{\vcorrect}[1]{{\vphantom{\vbox to #1em{}}}}
\nc{\onept}{\mathrm{B}} 
\nc{\undersett}[1]{\underset{\scriptstyle #1}}
\nc\noloc{\nobreak\mspace{6mu plus 1mu}{:}\nonscript\mkern-\thinmuskip\mathpunct{}\mspace{2mu}}

\nc{\HG}{\!{}^{^H}\overline{G}}
\nc{\uY}{\widetilde{Y}}

\nc{\ADD}{\mathsf{ADD}}
\nc{\MONADD}{\mathsf{MONADD}}
\nc{\SMONADD}{\mathsf{SMONADD}}
\nc{\Add}{\mathsf{Add}}
\nc{\SAD}{\mathsf{SAD}}
\nc{\Sad}{\mathsf{Sad}}
\nc{\Cat}{\mathsf{Cat}}
\nc{\CCat}{\textsf{-}\mathsf{Cat}}
\nc{\CAT}{\mathsf{CAT}}
\nc{\Dk}{\dual_{\kappa}}
\nc{\Dkk}{\dual_{\kappa'}}
\nc{\bs}{\backslash}
\nc{\biCpt}{\mathrm{biCpt}}
\nc{\biLCpt}{\mathrm{biLCpt}} 
\nc{\Groupoid}{\mathsf{Groupoid}}
\nc{\groupoid}{\mathsf{gpd}}
\nc{\gpd}{\groupoid}
\nc{\faithful}{\mathsf{faithful}}
\nc{\faith}{\mathsf{faithf}}
\nc{\exact}{\mathsf{ex}}
\nc{\smallfaithful}{\mathsf{f}}
\nc{\smallfused}{\mathsf{fus}}
\nc{\groupoidf}{\groupoid{}^{\smallfaithful}}
\nc{\groconn}{\groupoid_{\mathsf{conn}}}
\nc{\gps}{\mathsf{groups}} 
\nc{\group}{\mathsf{group}} 
\nc{\groupshort}{\mathsf{gr}}
\nc{\gpdG}{{\groupoidf_{\!\smallslash\!G}}} 
\nc{\GinG}{{\groupoidf_{G}}}
\nc{\gpdGfuz}{{\groupoid^{\smallfused}_{\!\smallslash\!G}}} 
\nc{\spanG}{{\widehat{\mathsf{gp}\,\,}\!\!\mathsf{d}}{}^\smallfaithful_{\!{}^{\scriptscriptstyle/}\!G}}
\nc{\biset}{\mathsf{biset}} 
\nc{\rfree}{\mathsf{rf}} 
\nc{\bifree}{\mathsf{bif}} 
\nc{\conj}{\mathrm{conj}} 
\nc{\smallslash}{{}^{\scriptscriptstyle/}}
\nc{\smallbs}{{}^{\scriptscriptstyle\backslash}}
\nc{\doublebs}{\smallbs\!\smallbs}

\nc{\Set}{\mathsf{Set}}
\nc{\set}{\mathsf{set}} 
\nc{\sset}{\textrm{-}\set}
\nc{\ssetfused}{\textrm{-}\underline{\set}} 
\nc{\ssetfuz}{\sset^{\smallfused}} 
\nc{\Comp}{\mathsf{Top}^{\mathsf{comp}}}
\nc{\pih}[1]{\tau_{1}#1}
\nc{\all}{\mathsf{all}}

\dmo{\Fun}{\mathrm{Fun}} 
\dmo{\PsFun}{\mathsf{PsFun}} 
\dmo{\PsFunlax}{\mathsf{PsFun}_{\mathsf{lax}}}
\dmo{\PsFunoplax}{\mathsf{PsFun}_{\mathsf{oplax}}}
\dmo{\BCDex}{\mathsf{BCDex}_{\II\mathsf{-str}}}
\dmo{\BCDexdex}{\mathsf{BCDex}_{\II\mathsf{-dex}}}
\dmo{\biMack}{\mathsf{Mack}} 
\dmo{\Mackey}{\mathsf{Mack}} 
\nc{\MMackey}{\,\textsf{-}\Mackey} 
\dmo{\twoFun}{2\mathsf{Fun}}

\nc{\Muniv}{\cat{M}^{\mathsf{univ}}}

\nc{\lG}{{}_{{\color{Gray}\scriptscriptstyle G}}}
\nc{\lH}{{}_{{\color{Gray}\scriptscriptstyle H}}}
\nc{\rG}{_{{\color{Gray}\!\scriptscriptstyle G}}}
\nc{\rH}{_{{\color{Gray}\!\scriptscriptstyle H}}}
\nc{\rK}{_{{\color{Gray}\!\scriptscriptstyle K}}}

\nc{\dual}{\Delta}

\nc{\ra}{\rightarrow}
\nc{\xra}{\xrightarrow}
\nc{\lto}{\leftarrow}
\nc{\olto}[1]{\overset{#1}\lto}

\nc{\C}{\mathbb{C}} 
\nc{\Cont}{\mathrm{C}} 
\nc{\A}{\mathrm{A}}
\nc{\Rep}{\mathrm{R}} 
\nc{\KK}{\mathsf{KK}} 
\nc{\Calg}{\mathsf{C^*\text{-}Alg}} 
\nc{\Kth}{\mathrm{K}} 
\nc{\Cell}[1]{\mathsf{Cell}(#1)}
\nc{\bigCell}[1]{\mathsf{Cell}(#1)_\mathrm{big}}
\nc{\bigC}{{\cat{C}_\mathrm{big}}} 
\nc{\bigT}{\cat{T}_\mathrm{big}} 
\nc{\Modules}{\mathsf{Mod}}
\nc{\stmod}{\mathsf{stmod}} 
\nc{\SH}{\mathsf{SH}} 
\nc{\Alg}{\mathsf{Alg}}
\nc{\Sep}{\mathsf{Sep}}
\nc{\BurnG}{\cat{A}(G)}
\nc{\Loc}[1]{\loc\langle\, #1 \,\rangle}
\nc{\Locc}[1]{\loc\langle\, #1 \,\rangle_{\aleph_1}}
\nc{\Lock}[1]{\loc\langle\, #1 \,\rangle_\kappa}
\nc{\Thick}[1]{\thick\langle\, #1 \,\rangle}
\nc{\tensLoc}[1]{\loc_\otimes\!\langle\, #1 \,\rangle}
\nc{\tensLocc}[1]{\loc_\otimes\!\langle\, #1 \,\rangle_{\aleph_1}}
\nc{\tensThick}[1]{\thick_\otimes\!\langle\, #1 \,\rangle}
\nc{\Cstar}[1]{\mathrm{C}^*\mathsf{alg}^{#1}}
\nc{\Cstarsep}[1]{\mathsf{C}^*\mathsf{sep}^{#1}}

\nc{\inftyKK}[1]{\mathbf{KK}^{#1}} 
\nc{\inftyKKloc}[1]{\mathbf{KK}^{#1}_{\oplus}} 
\nc{\inftyKKsep}[1]{\mathbf{KK}^{#1}_{\mathsf{sep}}} 
\nc{\inftyInd}[1]{\mathbf{Ind}(#1)}
\nc{\inftyCell}[1]{\mathbf{Cell}(#1)}
\nc{\inftyCellsep}[1]{\mathbf{Cell}(#1)_{\mathsf{sep}}}
\nc{\inftyFun}{\mathrm{Fun}}
\nc{\kk}[1]{\mathrm{kk}^{#1}}
\nc{\kksep}[1]{\mathrm{kk}_{\mathrm{sep}}^{#1}}

\nc{\cD}{\cat{D}}
\nc{\cG}{\cat{G}}
\nc{\cM}{\cat{M}}
\nc{\cN}{\cat{N}}

\nc{\DD}{\cat{D}}
\nc{\MM}{\cat{M}}
\nc{\NN}{\cat{N}}
\nc{\GG}{\mathbb{G}}

\nc{\gammap}[1]{\gamma^{(#1)}}
\nc{\what}[1]{\widehat{\cat{#1}}}

\nc{\und}[1]{{\kern1pt\underline{\kern-1pt{#1}\kern-1.5pt}\kern1.5pt}}

\nc{\Funplus}{\Fun_{+}}
\nc{\Mack}[1]{(Mack\,\ref{Mack-#1})}

\begin{document}


\title{Stratification in equivariant Kasparov theory}
\author{Ivo Dell'Ambrogio}
\author{Rub\'en Martos}
\date{\today}

\address{
\noindent Univ.\,Artois, UR\,2462, Laboratoire de Math\'ematiques de Lens, F-62300 Lens, France
}
\email{ivo.dellambrogio@univ-artois.fr}
\urladdr{https://idellambrogio.github.io/}

\address{
\noindent  Univ.\,Lille, CNRS, UMR\,8524 - Laboratoire Paul Painlev\'e, F-59000 Lille, France
}
\email{ruben.martos2@univ-lille.fr}
\urladdr{https://sites.google.com/view/ruben-martos/home}

\begin{abstract}
We study stratification, that is the classification of localizing tensor ideal subcategories by geometric means, in the context of Kasparov's equivariant KK-theory of C*-algebras.
We introduce a straightforward countable analog of the notion of stratification by Balmer--Favi supports and conjecture that it holds for the equivariant bootstrap subcategory of every finite group~$G$. 
We prove this conjecture for groups whose nontrivial elements all have prime order, and we verify it rationally for arbitrary finite groups. 
In all these cases we also compute the Balmer spectrum of compact objects.
In our proofs we use larger versions of the equivariant Kasparov categories which admit not only countable coproducts but all small ones; they are constructed in an Appendix using $\infty$-categorical enhancements and adapting ideas of Bunke--Engel--Land.
\end{abstract}


\thanks{The authors were supported by the Labex CEMPI (ANR-11-LABX-0007-01)}

\maketitle

\tableofcontents

\section{Introduction}
\label{sec:intro}%

\subsection*{Kasparov theory and tt-geometry}
Tensor triangular geometry (\emph{tt-geometry} for short), started by Balmer \cite{Balmer05a}, is based on the idea of exploiting the ambient geometry of any given essentially small tensor triangulated category (\emph{tt-category}) $\cat K$ in order to provide a classification of its \emph{thick tensor ideal subcategories}, and therefore a rough classification of its objects. 
To this end, Balmer introduces a topological space $\text{Spc}(\mathscr{K})$, the \emph{spectrum of~$\cat K$}, and a universal notion of supports $\supp(A)\subseteq \Spc (\cat K)$ for the objects $A$ of~$\cat K$.
This theory illuminates and unifies several classification problems in algebraic geometry, topology, representation theory, and beyond.
In each example of~$\cat K$, the classification proposed by Balmer (which always works abstractly) becomes especially useful if one manages to find a description of the space $\Spc(\cat K)$ and of the support theory $A\mapsto \supp(A)$ which is relevant to the example at hand. 
We refer to \cite{Balmer20} for a rich survey of concrete computations of spectra throughout mathematics.

The application of tt-geometry to noncommutative geometry was initiated by the first author \cite{DellAmbrogio10}. The starting point is the observation, due to Meyer and Nest \cite{MeyerNest06}, that the $G$-equivariant Kasparov category $\KK^G$ of any (second countable) locally compact group~$G$ is an essentially small tt-category equipped with arbitrary countable coproducts.  
Meyer and Nest used the tt-structure in order to reformulate the Baum-Connes conjecture and establish various stability properties for it.
In \cite{DellAmbrogio10}, it was pointed out that sufficient knowledge of the \emph{spectrum} of $\KK^G$ already implies the conjecture. 
Indeed, the Baum-Connes conjecture holds for $G$ as soon as $\Spc(\KK^G)$ is covered (via the natural maps) by the spaces $\Spc(\KK^H)$ with $H$ running through the compact subgroups of~$G$. This result indicates the interest of computing Balmer spectra of equivariant Kasparov categories. 

However, there are serious obstacles to studying $\Spc(\KK^G)$ concretely, because as a tt-category $\KK^G$ is rather unwieldy and falls outside the reach of the usual techniques.
To start with, $\KK^G$ is not a rigid tensor category, \ie its objects do not admit tensor duals in general.
Moreover, it appears not to admit any reasonable generating set, in any useful sense; it seems to have far too many complicated objects.
Both problems already occur when $G$ is the trivial group!

We clearly need to do better. 
And we can: in a companion paper \cite{DellAmbrogioMartos24pp} we explain how to derive the Baum-Connes conjecture from a different covering criterion, where the Balmer spectrum is replaced by a countable variant and where $\KK^G$ and $\KK^H$ are replaced by the decidedly more tractable subcategories $\Cell{G}$ and $\Cell{H}$ of \emph{cell algebras} introduced in~\cite{DellAmbrogio14}. 
For any group~$G$, by definition $\Cell{G}$ is the localizing (`bootstrap') subcategory of $\KK^G$ generated by the orbit algebras $\Cont_0(G/H)$ of all subgroups $H\leq G$; it is again a tt-category at least when $G$ is discrete or compact.
The computation of the above-mentioned countable version of the spectrum is closely related to the problem of \emph{stratification} for $\Cell{G}$, which is the focus of the present paper. 
Our present goal is to establish that, for cell algebras, stratification can be established at least in some interesting cases.

\subsection*{Stratification}
Generally speaking, stratification refers to a set of techniques for classifying \emph{localizing} tensor ideals in a large tt-category, \ie those thick ideals which are also closed under the formation of infinite coproducts. 
A typical situation for applying tt-geometry is when the essentially small tt-category $\cat K$ occurs as the subcategory $\cat T^c$ of `small' objects inside a large tt-category~$\cat T$ which admits arbitrary small coproducts and is itself generated (as a localizing subcategory) by~$\cat K$; the best case is when two notions of smallness, compactness and rigidity, agree in~$\cat T$, in which case one says that $\cat T$ is \emph{rigidly-compactly generated}.
In this situation, tensor-triangular stratification is now fairly well understood thanks to work of Barthel, Heard and Sanders~\cite{BHS23} (building on Neeman~\cite{Neeman92a}, Palmieri--Hovey--Strickland \cite{HoveyPalmieriStrickland97}, Benson--Iyengar--Krause \cite{BensonIyengarKrause08} and others), at least under the mild hypothesis that the spectrum $\Spc(\cat T^c)$ of the rigid-compact objects is known to be a \emph{weakly noetherian} space (\eg\ it is \emph{noetherian}: every open subset is quasi-compact).
In this case, such a big tt-category $\cat T$ is said to be \emph{stratified} if there is a certain canonical inclusion-preserving bijection (induced by an extension of support theory  to big objects due to Balmer--Favi~\cite{BalmerFavi11})
\begin{equation} \label{eq:usual-strat}
\left\{
\begin{array}{c} 
\textrm{localizing tensor ideals of } \cat T
\end{array}
\right\}
\overset{\sim}{\longrightarrow}
\left\{
\begin{array}{c} 
\textrm{subsets of } \Spc(\cat T^c)
\end{array}
\right\}
\end{equation}
between the localizing tensor ideals of $\cat T$ and arbitrary subsets of the Balmer spectrum of the tt-subcategory $\cat T^c$ of the compact-rigid objects in~$\cat T$. 

\subsection*{Results for finite groups}
We want to establish stratification for the category $\Cell{G}$, but we cannot directly apply the Barthel--Heard--Sanders (BHS) theory.
Firstly, $\Cell{G}$ only admits countable coproducts, not arbitrary small ones. 
Secondly, rigid and compact objects in $\Cell{G}$ do not always agree; they do agree when $G$ is finite, but not \eg\ when $G$ is infinite discrete.
For $G$ finite however, $\Cell{G}$ is known to be \emph{rigidly-compactly${}_{\aleph_1}$ generated}, which is a useful replacement of the usual notion for tt-categories with only countable (=~$\aleph_1$-small) coproducts. 
(For $G$ finite moreover, $\Cell{G}$ coincides with the a~priori larger \emph{equivariant bootstrap category} of~\cite{DEM14}, a fact which makes stratification in itself more interesting.)

For such categories $\cat T$, one can easily define a countable version of the Balmer--Favi support theory (see~\Cref{sec:countable-strat}), and we say that $\cat T$ is \emph{countably stratified} if the analogue of the  bijection \eqref{eq:usual-strat} holds, where of course now we read `localizing' as meaning `triangulated and closed under countable coproducts'.

We can now formulate our main result, 
where $\Rep(G)$ denotes the complex representation ring of~$G$:

\begin{Thm}[see \Cref{Thm:strat-big-one}]
\label{Thm:strat-Cell(C_p)}
If $G$ is a finite group in which every nontrivial element has prime order, the category $\Cell{G}$ of $G$-cell algebras is countably stratified. 
In particular, there is a canonical bijection  
\[
\left\{
\begin{array}{c} 
\textrm{localizing subcategories} \\
\textrm{of } \Cell{G}
\end{array}
\right\}
\overset{\sim}{\longrightarrow}
\left\{
\begin{array}{c} 
\textrm{subsets of the Zariski spectrum} \\ 
\Spec( \Rep(G))
\end{array}
\right\}
\]
sending a localizing subcategory $\cat L$ to its countable Balmer--Favi support $\supp (\cat L)$.
\end{Thm}

The above hypothesis on $G$ is certainly very restrictive but it still applies to many groups, such as $S_3$, $A_5$ and all $p$-groups of exponent~$p$ (see \Cref{Rem:our-class}).

In this context $\Rep(G)$ occurs as the endomorphism ring of the trivial $G$-C*-algebra $\mathbb C=\Cont_0(G/G)$, which is the tensor unit of $\Cell{G}$. 
The above form of the bijection is a consequence of the following computation of the Balmer spectrum:
\begin{Thm} [{see \Cref{Thm:Spc-prime-order}}]
\label{Thm:Spc-Intro}
For finite groups $G$ with only prime-order nontrivial elements, 
we have a canonical homeomorphism 
\[\Spc(\Cell{G}^c) \cong \Spec(\End(\unit)) = \Spec(\Rep(G)).\]
\end{Thm}
We do not know how $\Spc(\Cell{G}^c)$ looks in general, in fact this is a major obstacle to establishing stratification for more general groups.
However, the above homeomorphism is provided by the canonical continuous `comparison'  map 
\[
\rho\colon \Spc(\Cell{G}^c) \longrightarrow \Spec(\End(\unit)) = \Spec(\Rep(G))
\]
from the tt-spectrum to the Zariski spectrum, which always exists and which moreover is surjective for all finite (or even compact) groups~$G$; see \cite{Balmer10b}. 
Accordingly, we make the following conjectures:

\begin{Conj}
\label{Conj:spc}
Balmer's comparison map $\rho$ is a homeomorphism $\Spc(\Cell{G}^c)\cong \Spec(\Rep(G))$ for every finite group~$G$.
\end{Conj}

\begin{Conj}
\label{Conj:strat}
$\Cell{G}$ is countably stratified for every finite group~$G$.
\end{Conj}

Note that $\Spec(\Rep(G))$ is a noetherian space, since $\Rep(G)$ is a noetherian commutative ring. 
For this and other reasons, establishing \Cref{Conj:spc} would be a good start towards proving \Cref{Conj:strat} (although not formally necessary).

\begin{Rem} \label{Rem:boot}
When $G=1$ is the trivial group, $\Cell{G}$ is the classical Rosenberg--Schochet bootstrap category, for which both conjectures are already known to be true. 
This is essentially the content of \cite{DellAmbrogio11} (see also the proof of \Cref{Prop:strat-boot}).
The crucial case of \Cref{Thm:Spc-Intro} when $G$ is cyclic of prime order was proved in \cite{DellAmbrogioMeyer21}, and the more general case in our theorem will be derived from the latter thanks to a strong result of Arano--Kubota~\cite{AranoKubota18} (see \Cref{Thm:cyclic-generation}). 
\end{Rem}

Beyond the above special class of groups, we can verify both conjectures \emph{rationally} for arbitrary finite~$G$.
Let $\Rep(G)_\mathbb Q$ be the rationalization of the representation ring, and $\Cell{G}_\mathbb Q$ the rationalization of $\Cell{G}$ (\Cref{Ter:rational_cell}).

\begin{Thm}[see \Cref{Thm:strat-rational} and \Cref{Cor:spc-rational}]
\label{Thm:intro-rational}
For every finite group~$G$, 
$\Cell{G}_\mathbb Q$ is countably stratified and 
the comparison map $\rho$ is a homeomorphism $\Spc(\Cell{G}_{\mathbb{Q}}^c)\cong \Spec(\Rep(G)_{\mathbb{Q}})$.
The latter space is discrete and finite, and its points are in bijection with the conjugacy classes of cyclic subgroups of~$G$.
\end{Thm}

This is a relatively straightforward consequence of a result proved in \cite{BDM24pp}, according to which $\Cell{G}_\mathbb Q$ splits as a finite product of graded module categories over certain rational skew group algebras.
In particular $\Cell{G}_\mathbb Q$ is equivalent to an explicitly described semisimple abelian tensor category so its tt-geometry becomes rather straightforward to work out.
(In view of the closely related results of \cite{MeyerNadareishvili24pp}, we believe the analog of \Cref{Thm:intro-rational} should be true already after inverting $|G|$ instead of rationalizing; but we have not verified all details.)

The proof of \Cref{Thm:strat-Cell(C_p)} is considerably more difficult. 

\subsection*{Adjoining arbitrary small coproducts}
We had a choice of how to proceed, indeed of how to approach the problem of establishing stratification of $\Cell{G}$ for general finite groups. 
On the one hand, we could have developed a countable version of the powerful stratification theory of Barthel--Heard--Sanders so it applies to rigidly-compactly${}_{\aleph_1}$ generated tt-categories such as $\Cell{G}$. 
This would have been straightforward but exceedingly long and ultimately not so interesting, with the added problem that \emph{some} techniques and results would necessarily fail (because of the lack of a countable version of Brown representability for covariant homological functors).
On the other hand we could try to embed $\Cell{G}$ into a larger tt-category $\bigCell{G}$ which admits \emph{all} set-indexed coproducts and is genuinely rigidly-compactly generated, in the usual sense, and which therefore is directly amenable to the Barthel--Heard--Sanders techniques.

We have chosen the second approach, because it is more efficient and also because the tt-category $\bigCell{G}$ is of independent interest.
In \Cref{sec:app-big}, we construct $\bigCell{G}$ by closely adapting a recent construction of Bunke--Engel--Land \cite{BEL23pp} which make heavy use of Lurie's $\infty$-categories. 
To wit, we use general $\infty$-categorical techniques to define a certain presentably symmetric monoidal stable $\infty$-category $\inftyCell{G}$ and set $\bigCell{G}:= \Ho \inftyCell{G}$ to be its homotopy category. (For future reference, we actually construct the analog enlargement of the whole $\KK^G$ and do so for general countable groups.) 
For this enlargement to be useful to us, we must take care to extend the usual functorial properties of KK-theory to the larger category, and we must also ensure that the embedding $\Cell{G}\hookrightarrow \bigCell{G}$ preserves all the tt-structure; in particular, it should preserve the countable coproducts of $\Cell{G}$ and it should identify the two tt-subcategories of compact-rigid objects.

Once all this is done, we can prove that $\bigCell{G}$ is stratified (by the usual Balmer--Favi support theory) as planned, by applying a results from \cite{BHS23} and its sequels; see \Cref{Thm:strat-big-one}. 
Finally, by playing with the embedding we deduce from this the wished countable stratification of the original tt-subcategory~$\Cell{G}$.
This last deduction is not entirely trivial and requires some care, but the argument (which requires results form the theory of Neeman's well-generated triangulated categories) is very general and may be of independent interest; we dedicate \Cref{sec:app-comparison} to it.

\begin{Rem}
Results similar to ours have been recently established in topology. 
To wit, the analogues of Conjectures \ref{Conj:spc} and~\ref{Conj:strat} have been proved (for all finite~$G$) in \cite[Lem.\,8.11 and Thm.\,8.8]{BCHNP23pp} for the module category $\Mod_{KU_G}(\Sp^G)$; here $KU_G$ denotes $G$-equivariant topological complex K-theory seen as a commutative algebra in the symmetric monoidal $\infty$-category $\Sp^G$ of genuine $G$-spectra. This module category (or rather its homotopy category) has very similar properties to $\Cell{G}$.
In conversation with Tobias Barthel, Markus Hausmann and Maxime Ramzi, we were made aware of the possibility that our tt-$\infty$-category $\inftyCell{G}$ and $\Mod_{KU_G}(\Sp^G)$ could actually be equivalent. 
This would immediately imply both of our conjectures and would be very interesting in itself, by providing a robust bridge between equivariant topology and equivariant KK-theory.
\end{Rem}

\begin{Ack*}
We thank Ulrich Bunke for instructive conversations on the matter of \Cref{sec:app-big}, as well as  Henning Krause and Akhil Mathew for useful exchanges on the matter of \Cref{sec:app-comparison}.
\end{Ack*}

\section{Tensor triangular preliminaries}
\label{sec:prelim}%

We briefly review some basics on tensor triangulated categories and their geometry, referring to \cite{BalmerICM} for more background and plenty of examples, and to the further references given along the way for more details.
Our goal is simply to establish the language and collect results for ease of reference, hence the more tt-experienced readers may want to skip this section.

\subsection{Tensor triangulated categories} 
\label{subsec:ttcats}
A \emph{tensor triangulated category}, or just \emph{tt-category}, $\cat K$ is a triangulated category in the sense of Verdier (see \cite{Neeman01}) equipped with a \emph{tensor product}, that is a symmetric monoidal structure which is exact in both variables. We write $\Sigma\colon \cat K\overset{\sim}{\to} \cat K$ for the suspension functor of the triangulation, $\otimes\colon \cat K\times \cat K\to \cat K$ for the tensor functor, and $\unit$ or $\unit_\cat K$ for its tensor unit object.
The tt-category $\cat K$ is \emph{essentially small} if it is so as a category (it is equivalent to one with only a set of objects and morphisms).
It is \emph{rigid} if it is so as a symmetric monoidal category, meaning that every object $A\in \cat K$ is \emph{rigid} (or \emph{dualizable}): there exists a $B\in \cat K$ (the \emph{tensor dual} of~$A$) together with maps $\unit \to A\otimes B$ and $B\otimes A\to \unit$ satisfying the zig-zag equations which, in particular, imply that the endofunctors $A\otimes-$ and $B\otimes -$ of $\cat K$ are adjoint to one another on both sides.

A \emph{tensor triangulated functor}, or \emph{tt-functor}, is a functor $F\colon \cat K\to \cat L$ between two tt-categories which is exact and symmetric monoidal.
A \emph{tensor ideal} (or $\otimes$-ideal) $\cat J\subseteq \cat K$ is a triangulated subcategory (all our triangulated subcategories will be full) such that $A\otimes B \in \cat J$ whenever $A\in \cat J$ and $B\in \cat K$. A tensor ideal is \emph{thick} if it is so as a triangulated subcategory, that is if it is closed under retracts.
The full kernel on objects $\ker(F)= \{A\in \cat K\mid F(A)\cong 0\}\subseteq \cat K$ of any tt-functor $F\colon \cat K\to \cat L$ is a thick tensor ideal of~$\cat K$. 
Conversely, the Verdier quotient $\cat K/\cat J$ by a (thick) tensor ideal inherits a unique structure of tt-category making the canonical functor $\cat K\to \cat K/\cat J$ a tt-functor.

\begin{Not}
We write $\Thick{\cat E}$ and $\tensThick{\cat E}$ for the thick subcategory, respectively the thick tensor ideal, of $\cat K$ generated by a family of objects $\cat E\subseteq \cat K$.
\end{Not}

\subsection{The Balmer spectrum}
\label{subsec:Spc}
For any essentially small tt-category $\cat K$, its \emph{(triangular or Balmer) spectrum} is the topological space whose underlying set is
\[
\Spc(\cat K) := \{ \cat P \subsetneq \cat K\mid \cat P \textrm{ is a prime thick tensor ideal }\}
\]
(where \emph{prime} means $A\otimes B \in \cat P \Rightarrow A,B\in \cat P$) and whose topology is generated by the following basis of closed subsets:
\[
\supp(A):= \{ \cat P \in \Spc(\cat P)\mid A\not\in \cat P\} \quad (A\in \cat K).
\]
%
The set $\supp(A)$ is called the \emph{support of~$A$}; it has quasi-compact open complement.
The assignment $A\mapsto \supp(A)$ satisfies a list of basic compatibilities with the operations in $\cat K$, \ie defines a \emph{support data} on~$\cat K$, and the pair $(\Spc(\cat K), \supp)$ is universal (the finest) among such support data.

A subset $S\subseteq \Spc(\cat K)$ is \emph{Thomason} if it is a union of closed subsets with quasi-compact open complements. 
In case $\Spc(\cat K)$ is \emph{noetherian} (\ie every open subset is quasi-compact), Thomason subsets are the same as those subsets $S$ which are \emph{specialization closed} (\ie if $\cat Q \in \overline{\{\cat P\}}$ and $\cat P\in S$ then $\cat Q\in S $.)

The following  classification result is a major early milestone of tt-geometry:

\begin{Thm}[\cite{Balmer05a}]
\label{Thm:classif-tensorid}
Suppose that $\cat K$ is rigid (just for simplicity). Then there is a canonical bijection 
\[ 
\left\{
\begin{array}{c} 
\textrm{thick tensor ideals of } \cat K
\end{array}
\right\}
\overset{\sim}{\longrightarrow}
\left\{
\begin{array}{c} 
\textrm{Thomason subsets of } \Spc(\cat K)
\end{array}
\right\}
\]
which maps a thick ideal $\cat J$ to its support $\supp(\cat J):= \bigcup_{A\in \cat J} \supp(A)$, and 
whose inverse sends a Thomason subset $S$ to the thick tensor ideal
\[
\cat K_{S}:= \{A \in \cat K \mid \supp(A)\subseteq S\}
\]
of all objects supported on~$S$.
\end{Thm}

\subsection{Functoriality of the spectrum}
As explained in \cite{Balmer05a}, 
every tt-functor $F\colon \cat K\to \cat L$ of essentially small tt-categories induces a continuous map on spectra $\Spc(F)\colon \Spc(\cat L)\to \Spc(\cat K)$ by $\Spc(F)(\cat P)=F^{-1}(\cat P)$, and the procedure is functorial: $\Spc(F_1 \circ F_2)= \Spc(F_2)\circ \Spc(F_1)$ and $\Spc(\Id_\cat K)=\Id_{\Spc(\cat K)}$.

If $F\colon \cat K \to \cat L$ is fully faithful and \emph{cofinal} (\ie every $A\in \cat L$ is a retract of $F(B)$ for some $B\in \cat K$), the induced map $\Spc(F)$ is a homeomorphism; this applies in particular to the \emph{idempotent completion} $\cat K \hookrightarrow \cat K^\natural$. 
Thus we may assume that our tt-categories are idempotent complete without changing their spectrum.

If $q\colon \cat K\to \cat K/\cat J$ is a Verdier quotient by a tensor ideal, the induced map $\Spc(q)$ restricts to a homeomorphism onto its image: $\Spc(\cat K/\cat L) \overset{\sim}{\to}\{\cat P\mid \cat J\subseteq \cat P\}\subseteq \cat K$.

A tt-category $\cat K$ is \emph{local} if its spectrum has a unique closed point (one admitting no specializations: $\overline{\{\cat P\}}=\{\cat P\}$).
If $\cat K$ is rigid, this point must be the zero tensor ideal~$\{0\}$.
The \emph{local category} of $\cat K$ at a point~$\cat P$ is defined to be $\cat K_\cat P:= (\cat K/\cat P)^\natural$; it is an idempotent complete local tt-category, equpipped with a  canonical tt-functor $q_\cat P\colon \cat K\to \cat K_\cat P$, the quotient by $\cat P$ followed by idempotent completion.

Note that $q_\cat P(A)\not\cong 0 \Leftrightarrow \cat P \in \supp(A)$, which justifies the name `support': a point $\cat P$ lies in the support of $A$ precisely if $A$ survives in the local category at~$\cat P$.

\subsection{The comparison map}
\label{subsec:rho}
Balmer \cite{Balmer10b} introduces for every essentially small tt-category $\cat K$ a continuous map
\[
\rho= \rho_\cat K \colon \Spc(\cat K) \to \Spec(\End_\cat K(\unit)), 
\quad  \rho(\cat P) = \{f\colon \unit \to \unit \mid \cone(f)\not\in \cat P\}
\]
between the triangular spectrum and the Zariski spectrum of the (automatically commutative) endomorphism ring of the tensor unit~$\unit\in \cat K$.
This map is natural, in the sense that every tt-functor $F\colon \cat K\to \cat L$ gives rise to a commutative square
\[
\xymatrix{
\Spc (\cat L)
 \ar[r]^-{\Spc(F)}
  \ar[d]_{\rho_\cat L} &
 \Spc(\cat K)
  \ar[d]^{\rho_\cat K} \\
\Spec(\End_\cat L(\unit))
 \ar[r]^-{\Spec(F)} &
\Spec(\End_\cat K(\unit))
}
\]
of topological spaces, where the bottom map is induced by the ring morphism $F\colon \End_\cat K(\unit)\to \End_\cat L(\unit)$ obtained simply by restricting the functor~$F$.

As for any triangulated category, the suspension functor $\Sigma$ turns $\cat K$ into a $\mathbb Z$-graded category with graded Hom-spaces $\Hom^*_\cat K(A, B):=\Hom_\cat K(A, \Sigma^*B)$.
The tensor product turns the latter into graded (left and right) modules over $\End^*_\cat K(\unit)= \Hom_\cat K(\unit,\Sigma^*\unit)$, which is a graded commutative ring.
(Strictly speaking, in general here `graded commutative' means $f g = \epsilon^{\deg(f)\deg(g)} g f$ on homogeneous elements, where $\epsilon\in \End(\unit)$ denotes the invertible element whose action on $\Sigma(\unit)\otimes \Sigma(\unit)\overset{\sim}{\to}\Sigma(\unit)\otimes \Sigma(\unit)$ is the symmetry coming with the tensor structure. 
Note however that we have $\epsilon= -1$ whenever the tt-category has a model or enhancement of any kind, as well as under the hypothesis in \Cref{Def:big-tt-cat}(b).)
There is a graded version 
\[
\rho^*_\cat K\colon \Spc(\cat K) \to \Spec(\End^*_\cat K(\unit))
\]
of the comparison map, with similar formal properties, where on the right-hand side we now have the Zariski spectrum of homogeneous prime ideals.

Both comparison maps are surjective as soon as $\End^*_\cat K(\unit)$ is noetherian (or just coherent) in the graded sense (\cite[Thm.\,7.3, Cor.\,7.4]{Balmer10b}).
Injectivity holds less often and when it does it is typically much harder to prove, as we shall experience.

\subsection{tt-rings}
\label{subsection:tt-rings}
Given a tt-category~$\cat K$, one can consider monoids (a.k.a.\ algebras or rings) and modules over them, just as in any tensor category. 
Recall that a \emph{monoid} is an object $A$ equipped with a `multiplication' $m\colon A\otimes A\to A$ and `unit' $u\colon \unit \to A$ morphisms satisfying the usual associativity and unitality conditions via diagrams commuting in~$\cat K$, and similarly for modules (by which, unless otherwise stated, we always mean left modules).
Modules over $A$ and their morphisms form a category $\Mod_\cat K(A)$, the \emph{Eilenbert--Moore category} of~$A$. 
In the context of tt-geometry, this category is particularly useful for monoids of the following kind (because then, among other reasons, it essentially inherits a structure of tt-category):

\begin{Def} [\cite{Balmer14}]
\label{Def:tt-ring}
A \emph{tt-ring} in $\cat K$ is a monoid $(A,m,u)$ which is \emph{commutative}, \ie it satisfying $m\circ \sigma = m$ where $\sigma \colon A\otimes A\overset{\sim}{\to} A \otimes A$ is the symmetry of the tensor structure, and \emph{separable}, in the sense\footnote{This sense of `separable' has strictly nothing to do with the point-set topological meaning of the word used in the context of C*-algebras (in the sense of admitting a countable dense subset).} that its multiplication $m$ admits a section $s\colon A\to A\otimes A$ which is a morphism of both left and right $A$-modules.
\end{Def}

Balmer \cite{Balmer14} introduces a notion of \emph{degree} $\deg(A)\in \mathbb N\cup \{+\infty\}$ for tt-rings~$A$ in an idempotent-complete tt-category~$\cat K$, which we use as a black box and will not need to define.
There exist tt-rings of infinite degree (\cite{Gomez23pp}), and having finite degree is a necessary hypothesis for applying certain results. 
Later we will need the next fact:

\begin{Lem}
\label{Lem:degree}
Let $A_1, \ldots,A_n $ be tt-rings in an (idempotent complete) tt-category $\cat K$, all of which have finite degree. 
Then $\deg (A_1 \times \ldots \times A_n)\leq \deg(A_1)+\ldots + \deg(A_n)$.
\end{Lem}

\begin{proof}
We prove the case $n=2$ and the general one follows by an easy induction.
Write $A:=A_1 \times A_2$, and consider the tt-functor $q_\cat K\colon \cat K\to \cat K/\cat P\subseteq (\cat K/\cat P)^\natural$ to the local category  at a point $\cat P\in \Spc(\cat K) $ (\Cref{subsec:rho}).
By \cite[Thm.\,3.7(a)]{Balmer14} the degree cannot increase along tt-functors, so we have $\deg(q_\cat P (A_i))\leq \deg(A_i)$.
In particular the $q_\cat P(A_i)$ all have finite degree, and since the tt-category $(\cat K/\cat P)^\natural$ is local, by \cite[Cor.\,3.12]{Balmer14} we have an equality
\[
\deg (q_\cat P(A)) 
=  \deg (q_\cat P(A_1) \times q_\cat P (A_2))
= \deg(q_\cat P A_1)  + \deg(q_\cat P A_2) < +\infty.
\]
In particular $q_\cat P(A)$ has finite degree for all $\cat P$, hence by \cite[Thm.\,3.8]{Balmer14} we have
\[
\deg(A) = \max_{\cat P \in \Spc(\cat K)} \deg(q_\cat P(A))
\]
and therefore by combining the two
\begin{align*}
\deg(A) &= \max_{\cat P \in \Spc(\cat K)} \deg(q_\cat P(A_1)) + \deg(q_\cat P(A_2)) \\
& \leq \max_{\cat P \in \Spc(\cat K)} \deg(A_1) + \deg(A_2) = \deg(A_1) + \deg(A_2) < +\infty
\end{align*}
as claimed.
\end{proof}

\subsection{Big tt-categories}
\label{subsec:bit-ttcats}

As pointed out in the Introduction, many essentially small tt-categories $\cat K$ occur in nature as subcategories of compact and/or rigid objects in `big' tt-categories.
In the following, suppose that $\cat T$ is a tt-category admitting arbitrary small coproducts (all those indexed by some small set, not a proper class).

A triangulated subcategory $\cat L\subseteq \cat T$ is \emph{localizing} if it is closed under the formation of small coproducts. 
Localizing subcategories are automatically thick (in fact, more generally, any triangulated category admitting arbitrary countable coproducts is idempotent complete by \cite[Prop.\,1.6.8]{Neeman01}).
A \emph{localizing tensor ideal} is a localizing subcategory which is also a tensor ideal.

\begin{Not} 
\label{Not:loc}
We write $\Loc{\cat E}$ and $\tensLoc{\cat E}$ for the localising subcategory, respectively the localising tensor ideal, of $\cat T$ generated by a family of objects $\cat E\subseteq \cat T$.
\end{Not}

An object $C\in \cat T$ is \emph{compact} if the functor $\cat T(C,-)\colon \cat T\to \Ab$ to abelian groups preserves coproducts.
We write $\cat T^c$ for the (full triangulated) subcategory of compact objects of~$\cat T$. 
The triangulated category $\cat T$ is said to be \emph{compactly generated} if there exists a small set $\cat G\subseteq \cat T^c$ of compacts such that $\Loc{\cat G}=\cat T$; 
by results of Neeman \cite{Neeman96}, the condition that $\Loc{\cat G}=\cat T$ is equivalent to the (a~priori weaker) requirement that for all $A\in \cat T$, if $\cat T^*(C,A)\cong 0$ for all $C\in \cat G$ then $A\cong 0$. 
If $\cat T$ is compactly generated then $\cat T^c$ is essentially small, and of course $\Loc{\cat T^c}=\cat T$.

\begin{Def}
\label{Def:big-tt-cat}
A tt-category $\cat T$ with small coproducts is \emph{rigidly-compactly generated}, or is a \emph{big tt-category} for short, if it satisfies the following:
\begin{enumerate}
\item $\cat T$ is compactly generated.
\item The tensor structure and triangulation satisfy the compatibility conditions in \cite[App.\,A]{HoveyPalmieriStrickland97}; in particular, $\otimes$ preserves coproducts in both variables.
\item The rigid and compact objects of $\cat T$ coincide (in particular $\unit \in \cat T^c$).
\end{enumerate}
Since rigid objects always form a tensor subcategory and compact objects a triangulated subcategory, it follows that $\cat T^c$ is a tt-subcategory of~$\cat T$.
\end{Def}

Big tt-categories in the above sense have been intensely studied in recent years, in particular in relation to stratification as explained in \Cref{sec:strat-prelim}.

\subsection{Countably big tt-categories}
A version of big tt-categories which only admit $\alpha$-small coproducts for a given regular cardinal $\alpha$ was introduced in \cite{DellAmbrogio10}. 
 We only mention here the countable version which is relevant to KK-theory.
 
 \begin{Not}
 \label{Not:cbly-loc} 
 Suppose $\cat T$ admits arbitrary countable   (\ie $\aleph_1$-small) coproducts. 
 A \emph{localizing${}_{\aleph_1}$} (or \emph{$\aleph_1$-localizing}) subcategory of $\cat T$ is a triangulated subcategory which is closed under countable coproducts (and therefore is automatically thick).
 The decorated analogues $\Locc{\cat E}$ and $\tensLocc{\cat E}$ of \Cref{Not:loc} will denote the smallest localizing${}_{\aleph_1}$, resp.\ localizing${}_{\aleph_1}$ tensor ideal, subcategory generated by a family of objects~$\cat E\subseteq \cat T$.  
 \end{Not}

\begin{Def}
\label{Def:cbly-cptly-gen}
 Suppose that $\cat T$ is a triangulated category admitting arbitrary countable coproducts. 
An object $A\in \cat T$ is \emph{compact${}_{\aleph_1}$} if the functor $\cat T(A,-)\colon \cat T\to \Ab$ takes values in countable abelian groups and commutes with countable coproducts.
We say that $\cat T$ is \emph{compactly${}_{\aleph_1}$ generated} if there exists a countable set $\cat G\subseteq \cat T$ of compact${}_{\aleph_1}$ objects such that 
$\Locc{\cat G}= \cat T$ (\Cref{Not:cbly-loc}).
\end{Def}

As first observed in \cite{MeyerNest06}, compactly${}_{\aleph_1}$ generated categories still satisfy the basic Brown representability theorem and its many consequences which make Neeman's usual compactly generated categories so useful (see \cite[\S2.1]{DellAmbrogio10} for details).

\begin{Def}
\label{Def:cbly-big-ttcat}
Suppose now that $\cat T$ is a tt-category admitting all countable coproducts. 
We say $\cat T$ is \emph{rigidly-compactly${}_{\aleph_1}$ generated}, or is a \emph{countably big tt-category}, if it satisfies the following:
\begin{enumerate}
\item $\cat T$ is compactly${}_{\aleph_1}$ generated (\Cref{Def:cbly-cptly-gen}).
\item The tensor functor $\otimes$ preserves coproducts in both variables.
\item The rigid objects and compact${}_{\aleph_1}$ objects of $\cat T$ coincide.
\end{enumerate}
As in the case of `genuine' big tt-categories, if rigid and compact${}_{\aleph_1}$ objects agree in $\cat T$ they form a full tt-subcategory $\cat T$ which is essentially small, in fact which admits a countable skeleton (\ie it is equivalent to a tt-category with only countably many objects and morphisms). 
\end{Def}

\begin{Not}
We will still write $\cat T^c\subseteq \cat T$ for the tt-subcategory of compact${}_{\aleph_1}$-rigid objects.
In practice, we will often drop the decorations ${(\ldots)}_{\aleph_1}$ in `compact${}_{\aleph_1}$', `$\Locc{-}$' etc.\ if there is no ambiguity about the ambient category~$\cat T$. 
\end{Not}

\begin{War} \label{War:cardinals}
In the Appendices 
we will use Neeman's notion of \emph{well-generated} triangulated category (\cite{Neeman01}), a generalization of compactly generated categories which still have all small coproducts but admit more general sets of generators.
We will use this theory as a black box, avoiding further recollections. 
Beware that well-generated categories have their own established cardinal-related terminology, and the above use of sub-index decorations $(\ldots)_{\aleph_1}$ is meant to clearly avoid conflicting with it.
Unfortunately, in \Cref{sec:app-big} we will also make heavy use of the cardinal-relative language from Lurie's theory of $\infty$-categories, which itself is well-established and similar but only partially compatible with that of well-generated categories.
\end{War}

\section{Tensor triangulated categories of equivariant KK-theory}
\label{sec:KK-prelim}%

Let $G$ be a second countable locally compact group.
It was first proved, and exploited, by Meyer--Nest \cite{MeyerNest06} that Kasparov's~\cite{Kasparov88} $G$-equivariant bivariant K-theory (or `KK-theory') forms a tensor triangulated category~$\KK^G$.
In this section we recall a few basic features of equivariant KK-theory seen as a tt-category, referring to \cite{MeyerNest06} \cite{Meyer08} \cite{Blackadar98} for details. 
We also record a few more recent results, including possibly new observations (\Cref{Lem:End-finite}, \Cref{Prop:finite-degree}).

\subsection{Basic tt-structure}
The objects $\KK^G$ are separable complex $G$-C*-algebras, its (graded) Hom groups are Kasparov's KK-theory groups, and its composition and tensor product encode the so-called Kasparov product. There is a canonical functor $\Cstarsep{G}\to \KK^G$ from the category of separable $G$-C*-algebras and $G$-equivariant *-homomorphisms which is the identity on objects and satisfies a universal property (the latter will be exploited in \Cref{sec:app-big}), from which it follows that the category $\KK^G$ is essentially small.
Countable ($c_0$-)direct sums of separable $G$-C*-algebras are again separable $G$-C*-algebras (with coordinatewise $G$-action) and provide countable coproducts in the category~$\KK^G$.

The suspension functor of $\KK^G$ is given by $\Sigma\colon A\mapsto \Cont_0(\mathbb R, A)$ and its distinguished triangles arise from semi-split extensions or, equivalently, mapping cone sequences.
Bott periodicity becomes an equivalence $\Sigma \circ \Sigma \cong \Id_{\KK^G}$.
The tensor structure $-\otimes-$ extends the algebra-level minimal tensor product of C*-algebras equipped with the diagonal $G$-action, and its unit object is $\unit = \mathbb C$, the algebra of complex numbers with trivial $G$-action. 

If $G$ is compact, the graded commutative endomorphism ring of the unit is
\[
\End^*_{\KK^G}(\unit)\cong \Rep(G)[\beta^{\pm1}]
\]
where $\beta$  (of degree two) is the invertible Bott element and $\Rep(G)\cong \End(\unit)$ (in degree zero) is the complex representation ring of~$G$. 
(If $G$ is non-compact, the nature of this ring generally remains mysterious and is connected to deep questions such as the Baum--Connes conjecture.)

\subsection{Functoriality in the group}
\label{subsec:KK-fun-G}
Equivariant KK-theory is equipped with a rich functoriality in the group variable, which can even be neatly extended to groupoids.
We will only explicitly use a small part of this structure, namely:
\begin{itemize}
\item The restriction functor $\Res^G_H\colon \KK^G\to \KK^H$, for $H\leq G$ a closed subgroup.
\item The induction functor $\Ind^G_H\colon \KK^H\to \KK^G$, for $H\leq G$ a closed subgroup.
\item The conjugation equivalences $\Con_g=\Con_{g,H}\colon \KK^H\overset{\sim}{\to} \KK^{{}^gH}$, for $H\leq G$ a closed subgroup and $g\in G$ an element, where ${}^gH := gHg^{-1}$ denotes the conjugate subgroup. (On algebras, $\Con_{g}$ sends an $H$-C*-algebra $A$ to the same C*-algebra equipped with the ${}^gH$-action $ghg^{-1}\cdot a:=ga$.)
\end{itemize}
All these functors can be constructed at the level of C*-algebras and then extended to KK-theory using its universal property. 
Both restriction and conjugation functors are special cases of restricting an action along a continuous group morphism with closed image.
Once extended, these functors have nice properties, such as:
\begin{itemize}
\item All three kinds of functors are exact and preserve countable coproducts.
\item Restriction and conjugation functors are symmetric monoidal, hence are tt-functors of tt-categories (\Cref{subsec:ttcats}).
\item When $G/H$ is compact, $\Ind^G_H$ is a \emph{right} adjoint of $\Res^G_H$:
\[ 
\KK^G(A, \Ind^G_H(B))\cong \KK^H(B, \Res^G_H(A)). 
\]
\item When $G/H$ is discrete, $\Ind^G_H$ is a \emph{left} adjoint of $\Res^G_H$:
\[ 
\KK^G(\Ind^G_H(B), A)\cong \KK^H(B, \Res^G_H(A)). 
\]
\item There is a `projection formula' natural isomorphism
\begin{equation} \label{eq:KK-proj-formula}
\Ind^G_H( \Res^G_H(A) \otimes B) \cong A \otimes \Ind^G_H(B).
\end{equation}
\end{itemize}
Moreover, these functors are related by several functoriality and commutation natural isomorphisms, which can be easily checked at the algebra level, such as restriction and induction in stages 
\begin{align*}
& \Res^G_H \circ \Res^H_L   \cong \Res^G_L , \quad \Ind_H^G \circ \Ind_L^H \cong \Ind_L^G \quad (L\leq H\leq G),
\end{align*}
as well as 
\[
\Con_{g,L} \circ \Res^H_L \cong \Res^{{}^gH}_{{}^gL}\circ \Con_{g,H}, \quad \Con_{g,G} \cong \Id_{\KK^G} \quad (L\leq H\leq G),
\]
(used in the proof of \Cref{Thm:Spc-prime-order}), among many others.
\begin{Rem} \label{Rem:G2F}
It is possible to show that the above structure turns the assignment $G\mapsto \KK^G$ (when restricted to \emph{finite} groups~$G$) into a Green 2-functor in the sense of~\cite{DellAmbrogio22}.
We will not make use of this fact, although it would provide alternative proofs for some results obtained later on.
\end{Rem}

\subsection{$G$-cell algebras and the equivariant bootstrap category}
As mentioned in the Introduction, $\KK^G$ does not seem to admit any good generating set. 
To remedy this, one can choose a suitable set of objects and replace $\KK^G$ with the smaller, but more tractable, localizing subcategory generated by this set (\Cref{Not:loc}).

\begin{Def}[{\cite[Def.\,2.2]{DellAmbrogio14}}]
\label{def.CellCat}
	Let $G$ be a locally compact group. 
	A \emph{(standard) orbit algebra} is a commutative $G$-C$^*$-algebra of the form $C_0(G/H)$, for some closed subgroup $H\leq G$. We denote by $\Cell{G}$ the localising subcategory of $\KK^G$ generated by orbit $G$-cell algebras:
		$$\Cell{G}:=\Loc{\Cont_0(G/H)\mid H \textrm{ a closed subgroup of } G}\quad \subset \quad \KK^G.$$
		The objects in $\Cell{G}$ are called \emph{$G$-cell algebras}. 
\end{Def}

As a localising subcategory of $\KK^G$, $\Cell{G}$ is still an essentially small category admitting all countable coproducts. Observe that $\Cell{G}$ contains the tensor unit $\mathbb C = \Cont_0(G/G) $. 
One can show, at least when $G$ is discrete or a compact Lie group, that $\Cell{G}$ is closed under the tensor product and therefore defines a tt-subcategory of $\KK^G$; and moreover the restriction, induction and conjugation functors between equivariant Kasparov categories recalled in \Cref{subsec:KK-fun-G} all restrict to exact functors and/or tt-functors between the corresponding categories of cell algebras (see \cite[Proposition 2.7]{DellAmbrogio14}).

The case of finite groups is particularly nice, because of the following observation:

\begin{Prop} [{\cite[Prop.\,2.9]{DellAmbrogio14}}]
\label{Prop:Cell(G)-cble-bigtt}
Let $G$ be any finite group.
Then $\Cell{G}$ is a rigidly-compactly${}_{\aleph_1}$ generated tt-category (\Cref{Def:cbly-big-ttcat}), with orbit $G$-cell algebras $\Cont_0(G/H)= \Cont(G/H)$ $(H\leq G)$ providing a finite set of compact${}_{\aleph_1}$ and rigid objects (each $\Cont(G/H)$ is its own tensor dual).
In particular $\Cell{G}^c$ is a rigid essentially small tt-category.
\end{Prop}

The case $\Cell{1}= \Loc{\mathbb C}\subset \KK$ of the trivial group is the classical \emph{bootstrap category} of Rosenberg and Schochet \cite{RosenbergSchochet86}, which has many useful characterizations.
Indeed, there are multiple ways one may wish to generalize the bootstrap category to the equivariant setting, for instance by choosing different generating sets of objects than orbit algebras.
For finite groups, it turns out that several of these possible generalizations agree:

\begin{Thm} \label{Thm:cyclic-generation}
When~$G$ is finite, the following full subcategories of $\KK^G$ coincide:
\begin{enumerate}
\item The \emph{$G$-equivariant bootstrap category} defined in~\cite{DEM14}, that is those separable $G$-C*-algebras $\KK^G$-equivalent to a $G$-action on a Type~I C*-algebra.
\item The localizing subcategory generated by \emph{elementary} $G$-C*-algebras, \ie those of the form $\Ind^G_H(B)$ where $H\leq G$ is a subgroup and where $B\in \KK^H$ has for its underlying C*-algebra one of the matrix algebras~$M_n(\mathbb C)$ ($n\in \mathbb N$). 
\item $\Cell{G} = \Loc{\Cont(G/H)\mid H\leq G}$, the above category of $G$-cell algebras.
\item $\Loc{\Cont(G/H)\mid H\leq G \; \mathrm{ cyclic } }$, where we only take cyclic subgroups of~$G$.
\end{enumerate}
\end{Thm}

\begin{proof}
This is proved in \cite{MeyerNadareishvili24pp}.
The inclusions (a)$\supseteq$(b)$\supseteq$(c)$\supseteq$(d) are easy, and the equivalence of~$(a)$ and~$(d)$ is \cite[Cor.\,3.3]{MeyerNadareishvili24pp}.
The proof is based on a strong result of Arano--Kubota stating that if $A\in \KK^G$ is such that $\Res^G_H(A)\cong 0$ in $\KK^H$ for every cyclic subgroup $H\leq G$, then $A\cong 0$ in $\KK^G$ (see \cite[Cor.\,3.13]{AranoKubota18}; their result applies more generally to compact Lie groups $G$ and $\sigma$-C*-algebras $A,B$).
\end{proof}

In this article we will only use the equivalence of (c) and~(d).
However, in view of the equivalence with~(a) we also refer to $\Cell{G}$ as the equivariant bootstrap category of the finite group~$G$.
We record a nice feature of its compact objects:

\begin{Lem} \label{Lem:End-finite}
For every finite group $G$,
the graded endomorphism ring  of the tensor unit $\End^*(\unit)\cong \Rep(G)[\beta^{\pm1}]$ is noetherian, and the graded module $\End^*(A)=\KK^G_*(A,A)$ over it is finitely generated for every compact${}_{\aleph_1}$ object $A\in \Cell{G}^c$.
\end{Lem}

\begin{proof}
When $G$ is finite (or even compact Lie), $\Rep(G)$ is a finitely generated ring and therefore noetherian by \cite[Cor.\,3.3]{Segal68a}. 
If $H\leq G$ is any (closed) subgroup, $\Rep(H)$ is finitely generated when seen as an $\Rep(G)$-module via the restriction map $\res^G_H\colon \Rep(G)\to \Rep(H)$, by \cite[Prop.\,3.2]{Segal68a} (attributed to Atiyah).
The former is easily seen to be equivalent to the noetherianity of the graded commutative ring $\Rep(G)[\beta^{\pm1}]$, and the latter implies the remaining claim by the following argument.

 For each $H\leq G$,  we deduce by the adjunction $\Res^G_H \dashv \Ind^G_H$ that  the $\Rep(G)$-module
\[ \KK^G_*(\unit, \Cont(G/H)) = \KK^G_*(\unit, \Ind^G_H(\unit))\cong \KK^H_*(\unit,\unit)=\Rep(H)[\beta^{\pm1}]\] 
is finitely generated. 
Since $\Cont(G/H)$ is self-dual, we further deduce that
\[
\End^*(\Cont(G/H)) \cong \KK^G_*(\unit, \Cont(G/H)\otimes \Cont(G/H))\cong \bigoplus_{[x]\in H\backslash G/{}^x\!H} \KK^G_*(\unit , \Cont(G/H \cap {}^xH))
\]
is also finitely generated. (Note that all isomorphisms used here are $\Rep(G)$-linear, because so are the restriction functors to subgroups and hence the adjunction isomorphisms, as one easily checks.)
Since $\Cell{G}^c=\Thick{\Cont(G/H)\mid H\leq G}$, we conclude the same about $\End^*(A)$ for an arbitrary $A\in \Cell{G}^c$ by a routine induction argument on the triangular length of $A$ with respect to the generators.
\end{proof}

\subsection{The induced tt-rings $\Cont(G/H)$}
\label{subsec:C(G/H)}

Recall that $\Cont_0(G/H)\cong \Ind^G_H(\unit)$ for any $G$ and for any closed subgroup $H\leq G$.
If $G/H$ is compact, $\Ind^G_H$ is right adjoint to the tensor functor $\Res^G_H$ and therefore inherits a lax monoidal structure, which applied to the tensor unit $\unit=\mathbb C\in \KK^H$ yields a commutative monoid structure on the object $\Cont(G/H)=\Cont_0(G/H)$ of~$\KK^G$ (see \eg\ \cite[\S2]{BDS15} for details).

If moreover $G$ is finite, the multiplication map of $\Cont(G/H)$ admits a $\Cont(G/H)$-bilinear section, as observed in~\cite[Rem.\,2.10]{BDS15}, so that $\Cont(G/H)$ becomes a tt-ring in $\KK^G$ in the sense of \Cref{Def:tt-ring}. 
(Alternatively, this is also a consequence of the Green 2-functor structure of \Cref{Rem:G2F} together with \cite[Thm.\,9.2]{DellAmbrogio22}.)

\begin{Rem}
\label{Rem:decomp-ttring}
One can easily check (\eg\ from \cite[Constr.\,4.1]{BDS15}) that this multiplication is simply the usual multiplication of the commutative C*-algebra $\Cont(G/H)$.
It follows that decompositions of the form
\[
\Res^G_L \Cont(G/H) \cong \prod_{[x] \in L \backslash G/ H} \Cont(L /L \cap \!{}^xH )  \quad (L,H\leq G)
\]
(and similar ones arising from decompositions of left $G$-sets) are isomorphisms of monoids in $\KK^G$ and therefore of tt-rings.
\end{Rem}

\begin{Prop}
\label{Prop:finite-degree}
For every finite group $G$ and subgroup $H\leq G$, the above tt-ring $\Cont(G/H)$ has finite degree in the sense of~\cite{Balmer14} in the tt-category $\Cell{G}^c$.
\end{Prop}

\begin{proof}
We are going to appeal (twice) to \cite[Thm.\,3.7(b)]{Balmer16}, which contains the following criterion: Let $A\in \cat K$ be a tt-ring of the tt-category $\cat K$, and suppose $F\colon \cat K\to \cat L$ is a tt-functor such that its restriction to the thick tensor-ideal $\cat K_{\supp(A)}\subseteq \cat K$ is conservative (\ie it detects the vanishing of objecs); then the degree of $A$ equals that of~$F(A)$. (As with the rest of the theory of degree, this criterion assumes both $\cat K$ and $\cat L$ to be idempotent complete.)

First we apply this to the tt-functor $F:=(\Res^G_S)_S \colon \Cell{G}^c\to \prod_S \Cell{S}^c$, where $S$ runs through all cyclic subgroups of~$G$.
We see that $F$ is conservative on the whole category $\Cell{G}^c$ by \Cref{Thm:cyclic-generation}(d) (in fact this follows directly from the Arano--Kubota result \cite[Cor.\,3.13]{AranoKubota18} used in its proof), and therefore we reduce to proving that $F(\Cont(G/H)) = (\Res^G_S\Cont(G/H))_S$ has finite degree in the product tt-category.

Since the degree of a product  of finitely many finite-degree  tt-rings is at most the sum of their degrees (by \Cref{Lem:degree}), and because of the isomorphism
\[
\Res^G_S \Cont(G/H) \cong \prod_{[x] \in S \backslash G/ H} \Cont(S / S \cap \!{}^xH )
\]
 of tt-rings (\Cref{Rem:decomp-ttring}), we further reduce to showing that $\Cont(S/T)$ has finite degree in $\Cell{S}^c$ for every cyclic group $S$ and subgroup $T\leq S$.

To prove the latter we apply Balmer's criterion again, this time taking the tt-functor $F$ to be $\Res^S_T\colon \Cell{S}^c\to \Cell{T}^c$.
To see that this $F$ verifies the required hypothesis, observe that we have an inclusion
\[
\Cell{A}^c_{\supp (\Cont(S/T))} 
= \tensThick{\Cont(S/T)}
\subseteq \Ind^S_T \big(\Cell{T}^c \big) =: \cat J
\]
of the tensor-ideal of objects supported on $\supp(\Cont(S/T))$ into the essential image of the induction functor.
Indeed, the first equality displayed above is by the classification of thick tensor ideals (\Cref{Thm:classif-tensorid}); the inclusion on the right follows from the fact that $\Ind^S_T (\Cell{T}^c)$ is a thick tensor ideal
(an easy consequence of the projection formula \eqref{eq:KK-proj-formula} for $\Res^S_T$ and~$\Ind^S_T$) which contains $\Cont(S/T)= \Ind^S_T(\unit)$.

Hence it suffices to show that $F= \Res^S_T$ is conservative on~$\cat J$. 
Recall that the two adjunctions $\Ind^S_T \dashv \Res^S_T \dashv \Ind^S_T$ can be chosen so that the composite
\[
A \to \Res^S_T \Ind^S_T (A) \to A
\]
of the unit of $\Ind^S_T \dashv \Res^S_T$ followed by the counit of $\Res^S_T \dashv \Ind^S_T$ is the identity map of~$A$, for all $A\in \Cell{T}^c$ (this is well-known; see \Cref{Thm:bigCell}(2)(b)  for a proof if necessary).
In particular, it follows that $\Res^S_T \Ind^S_T (A)\cong 0$ implies $A\cong 0$ and therefore $\Ind^S_T(A)\cong 0$, showing that indeed $\Res^S_T$ is conservative on~$\cat J$.
Thus by Balmer's criterion the degree of $\Cont(S/T)$ equals that of the product tt-ring
\[
\Res^S_T(\Cont(S/T))\cong \prod_{[y] \in T\backslash S / T} \underbrace{\Cont(T/T\cap\!{}^yT)}_{\Cont(T/T)} \cong \unit^{\times |S/T|}
\]
which is equal to the finite index $|S/T|$ by \cite[Thm.\,3.9(a)]{Balmer14}.
\end{proof}

\section{Stratification}
\label{sec:strat-prelim}%

Our main reference for the stratification theory of tt-categories is \cite{BHS23}, whose terminology and notations
we shall adopt.
Throughout $\cat T$ will denote a big tt-category, \ie one which is rigidly-compactly generated (\Cref{Def:big-tt-cat}).

\begin{Def}
\label{Def:general-supp-and-strat}
A \emph{theory of supports}, or \emph{support theory}, on $\cat T$ is a pair $(X,\sigma)$ consisting of a topological space $X$ and a function $\sigma$ assigning a subset $\sigma(A)\subseteq X$ to every object of $\cat T$; this assignment must satisfy five basic compatibility axioms with the suspension, direct sums, triangles, coproducts and tensor products in~$\cat T$ (see \cite[Def.\,7.1]{BHS23}).
We can extend the support function $\sigma$ to arbitrary classes of objects $\cat{S} \subseteq \cat T$ by setting 
$\sigma (\cat{S}) := \bigcup_{A\in \cat{S}} \sigma(A)$. 
The support theory $(X,\sigma)$ is said to \emph{stratify~$\cat T$} if the assignment $\cat{S}\mapsto \sigma(\cat{S})$ restricts to a bijection as follows:
\[
\left\{
\begin{array}{c} 
\textrm{localizing tensor ideals of } \cat T
\end{array}
\right\}
\overset{\sim}{\longrightarrow}
\left\{
\begin{array}{c} 
\textrm{subsets of }X
\end{array}
\right\}.
\]
%
\end{Def}

\begin{Rec} \label{Rec:noeth}
A topological space is \emph{noetherian} if every open subset is quasi-compact. 
For example the Zariski spectrum of any noetherian commutative ring is a noetherian space. 
In \cite{BHS23}, the notion of a \emph{weakly noetherian} space is introduced.
For our present purposes, we only need to know that all noetherian spaces are weakly noetherian.
Weakly noetherian spaces serve the following:
\end{Rec}

\begin{Exa}[Balmer--Favi supports]
\label{ex.BalmerFaviSupp}
The main example of a support theory is due to Balmer--Favi~\cite{BalmerFavi11}, and we simply write it as $\sigma(A)=\supp(A)$ for all~$A$. 
The target space $X$ of this support theory is the Balmer spectrum $\Spc (\cat T^c)$ of the tt-subcategory of compact objects, and the support of an object $A$ is given by
\[
\supp(A) := \left\{  \cat P\in \Spc (\cat T^c) \mid g(\cat P) \otimes A \neq 0  \right\},
\]
where $g(\cat P)$ is a certain $\otimes$-idempotent object of~$\cat T$. 
On compact objects, the Balmer--Favi support agrees with Balmer's universal support data recalled in \Cref{subsec:Spc}, so there is no notational ambiguity.
We refer to \cite[\S2]{BHS23} for details.
Note that the objects $g(\cat P)$ are all well-defined, and therefore the Balmer--Favi support theory is, precisely when $\Spc(\cat K)$ is weakly noetherian (\Cref{Rec:noeth}).
\end{Exa}

\begin{Ter} \label{Ter:BHS-stratified}
Unless otherwise stated,
when we say that $\cat T$ is \emph{stratified in the sense of \cite{BHS23}}, or simply \emph{stratified}, we mean that $\Spc(\cat K)$ is weakly noetherian (\eg\ noetherian) and that $\cat T$ is stratified by its Balmer--Favi support theory. 
\end{Ter}

The following general fact can be useful for computing Balmer--Favi supports:

\begin{Lem}
\label{Lem:general-detection}
Suppose the big tt-category $\cat T$ has noetherian spectrum and is stratified. Let $\cat P\in \Spc (\cat T^c)$ be a point of the spectrum and let $A\in\cat T$ be any object. The following are equivalent:
\begin{enumerate}
\item The Balmer--Favi support of $A$ is $\supp (A) = \{\cat P\}$.
\item $\tensLoc{A} = \tensLoc{g(\cat P)}$.
\item For all $B \in \cat T$, we have $ \cat P\in \supp (B)  \Leftrightarrow  B \otimes A \neq 0 $.
\end{enumerate}
\end{Lem}

\begin{proof}
Since $\cat T$ is stratified, (a) implies (b) by \Cref{Thm:strat-minimality}(b) below.

Now suppose (b) is true and let $B \in\cat T$. From (b) we deduce that
\[ \tensLoc{ B \otimes A} = \tensLoc{ B \otimes g(\cat P) }, \]
so that in particular $B \otimes A\neq 0$ if and only if $B \otimes g(\cat P)\neq 0$, that is, (c) is true.

Finally, suppose (c) is true and let $\cat Q \in \Spc (\cat T^c)$ be any point of the spectrum.
By \cite[Lem.\,2.13]{BHS23}, we know that  $\supp (g(\cat Q)) =\{\cat Q\}$. 
From this together with~(c) (applied to $B = g(\cat Q)$) we deduce  that
\[ 
\cat Q = \cat P 
\quad\Leftrightarrow\quad
 \cat P \in \supp (g(\cat Q)) 
\quad \overset{(c)}{\Leftrightarrow} \quad
g(\cat Q) \otimes A \neq 0
\quad \Leftrightarrow \quad 
\cat Q \in \supp(A) , 
\]
that is $\supp (A) = \{\cat P \}$ as per~(a). 
Hence the three conditions are equivalent.
\end{proof}

The following criterion is a central result of stratification theory.

\begin{Thm}
\label{Thm:strat-minimality}
Let $\cat T$ be a big tt-category, and suppose (for simplicity) that $\Spc (\cat T^c)$ is noetherian.
Then the following are equivalent:
\begin{enumerate}
\item \emph{`Minimality':} $\tensLoc{g(\cat P)}$ is a minimal (nonzero) localizing tensor ideal of~$\cat T$ for every prime $\cat P\in \Spc (\cat T^c)$.
\item For every $A\in \cat T$, we have $\tensLoc{A} = \tensLoc{ g(\cat P) \mid \cat P \in \supp (A) }$.
\item $\cat T$ is stratified (\Cref{Ter:BHS-stratified}).
\end{enumerate}
\end{Thm}

\begin{proof}
This is \cite[Thm.\,4.1]{BHS23}, improved by \cite[Thm.\,3.21]{BHS23}, which says that the local-to-global principle (which in general must be included as part of requirement~(a)) automatically holds when the Balmer spectrum is noetherian.
\end{proof}

\begin{Ter} \label{Ter:min-at-P}
One says that \emph{minimality holds for $\cat T$ at $\cat P$} if the statement of~(a) holds for a specific tt-prime~$\cat P\in \Spc(\cat T^c)$.
\end{Ter}

The theorem implies that, in order to establish \emph{stratification} for~$\cat T$ with noetherian spectrum, it suffices to prove \emph{minimality} for~$\cat T$ separately at each point of its spectrum (as in \Cref{Ter:min-at-P}). 
In particular, one can work with different methods at different primes. 
We next mention two such methods we will use: a reduction by finite localization and a reduction by finite \'etale morphisms.

\begin{Rec}[Finite localizations]
\label{Rec:fin_locs}
Let $\cat T$ be a big tt-category.
A finite localization $F\colon \cat T\to \cat S$ is a tt-functor between big tt-categories which is (equivalent to) a Verdier quotient of the form $\cat T\to \cat T/\cat L$, whose kernel $\cat L=\Ker(F)$ on objects is generated as a localizing subcategory by a set of compact objects of~$\cat T$, or equivalently (since $\cat L\cap \cat T^c = \cat L^c$) such that $\cat L= \Loc{\cat L^c}$. Moreover, the canonical functor $\cat T^c / \cat L^c \to (\cat T/\cat L)^c$ is a cofinal fully faithful tt-functor (see~\cite{Neeman92b}), so in particular it induces a homeomorphism $  \Spc ((\cat T/\cat L)^c) \cong \Spc (\cat T^c / \cat L^c)$.
It follows that $\cat S = \cat T/\cat L$ is again a big tt-category in our sense.
If $\Spc(\cat T^c)$ is noetherian then so is $\Spc (\cat S^c)$, because the latter identifies with a subspace of the former by the induced embedding 
$\Spc ((\cat T/\cat L)^c) \cong \Spc (\cat T^c / \cat L^c) \to \Spc (\cat T^c)$, $\cat Q\mapsto F^{-1}(\cat Q)$ (\cite[Prop.\,3.13]{Balmer05a}).
\end{Rec}

\begin{Prop} 
\label{Prop:local-reduction}
Let $\cat T$ be a big tt-category with $\Spc (\cat T^c)$ noetherian and let $F\colon \cat T\to \cat  T/\cat L$ be a finite localization functor. 
Then for every $\cat Q\in \Spc ((\cat T/\cat L)^c)$, minimality at $\cat Q$ in $\cat T/\cat L$ holds if and only if minimality at $\cat P:= F^{-1}(\cat Q) \in \Spc (\cat T^c)$ holds in~$\cat T$.

\end{Prop}

\begin{proof}
This is essentially \cite[Prop.\,5.2 and Thm.\,3.22]{BHS23}. 
Applied to~$\cat T$, these results say that minimality in $\cat T$ at a point $\cat P\in \Spc (\cat T^c)$ holds if and only if the category $\cat T/ \Loc{\cat P}$ satisfies minimality at its unique closed point $\{0\}$.

Now suppose $\cat P\in \Spc (\cat T^c)$ is of the form $\cat P= F^{-1} (\cat Q)$ for $\cat Q\in \Spc (\cat T/\cat L)^c$, (\ie $\cat L\subseteq \cat P$ and $\cat Q = \cat P/\cat L$). 
Consider the induced commutative diagram of finite localization functors (\Cref{Rec:fin_locs}):
\[
\xymatrix{
\mathscr{ T} 
 \ar[r]^-{F} 
  \ar[d] & 
\mathscr T/\cat L 
 \ar[d] \\
\mathscr T/\Loc{ \cat P}
  \ar[r]^-{\cong} &
(\mathscr T/\cat L) / \Loc{ \cat Q }
}
\]
In details, the vertical arrows are the quotient functors by the localizing tensor ideals $\Loc{ \cat P} = \tensLoc{\cat P}$
 and $\Loc{ \cat Q } = \tensLoc{\cat Q}$, respectively, and the bottom horizontal arrow is the functor induced by the universal property of Verdier quotients; it is an equivalence of big tt-categories, since both composite functors from $\cat T$ to $(\cat T/\cat L)/ \Loc{ \cat Q }$ have the same kernel~$\cat P$.

Note that the above-cited result equally applies to the local category~$\cat T/\cat L$: minimality in $\cat T/\cat L$ at a point $\cat Q\in \Spc (\cat T/\cat L)^c$ holds if and only if the category $(\cat T/ \cat L)/ \Loc{ \cat Q}$ satisfies minimality at its unique closed point~$\{0\}$.  

Combining the two criteria, the commutativity of the square implies that minimality at $\cat P$ in $\cat T$ holds if and only if minimality at $\cat Q$ in $\cat T/\cat L$ holds, as claimed.
\end{proof}

\begin{Def}[Finite \'etale functor]
\label{Def:finite-etale}
A coproducts preserving tt-functor $F\colon \cat T\to \cat S$ between two rigidly-compactly generated tt-categories is said to be \emph{finite \'etale} if there exists a tt-ring (a.k.a.\ commutative separable monoid, see \Cref{subsection:tt-rings}) $A $ in $\cat T^c$ and an equivalence $\cat S\simeq \Mod_{\cat T}(A)$ of tt-categories which identifies $F$ with the free-$A$-module functor $\cat T\to \Mod_{\cat T}(A)$.
We refer to \cite{Sanders22} for explanations and many examples of this notion.
\end{Def}

\begin{War}
Beware of the terminology: finite localizations are smashing localizations and therefore are examples of \'etale functors. However localizations, even \emph{finite} ones, are only rarely \emph{finite \'etale}. (For localizations, `finite' refers to the kernel being generated by compact objects; for \'etale functors, `finite' refers to $U(\unit)$ being compact, where $U$ is the right adjoint.)
\end{War}

\begin{Thm} [{\cite[Thm.\,6.4]{BHS23}}]
\label{Thm:etale}
Let $F\colon \cat T\to \cat S$ be a finite \'etale functor (\Cref{Def:finite-etale}) of big tt-categories, and let $\varphi := \Spc (F^c) \colon \Spc (\cat S^c) \to \Spc (\cat T^c)$ denote the induced map between spectra, both of which we assume to be noetherian.  
Suppose $\cat Q \in \Spc (\cat S^c)$ is such that  $\varphi^{-1} (\{\varphi (\cat Q)\}) = \{\cat Q\}$ (for instance when $\varphi$ is injective). Then minimality at $\cat Q $ in $\cat S$ implies minimality at $\varphi(\cat Q)$ in~$\cat T$. 
\qed
\end{Thm}

We conclude our speedy tour of stratification theory by citing the following powerful `quasi-finite descent' (global) criterion for stratification, due to Barthel--Heard--Sanders together with Castellana,  which we will need in a situation when the hypothesis  on the fiber of $\varphi(\cat Q)$ in \Cref{Thm:etale}  fails:

\begin{Thm} [{\cite[Thm.\,17.16]{BCHS24pp}}]
\label{Thm:quasi-finite}
Let $F\colon \cat T\to \cat S$ be a coproducts preserving functor of big tt-categories and let $\varphi := \Spc (F^c)$ be the induced map of Balmer spectra, both of which we assume to be weakly noetherian.
Suppose that the right adjoint $U$ of $F$ preserves compact objects and that $\varphi$ is surjective with discrete fibers. 
Then if $\cat S$ is stratified so is~$\cat T$.
\qed
\end{Thm}

\subsection{A countable version of stratification}
\label{sec:countable-strat}%
\medskip

We introduce the countable variant of stratification which is required in order to make sense of the results and conjectures about $\Cell{G}$ we formulated in the Introduction.
This variant is rather straightforward, but we must pay attention to certain details.

Now $\cat T$ denotes a \emph{rigidly-compactly${}_{\aleph_1}$ generated} tt-category as in \Cref{Def:cbly-big-ttcat}, keeping in mind our main example $\cat T=\Cell{G}$ of the equivariant bootstrap category of a finite group~$G$ (\Cref{Prop:Cell(G)-cble-bigtt}). 

\begin{Cons} 
\label{Cons:countable-strat}
With $\cat T$ as above, assume for simplicity that the Balmer spectrum $\Spc (\cat T^c)$ is noetherian.
Let $Y$ be a Thomason (\ie specialization closed) subset of the Balmer spectrum $\Spc (\cat T^c)$ and $\cat{T}^c_Y$ the corresponding thick tensor ideal of~$\cat T^c$, as per \Cref{Thm:classif-tensorid}. 
We denote by $e_Y \to \unit \to f_Y \to \Sigma e_Y$ 
the associated idempotent triangle in~$\cat T$ which is uniquely determined by the properties 
\[
\Ker (f_Y \otimes -)= \Loc{\cat T^c_Y} = e_Y\otimes \cat T
\] 
(see \cite[\S2.1]{DellAmbrogio10} if necessary). 
For every $\cat P \in \Spc (\cat T^c)$, let 
\[
g(\cat P):= e_{Y_1} \otimes f_{Y_2}  \quad \in \cat T
\]
 where $Y_1,Y_2$ are specialization closed subsets such that $Y_1 \cap Y_2^c = \{\cat P\}$, for instance $Y_1=\overline{\{\cat P\}} = \{\cat Q \mid \cat Q \subseteq \cat P\}$ and $Y_2 = \{\cat Q \mid \cat P \not\subseteq \cat Q\}$.
For any object $A\in \cat T$, we define its \emph{countable Balmer--Favi support} to be
\[
 \supp (A) := \{\cat P \in \Spc (\cat T^c)\mid g(\cat P)\otimes A \not\cong 0\}
 \]
and we extend the definition to collections $\cat{S}$ of objects of $\cat T$ by
\[ 
\supp (\cat{S}) := \bigcup_{A\in \cat{S}}\supp (A) . 
\]
We say that $\cat T$ is \emph{stratified}, or \emph{countably stratified} for emphasis, if the resulting application $\cat S\mapsto \supp (\cat S)$ induces a bijection between the set of localizing tensor ideals of $\cat T$ and the set of all subsets of $\Spc (\cat T^c)$. 
\end{Cons}

\begin{Rem}
\label{Rem:basic-props}
The basic properties of the usual Balmer--Favi support remain true for its countable version, with only obvious minimal adjustments necessary. 
Indeed, much of the theory developed in \cite{BHS23} should go through, although there are limits to this due to the lack of a countable version of Brown representability for the dual (\ie for covariant homological functors), see \cite[Ex.\,2.11 and 2.22]{DellAmbrogio10}.
In any case, as explained in the introduction, we have chosen a different path towards applying stratification to equivariant KK-theory.
\end{Rem}

\section{Rational stratification for $G$-cell algebras}
\label{sec:strat-rat}

In this section we prove \Cref{Thm:intro-rational} of the Introduction, which in fact does not need any sophisticated stratification theory, nor the enlarged categories built in \Cref{sec:app-big}, but instead is a rather straightforward consequence of the (semi)simple structure of rational cell algebras.

Fix a finite group~$G$. 
We recall the standard (for tt-categories) notion of rationalization which is compatible with coproducts, as applied to $G$-cell algebras.

\begin{Ter}
\label{Ter:rational_cell}
Let $\Cell{G}$ be the $G$-equivariant bootstrap category of \Cref{def.CellCat}, and consider the Verdier quotient
\[
\Cell{G}_\mathbb Q :=\Cell{G}/\tensLoc{\cone(f)\mid f\in \mathbb Z\smallsetminus\{0\}  \subseteq \Rep(G)= \End_{\Cell{G}}(\unit) }.
\]
This is the \emph{rationalized} bootstrap category.
It is again a rigidly-compactly${}_{\aleph_1}$ generated tt-category, and is such that $\Cell{G}_\mathbb Q(A, B)=\mathbb Q \otimes_\mathbb Z \KK^G(A, B)$ whenever $A$ is a compact${}_{\aleph_1}$ object of $\Cell{G}$ and $B\in \Cell{G}$ arbitrary.
(See \eg\ \cite[\S2]{DellAmbrogio10}.) 
\end{Ter}

As mentioned in the introduction, this category admits a relatively simple purely algebraic model.
Namely, \cite[Thm.\,C]{BDM24pp} provides an equivalence
\begin{align} \label{eq:rat-equiv}
\Cell{G}_{\mathbb{Q}}
\simeq 
{\prod_{\mathrm{Cl}(H),\, H \leq G\; \textrm{cyclic}}}  
\underbrace{ \big( \mathbb{Q}(\zeta_{|H|}) \rtimes_c W_G(H) \big) \MMod^{\mathbb Z/2}_{\aleph_1} }_{=: \; \cat T_H}
\end{align}
 of tensor categories with a finite product of semisimple abelian tensor categories  $\cat T_H$ of countable $\mathbb Z/2$-graded modules over certain (ungraded) skew group rings $\mathbb{Q}(\zeta_{|H|}) \rtimes_c W_G(H)$, one for each conjugacy class of cyclic subgroups $H$ of~$G$.
 The suspension $\Sigma$ of $\Cell{G}$ corresponds on the right-hand side to degree-shift of graded modules.
 
We derive from this structural result the following observation:
 
\begin{Lem}
Each tt-category $\cat T_H$ is a (countable) \emph{tt-field} in the sense of \cite{BKS19} (adapted to the countable setting), meaning that it is rigidly-compactly${}_{\aleph_1}$ generated and satisfies two extra properties:
 \begin{itemize}
 \item[(F1)] The triangulated category $\cat T_H$ is \emph{pure-semisimple}, that is, each object is a coproduct of compact${}_{\aleph_1}$ objects.
 \item[(F2)] For every nonzero object $A\in \cat T_H$, the functor $A\otimes -\colon \cat T_H\to \cat T_H$ is faithful, that is any morphism $f$ in $\cat T_H$ is zero as soon as $A\otimes f$ is zero. (Clearly, by F1 it suffices to check this property for compact objects~$A$.)
 \end{itemize}
\end{Lem}

\begin{proof}
Since $\Cell{G}$ is rigidly-compactly${}_{\aleph_1}$ generated, so is its direct factor~$\cat T_H$ as one sees easily.
The rest follows from the next two facts  (\cf \cite[Prop.\,11.5]{BBB24pp}): 
\begin{itemize}
\item
$\cat T_H$ is semisimple abelian, and
\item the endomorphism ring $\End_{\cat T_H}(\unit)$ of its $\otimes$-unit is a field. 
\end{itemize}
The first fact was already mentioned. 
To see why the second fact is true, recall from \cite[\S2]{BDM24pp} that if a group $W$ acts on a ring $S$ via the action~$\alpha$, the tensor-unit for the module category over the skew group ring $S\rtimes_\alpha W$ is $(S,\alpha)$, and one can promptly verify that $\End(S,\alpha)\cong (S^W)^\op$ is the opposite of the fixed-point subring of~$S$, via $f\mapsto f(1)$.
Hence in the present case of $\cat T_H$, the ring $\End_{\cat T_H}(\unit)$ is a fixed-point subfield of the cyclic extension $\mathbb Q(\zeta_{|H|})$, and so it itself is a field.

Now to prove the lemma, note that by construction $\mathbb{Q}(\zeta_{|H|}) \rtimes_c W_G(H)$ is a finite dimensional algebra over the field~$\mathbb Q$ and therefore every (graded, countable) module is a coproduct of finite dimensional (graded, countable) modules; moreover, by semisimplicity, the finite dimensional modules are precisely the compact objects of~$\cat T_H$. This implies~(F1). 

As for~(F2), let $C\in \cat T_{H}^c$ be any compact object. 
If $C\neq 0$, the dimension
\[
\dim (C) \colon \unit \to C \otimes C^\vee \to \unit
\]
(\ie the monoidal trace of $\id_C$, which is well-defined since $C$ is compact and hence rigid) is a nonzero endomorphism of~$\unit$. 
As $\End_{\cat T_H}(\unit)$ is a field, it follows that $\dim (C)$ is invertible.
By tensoring $\dim(C)$ with any map~$f$ of~$\cat T$, we see that $f$ is a retract of~$C\otimes f$.
Hence in particular $C\otimes f=0$ implies $f=0$, proving~(F2).
\end{proof}
 
 At this point, the stratification of $\Cell{G}_\mathbb Q$ and the finite discreteness of its spectrum can be derived as in \cite[Prop.\,18.3 and Thm.\,18.4]{BCHS24pp}, using the pure semisimplicity (property~(F1)) of $\Cell{G}_\mathbb Q$ inherited from that of its factors $\cat T_H$. However, this argument is rather long as it would require adapting many general results to the countable case.  
Here below we provide more direct proofs.
 
\begin{Cor}
\label{Cor:tt-field}
As for any tt-field, the spectrum $\Spc (\cat T_H^c) = \{(0)\}$ has a single point, and the only localizing tensor ideals of $\cat T_H$ are the trivial ones, $(0)$ and~$\cat T_H$.
\end{Cor}

\begin{proof}
Let $\cat T $ be any tt-field as above. 
The claim on the spectrum, which only involves the tt-category of compacts, is precisely \cite[Prop.\,5.15]{BKS19}; this is equivalent to saying that  any nonzero compact object generates $\cat T^c$ as a thick $\otimes$-ideal.

Similarly $\tensLoc{A}= \cat T$ for any nonzero object $A\in \cat T$.
This is because by (F1) each $A\neq 0$ contains a nonzero compact direct summand~$C$, hence $\tensThick{C}= \cat T^c$ and therefore $\cat T=\tensLoc{\cat T^c}=\tensLoc{C}\subseteq \tensLoc{A}$, so that $\cat T= \tensLoc{A}$.
\end{proof}

\begin{Cor}\label{Cor:spc-rational}
The comparison map $\rho$ is a homeomorphism $\Spc( \Cell{G}^c_{\mathbb{Q}} )\cong \Spec( \Rep(G)_{\mathbb{Q}}) $ of finite discrete spaces, whose points are in bijection with the conjugacy classes of cyclic subgroups of~$G$.
\end{Cor}

\begin{proof}
Recall from \Cref{subsec:rho} that the Balmer and Zariski spectra are functorial (for tt-functors and ring morphisms respectively) and note that they send finite products to finite disjoint unions of spaces. 
Moreover, $\rho$ is natural in the tt-category. 

It follows that the equivalence \eqref{eq:rat-equiv} induces a commutative square 
\[
\xymatrix{
{\underset{\mathrm{Cl}(H)}{\coprod}}\Spc(\cat T_{H}^c) 
 \ar[r]^-{\cong} 
  \ar[d]_-{{\coprod} \; \rho_{\cat T^c_H}} & 
\Spc(\Cell G_{\mathbb{Q}}^c) 
 \ar[d]^-{\rho_{\Cell{G}^c_\mathbb{Q}}} \\
{\underset{\mathrm{Cl}(H)}{\coprod}} \Spec(\End_{\cat T_H^c}(\unit))
  \ar[r]^-{\cong} &
\Spec(\Rep(G)_{\mathbb{Q}})
}
\]
where the horizontal maps (induced by the projection tt-functors $\Cell{G}^c_\mathbb Q\to \cat T_H^c$) are isomorphisms of finite discrete spaces. In particular, 
Each map $\rho_{\cat T^c_H}$ on the left is a homeomorphism by \Cref{Cor:tt-field}, hence so is the right vertical map.
\end{proof}

\begin{Thm}\label{Thm:strat-rational}
For every finite group~$G$, the tt-category $\Cell{G}_\mathbb Q$ is countably stratified in the sense of \Cref{Cons:countable-strat}.
\end{Thm}

\begin{proof}
We must show that the map
\[
\sigma\colon 
\left\{
\begin{array}{c} 
\textrm{localizing tensor ideals of } \Cell{G}_\mathbb Q
\end{array}
\right\}
\longrightarrow
\left\{
\begin{array}{c} 
\textrm{subsets of } \Spc(\Cell{G}_\mathbb Q^c)
\end{array}
\right\}
\]
sending a localizing subcategory $\cat L$ to the subset 
\[
\supp(\cat L) = \bigcup_{A\in \cat L}  \supp(A) = \{ \cat P \mid \exists\, A\in \cat L \textrm{ such that } g(\cat P)\otimes A \not\cong 0\}
\]
is a bijection. 
Note that the space $\Spc(\Cell{G}_\mathbb Q^c)$, being finite discrete, is noetherian, hence the construction of the objects $g(\cat P)$ makes sense for all points~$\cat P$ and the countable Balmer--Favi support is well-defined.
In particular $\sigma$ is surjective (as always in this situation) because for any subset $S\subseteq \Spc(\Cell{G}^c_\mathbb Q)$ we have $\supp (\tensLoc{g(\cat P) \mid \cat P \in S}) = \supp(\{g(\cat P) \mid \cat P\in S\}) = S$ by the basic compatibility axioms of a support theory (\cf \Cref{Rem:basic-props}).

It remains to show injectivity, which in this case can also be easily seen directly.
Indeed, let us identify each $\cat T_H$ with its image inside $\Cell{G}_\mathbb Q$ via \eqref{eq:rat-equiv}, and consider the corresponding decomposition $\unit\cong \bigoplus_{H} \unit_H$ of the tensor unit, with $\unit_H$ being the (image of) the tensor unit of~$\cat T_H$. 
Write $\cat P_H$ for the tt-prime indexed by the conjugacy class of the cyclic subgroup~$H$. 

By discreteness, both $Y_1=\{\cat P_H\}$ and its complement $Y_2$ are Thomason (\ie specialization closed) subsets of the spectrum.
Using the notations of \Cref{Cons:countable-strat} for the defining $\otimes$-idempotents and their idempotent triangles (which are all split by semisimplicity), we immediately deduce from the equalities $\Loc{\cat T^c_H} = \cat T_H = \unit_H \otimes \Cell{G}_\mathbb Q = \Ker (\prod_{L\neq H}\unit_L\otimes -)$ that $g(\cat P_H)= e_{Y_1} = f_{Y_2} = \unit_H$. 

Note that by \eqref{eq:rat-equiv} every object $A\in \Cell{G}_\mathbb Q$ decomposes as $A\cong \bigoplus_H \unit_H\otimes A$ with $\unit_H\otimes A\in \cat T_H$
Similarly, if $\cat L$ is a localizing $\otimes$-ideal (\ie by semisimplicity: a full $\otimes$-ideal closed under translations, retracts and countable direct sums), we have $\cat L = \bigoplus_H \unit_H\otimes \cat L$, with $\unit_H \otimes \cat L$ a localizing ideal of~$\cat T_H$. 
By \Cref{Cor:tt-field}, $\unit_H \otimes \cat L$ must be either zero or the whole of~$\cat T_H$, and the latter happens iff $g(\cat P_H)=\unit_H$ belongs to~$\cat L$.
Therefore we have $\cat L = \bigoplus_{\cat P_H \in \supp (\cat L)} \cat T_H =  \bigoplus_{\cat P_H \in \supp (\cat L)} \tensLoc{g(\cat P_H)}$, and the latter is equal to $ \tensLoc{g(\cat P) \mid \cat P\in \supp(\cat L)}$ by semisimplicity and the fact that the factors $\cat T_H$ are $\otimes$-orthogonal.
Therefore each localizing tensor ideal is determined by its support, showing the injectivity of~$\sigma$.
\end{proof}

\section{Integral stratification for groups with prime-order elements}
\label{sec:main}%


The goal of this section is to prove \Cref{Thm:strat-Cell(C_p)} of the Introduction, which says that the tt-category $\Cell{G}$ is stratified in the countable sense of \Cref{Cons:countable-strat} whenever $G$ is a finite group where every nontrivial element has prime order.

\begin{Rem} \label{Rem:our-class}
The latter condition on a finite group is obviously equivalent to requiring that all its cyclic subgroups are either the trivial group or of prime order; this is the form in which we will use the hypothesis. 
Such groups have been classified, see \cite{CDLS93}~\cite{Deaconescu89}. 
They consist precisely of: 
the $p$-groups with exponent $p$ for a prime number~$p$ (this includes elementary abelian $p$-groups, \ie those of the form $C_p^{r} = C_p \times\ldots \times C_p$ for $r\in \mathbb N$, as well as nonabelian examples for all $p\neq 2$);
Frobenius groups of a certain type; 
and the alternating group~$A_5$.
An example of the Frobenius kind is~$S_3$, the permutation group of three elements.
\end{Rem}

In this section we use the enlargement of KK-theory constructed in \Cref{sec:app-big}. 

\begin{Hyp} 
\label{Hyp:C_p}
We only need the following facts, all part of \Cref{Thm:bigCell}:
\begin{enumerate}
\item For every finite group~$G$, there exists a rigidly-compactly generated tt-category (in the sense of \Cref{subsec:bit-ttcats}) $\bigCell{G}$ together with a fully faithful coproduct-preserving tt-functor $\iota_G\colon \Cell{G}\to \bigCell{G}$.
\item The above embedding $\iota_G$ restricts to an equivalence $\Cell{G}^c\overset{\sim}{\to} \bigCell{G}^c$ between the tt-subcategories of rigid-compacts${}_{\aleph_1}$ and  of rigid-compacts.
\item For every subgroup~$H$ of a finite group~$G$, there is a \emph{restriction} tt-functor $\Res^G_H\colon \bigCell{G}\to \bigCell{H}$ extending (up to isomorphism) the usual restriction functor in KK-theory along the embeddings.
\item Each restriction functor $\Res^G_H$ admits a two-sided adjoint $\Ind^G_H$ extending the usual induction in KK-theory, and the adjunctions can be chosen so that the composite $\Id\to \Res^G_H\Ind^G_H\to \Id$ of the unit of one adjunction followed by the counit of the other one is the identity of~$\Id_{\bigCell{H}}$.
\end{enumerate}
\end{Hyp}

\begin{Rem}
\label{Rem:pres-cpts}
Like any symmetric monoidal functor, each restriction $\Res^G_H$ preserves rigid hence compact objects. 
Each induction $\Ind^G_H$ does as well by \cite[Thm.\,5.1]{Neeman96}, since it has a coproduct preserving right adjoint (namely~$\Res^G_H$).
\end{Rem}

As explained in the Introduction, we deduce \Cref{Thm:strat-Cell(C_p)} from the following:

\begin{Thm}
\label{Thm:strat-big-one}
Let $G$ be a finite group with every nontrivial element of prime order.
Then the (genuinely) rigidly-compactly generated category $\bigCell{G}$ of \Cref{Thm:bigCell} is stratified by the Balmer--Favi support theory, in the sense of~\cite{BHS23}.
\end{Thm}

\begin{proof}[Proof of \Cref{Thm:strat-Cell(C_p)} from \Cref{Thm:strat-big-one}]
For any finite group~$G$ the embedding $\iota_G$ satisfies the hypotheses of \Cref{Prop:strat-restriction} (because of parts (a) and~(b) of \Cref{Hyp:C_p}).
If $G$ only has prime-order nontrivial elements, \Cref{Thm:strat-big-one} says that $\bigCell{G}$ is stratified in the usual sense and it therefore follows from  \Cref{Prop:strat-restriction} that $\Cell{G}$ is stratified in the countable sense.
\end{proof}

In the remainder of this section we will prove \Cref{Thm:strat-big-one} and will only work with the `genuine' big tt-category $\bigCell{C_p}$, to which we can directly apply all known results for big tt-categories including the stratification theory recalled in \Cref{sec:strat-prelim}.
To simplify notations, for all $G$ we will treat the embedding $\iota_G$ as the inclusion of a tt-subcategory and accordingly we will identify their rigid objects
\[
\Cell{G}^c= \bigCell{G}^c.
\]
as well as the restriction and induction functors when restricted to these categories.

In order to prove \Cref{Thm:strat-big-one}, we will build on the case when $G=C_p$ is a cyclic group of prime order~$p$.
By the main result of~\cite{DellAmbrogioMeyer21}, we already know that in this case the canonical comparison map (\Cref{subsec:rho})
\begin{equation} \label{eq:Spc(C_p)}
\rho\colon \Spc (\Cell{C_p}^c) \overset{\sim}{\longrightarrow} \Spec (\Rep(C_p))
\end{equation}
is a homeomorphism.
We begin by generalizing this to our larger class of groups:

\begin{Thm} \label{Thm:Spc-prime-order}
The comparison map $\rho$ is a homeomorphism $\Spc(\Cell{G}^c)\overset{\sim}{\to}\Spec(\Rep(G))$
for any finite group $G$ with only prime-order nontrivial elements.
\end{Thm}

\begin{proof}
Consider the tt-functor on compact objects
\[
F:= (\Res^G_H)_H \colon \Cell{G}^c \longrightarrow \prod_{H} \Cell{H}^c
\]
where $H$ runs through all the finitely many cyclic subgroups of~$G$.
By the hypothesis on $G$, every $H$ is either cyclic of prime order or the trivial group.

Since $\rho$ is natural and maps finite products to finite coproducts, $F$ induces the following commutative diagram of spaces:
\[
\xymatrix{
{\underset{H}{\coprod}}\Spc(\Cell{H}^c) 
 \ar[r]^-{\Spc(F)}
  \ar[d]_-{{\coprod_H} \; \rho_{H}} & 
\Spc(\Cell {G}^c) 
 \ar[d]^-{\rho_G} \\
{\underset{H}{\coprod}} \Spec(\Rep(H))
  \ar[r]^-{\Spec(f)} &
\Spec(\Rep(G))
}
\]
The ring morphism $f = (\res^G_H)_H \colon \Rep(G)\to \prod_H \Rep(H)$ has for components the usual restriction maps of the character ring; it also agree with the the morphism obtained by specializing $F$ to the endomorphism rings of the unit objects. 

We need to show that $\rho_G$ is a homeomorphism. 
In fact, by  \cite[Cor.\,2.8]{Eike21pp} we only need to show that it is bijective, since the tt-category $\Cell{G}^c$ is rigid and satisfies the finiteness hypotheses of \emph{loc.\,cit.\ }precisely thanks to \Cref{Lem:End-finite}. 

Since the graded endomorphism ring $\End^*_{\Cell{G}^c}(\unit)= \Rep(G)[\beta^{\pm1}]$ is graded noetherian (\Cref{Lem:End-finite} again), $\rho_G$ is surjective by  \cite[Cor.\,7.4]{Balmer10b}. 
It remains to see that it is injective.

Note that, in the above diagram, the left vertical map is a homeomorphism because its component at each $H$ is a homomorphism, either by \eqref{eq:Spc(C_p)} (if $H$ has prime order) or else by \cite[Thm.\,1.2]{DellAmbrogio10} (if $H$ is the trivial group).

The bottom map $\Spec(f)$  is surjective by \cite[Prop.\,3.7]{Segal68a}.
More precisely, Segal's proposition implies that every prime ideal $\mathfrak p \subset \Rep(G)$ is of the form $\mathfrak p_{H,\mathfrak q}:= (\res^G_H)^{-1}(\mathfrak q)$ for some prime ideal $\mathfrak q \subset \Rep(H)$, where $H$ is a cyclic subgroup of~$G$ which can be chosen such that $\mathfrak q$ does \emph{not} come from a smaller (cyclic) subgroup (\ie $H$ is the `support' of $\mathfrak p$ in Segal's terminology).
For each $\mathfrak p$, such a minimal pair $(H,\mathfrak q)$ with $\mathfrak p=\mathfrak p_{H,\mathfrak q}$ is unique up to conjugation by an element of~$G$. 

We claim that the top horizontal map $\Spc(F)$ is also surjective. 
Indeed, note that the tt-functor $F$ admits the right adjoint $U:= \bigoplus_H \Ind^G_H$.
Since $\Cell{G}^c$ is rigid, we can apply \cite[Thm.\,1.7]{Balmer18} to the adjunction $F\dashv U$ to conclude that the image of $\Spc(F)$ is equal to the support of the object $U(\unit) = \bigoplus_H \Cont(G/H)$. 
Since we have $\Cell{G}^c = \Thick{ \bigoplus_H \Cont(G/H) \mid H \textrm{ cyclic} }$ by \Cref{Thm:cyclic-generation}, we deduce that this image is the whole spectrum of $\Cell{G}^c$.
Hence $\Spc(F)$ is surjective as claimed.

Finally, to verify the injectivity of~$\rho_G$ take two tt-primes $\cat P_1, \cat P_2 \in \Spc(\Cell{G}^c)$ such that $\rho_G(\cat P_1)= \rho_G(\cat P_2)=: \mathfrak p \in \Spec(\Rep(G))$.
By the surjectivity of $\Spc(F)$, they are of the form $\cat P_i=(\Res^G_{H_i})^{-1}\cat Q_i$ for some $\cat Q_i \in \Spc(\Cell{H_i}^c)$ and $H_i$ ($i=1,2$).
The tt-primes $\cat Q_i$ correspond bijectively to some $\mathfrak q_i := \rho_{H_i}(\cat Q_i) \in \Spec(\Rep(H_i))$. 

Note that we may (and will) assume that each $\mathfrak q_i$ does not come from a smaller subgroup of~$H_i$, because if it does, we can swap it for a prime $\mathfrak q_i' \subset \Rep(H_i')$ ($H_i'\subsetneq H_i$) mapping to $\mathfrak q_i$ and such that $H_i'$ is minimal (among those for which there exists an antecedent of~$\mathfrak p$); and since the vertical maps $\rho_H$ on the left are bijections, we can also swap the tt-primes $\cat Q_i$, in parallel, for the unique $\cat Q_i'$ corresponding to the~$\mathfrak q_i'$. 
Clearly $\Spc(\Res^{H_i}_{H_i'})(\cat Q_i')  = \cat Q_i$ and therefore $\Spc(\Res^G_{H_i'})(\cat Q_i')= \cat P_i$ too, by restriction in stages.

Now, by the commutativity of the above square, we have 
\[
(\res^G_{H_1})^{-1}(\mathfrak q_1) = \mathfrak p = (\res^G_{H_2})^{-1}(\mathfrak q_2).
\]
Since we have assumed that the pairs $(H_i, \mathfrak q_i)$ are both minimal for~$\mathfrak p$, they must be conjugate: there exists an element $g\in G$ such that ${}^gH_1 = H_2$ and 
$(\conj_g)^{-1}(\mathfrak q_1) = \mathfrak q_2$, where $\conj_g\colon \Rep({^g}H_1)\to \Rep(H_1)$ is the conjugation ring map.
The latter is just the restriction, at the endomorphism ring of the unit, of the conjugation functor $\Con_g\colon \Cell{{}^gH_1}\to \Cell{H_1}$.
Therefore we must have $(\Con_g)^{-1}(\cat Q_1)=\cat Q_2$, using once again the bijectivity of the left vertical maps.

Since $\Con_g \circ \Res^G_{{}^gH_1}  \cong \Res^G_{H_1} \circ \Con_g  \cong \Res^G_{H_1}$, we deduce from this that 
\[
\cat P_1 = (\Res^G_{H_1} )^{-1} (\cat Q_1)
=  (\Res^G_{{}^gH_1})^{-1} (\Con_g)^{-1} (\cat Q_1)
= (\Res^G_{H_2} )^{-1} (\cat Q_2)  = \cat P_2.
\]
Thus $\rho_G$ is indeed injective, and therefore is a homeomorphism.
\end{proof}

\begin{Lem} \label{Lem:Cell-finite-etale}
For every subgroup $H\leq G$, restriction $\Res^G_H\colon \bigCell{G}\to \bigCell{H}$ is a finite \'etale tt-functor  (\Cref{Def:finite-etale}).
\end{Lem}

\begin{proof}
The analog result is proved in \cite{BDS15} for the usual restriction functor $\KK^G\to \KK^H$, and a similar argument will work in the present situation.
Namely, as $\Res^G_H$ is a coproduct-preserving tt-functor between two rigidly-compactly generated tt-categories, by \cite[Prop.\,2.5]{BDS16} the adjunction $\Res^G_H\dashv \Ind^G_H$ satisfies the canonical `right projection formula' 
\[
\Ind^G_H(B \otimes \Res^G_H(A)) \cong \Ind^G_H(B) \otimes A
\]
in the sense of \cite[Def.\,2.7]{BDS15}. 
Moreover, by \Cref{Hyp:C_p}(d), the counit of this same adjunction admits a natural section. 
By \cite[Thm.\,2.9]{BDS15}, it follows from the two latter properties and the fact that the categories are idempotent-complete that the canonical comparison functor 
\[ 
\bigCell{H}\overset{\sim}{\to} \Mod_{\bigCell{G}}(\Cont(G/H)) 
\]
is an equivalence, where $\Cont(G/H)= \Ind^G_H(\unit)$ is equipped with the commutative monoid structure in $\Cell{G}^c = \bigCell{G}^c\subseteq \bigCell{G}$ inherited from $\Res^G_H\dashv \Ind^G_H$.

As recalled in \Cref{subsec:C(G/H)}, this monoid is a tt-ring, and we conclude that $\Res^G_H$ is a finite \'etale tt-functor.
\end{proof}

Next, we establish stratification for a group of prime order, $G=C_p$.
For this we need to redeploy the same divide-and-conquer strategy used to establish~\eqref{eq:Spc(C_p)} in~\cite{DellAmbrogioMeyer21}, which now involves considering two tt-functors out of $\bigCell{C_p}$:
\begin{equation} \label{eq:two-tt-functors}
\xymatrix{
\bigCell{1} & \bigCell{C_p} \ar[l]_-{\Res^{C_p}_1} \ar[r]^-{Q} & \bigCell{C_p}/\Loc{ \Cont(C_p) } =: \mathrm Q(C_p) 
}
\end{equation}
The left one is restriction to the trivial subgroup. 
The right one is the Verdier quotient by the localizing subcategory generated by the compact generator $\Cont(C_p)$. 
\begin{Lem}
\label{Lem:two-tt-functors}
The functor $Q$ in \eqref{eq:two-tt-functors} is a finite localization (\Cref{Rec:fin_locs}).
\end{Lem}
\begin{proof}
First note that $\Ker (Q) = \Loc { \Cont(C_p) }$ is a tensor ideal (this follows easily from the decomposition $\Cont(C_p)\otimes \Cont(C_p)\cong \Cont(C_p\times C_p) \cong \Cont(C_p)^{\oplus p}$ and the fact that $\bigCell{C_p}=\Loc{ \mathbb C, \Cont(G)}$).
Therefore $\Loc{ \Cont(G) } = \tensLoc { \Cont(G) }$, and $Q$ is a tensor-triangulated functor as claimed. 
It is a finite localization because $\Cont(G)$ is a compact-rigid object of $\bigCell{C_p}$. 
\end{proof}

\begin{Rec}
\label{Rec:Spc-decomp}
As discussed in \cite[Rem.\,6.4]{DellAmbrogioMeyer21}, the Zariski spectrum of the representation ring $\Rep(C_p)\cong \mathbb Z[\widehat{C_p}] \cong \mathbb Z [x]/(x^p-1)$ decomposes via two maps
\begin{equation} \label{eq:spec-decomp}
\vcenter{
\xymatrix{
\Spec (\mathbb Z) \ar[r] &
 \Spec (\mathbb Z[x] / (x^p-1)) &
  \Spec (\mathbb Z[\vartheta,p^{-1}]) \ar[l]
}}
\end{equation}
which are homeomorphisms on their images and are jointly surjective with disjoint images. The left map is induced by sending $x$ to zero, the right one by inverting $p$ and sending the cyclotomic polynomial $\Phi(X)= 1 +2 +\ldots + x^{p-1}$ to zero; here $\vartheta$ denotes the image of~$x$ in the resulting ring.
Now, note that the two functors in~\eqref{eq:two-tt-functors} are symmetric monoidal ($Q$ by~\Cref{Lem:two-tt-functors}) and therefore can be restricted to rigid-compact objects
\[ 
\xymatrix{
\Cell{1}^c & \Cell{C_p}^c \ar[l]_-{\Res^{C_p}_1} \ar[r]^-{Q} & \mathrm Q(C_p)^c 
}
\]
and then induce maps on Balmer spectra.
Together with the natural maps~$\rho$, the induced maps form the following commutative diagram of topological spaces:
\begin{equation*} 
\vcenter{
\xymatrix{
\Spc (\Cell{1}^c) \ar[d]_{\rho}^\cong \ar[rr]^-{\Spc (\Res^{C_p}_1)} &&
 \Spc (\Cell{C_p}^c)  \ar[d]_{\rho}^\cong  &&
  \Spc (\mathrm Q(C_p)^c) \ar[ll]_-{\Spc (Q)} \ar[d]_{\rho}^\cong \\
\Spec (\Rep(1)) \ar[rr]^-{} &&
 \Spec (\Rep(C_p)) &&
  \Spec (\End_{\mathrm Q(C_p)}(\unit)) \ar[ll]_-{}
}}
\end{equation*}
By the computation in \cite[Prop.\,4.10]{DellAmbrogioMeyer21}, the bottom line identifies with~\eqref{eq:spec-decomp}.
As all three instances of $\rho$ are homeomorphisms, the top row similarly displays $\Spc (\Cell{C_p}^c)$ as the disjoint union of the two outer Balmer spectra.
Indeed, the images of the two maps are precisely the complementary sets
\[
\Img (\Spc(\Res^{C_p}_1)) = \supp (\Cont(C_p)) = \{\cat P \mid \Cont(C_p) \not\in \cat P\} 
\]
and
\[
\Img (\Spc(Q)) = \{ \cat P \mid \Cont(C_p) \in \cat P\},
\]
respectively. 
\end{Rec}

\begin{Prop}
\label{Prop:strat-boot}
The big tt-categories $\bigCell{1}$ and $\mathrm Q(C_p)$ are stratified.
\end{Prop}

\begin{proof}
Both cases can be treated with the formal results of \cite{DellAmbrogioStanley16}.
To wit, \cite[Thm.\,1.3]{DellAmbrogioStanley16} shows that if a big tt-category $\cat T$ satisfies the following two hypotheses
\begin{enumerate}[\rm(1)]
\item $\cat T$ is monogenic, \ie $\cat T=\Loc{\unit}$, and
\item as a graded commutative ring, $\End^*_\cat T(\unit)$ is noetherian and regular,
\end{enumerate}
then $\cat T$ is stratified (in the general sense of \Cref{Def:general-supp-and-strat}) by the Benson--Iyengar--Krause support theory of \cite{BensonIyengarKrause08} with respect to the canonical action of $\End^*_\cat T(\unit)$ on~$\cat T$.
Note that we do not need to know what the latter support theory is because, if it stratifies~$\cat T$, it must agree with the Balmer--Favi support theory thanks to \cite[Cor.\,7.11]{BHS23} (which makes use of the finiteness property in~\Cref{Lem:End-finite}).

Thus it remains to show that $\bigCell{1}$ and $\mathrm Q(C_p)$ satisfy conditions (1) and~(2). 
Condition (1) is obviously true by construction in both cases.
Condition (2) is also verified, because we know that $\End^*_{\bigCell{1}}(\unit) \cong \mathbb Z[\beta^{\pm 1}]$ and $\End^*_{\mathrm Q(C_p)} \cong \mathbb Z[\vartheta,p^{-1}] [\beta^{\pm 1}]$, where $\beta$ is invertible of degree two and the ring in front of it lies in degree zero.
In particular, both graded rings are regular and noetherian (\cf \cite[Lem.\,6.2]{DellAmbrogioMeyer21} if needed for the regularity of the second one).
This proves the claim. 

(Alternatively for the Bootstrap category $\bigCell{1}$, the result can be proved almost \emph{verbatim} as the main result of~\cite{DellAmbrogio11}, which is about the usual countable version $\Cell{1}$ of the Bootstrap category, by omitting all countability restrictions in the arguments and by identifying the stratifying support theory $(\Spec (\mathbb Z), \sigma)$ used in \emph{loc.\,cit.\ }first with the Benson-Iyengar-Krauses support theory and then with the Balmer--Favi support theory, once again thanks to \cite[Cor.\,7.11]{BHS23}.)
\end{proof}

\begin{Lem} \label{Lem:stratif-for-Cp}
The big tt-category $\bigCell{C_p}$ is stratified for any prime number~$p$.
\end{Lem}

\begin{proof}
Our proof strategy is to apply \Cref{Thm:strat-minimality} by proving minimality separately for each of the two kinds of primes ideals in $\Cell{C_p}^c$, according to the decomposition of \Cref{Rec:Spc-decomp}. Let $\cat P\in \Spc (\Cell{C_p}^c)$.

Suppose first that $\cat P$ lies in the image of $\varphi:=\Spc (\Res^{C_p}_1)$, \ie that there exists a $\cat Q\in \Spc (\bigCell{1})$ such that $\cat P = (\Res^{C_p}_1)^{-1}(\cat Q) = \varphi(\cat Q)$.
We can invoke the finite \'etale reduction \Cref{Thm:etale} to deduce the minimality at $\cat P$ in $\bigCell{C_p}$ from the minimality at $\cat Q $ in $\bigCell{1}$,
because the functor $\Res^{C_p}_1$ is finite \'etale, both the source spectrum $\Spc (\Cell{1}^c) \cong \Spec (\mathbb Z)$ and the target spectrum $\Spc (\Cell{C_p}^c)\cong \Spec (\mathbb Z[x]/(x^p-1))$ are noetherian spaces, and the map $\varphi$ is injective. 
Since $\bigCell{1}$ is stratified by \Cref{Prop:strat-boot}, minimality holds at all such $\cat Q\in \Spc (\bigCell{1})^c$, and we conclude.

Now suppose instead that $\cat P$ is in the image of $\varphi:= \Spc (Q)\colon \Spc (\mathrm Q(C_p)^c) \to \Spc (\Cell{C_p}^c)$, \ie that there exists a $\cat Q\in \Spc (\mathrm Q(C_p)^c)$ such that $\cat P= Q^{-1}(\cat Q)= \varphi(\cat Q)$. 
 By \Cref{Prop:local-reduction} applied to $F:=Q$, minimality at $\cat P$ in $\bigCell{C_p}$ is equivalent to the minimality at $\cat Q$ in~$\mathrm Q(C_p)$.
But the latter holds for all $\cat Q$, since $\mathrm Q(C_p)$ is also stratified  (\Cref{Prop:strat-boot}). 
\end{proof}

\begin{proof}[Proof of \Cref{Thm:strat-big-one}]
We are going to apply \Cref{Thm:quasi-finite} to the functor 
\[
F:= (\Res^G_H)_H \colon \cat T:=\bigCell{G} \longrightarrow \prod_{H} \bigCell{H} =: \cat S
\]
where $H$ runs through the finitely many cyclic subgroups of~$G$ (we already used the restriction of $F$ to compact objects in the proof of \Cref{Thm:Spc-prime-order}; here we need the whole functor as in \Cref{Hyp:C_p}).
Thus $F$ is a coproducts preserving tt-functor between big tt-categories whose Balmer spectra are noetherian by \Cref{Thm:Spc-prime-order}.
As obtained in the course of proving the latter proposition (a consequence of \Cref{Thm:cyclic-generation}), we know that the map $\varphi= \Spc(F^c)\colon \Spc(\cat D^c)\to \Spc(\cat C^c)$ is surjective.

Moreover, the fibers of $\varphi$ are discrete. 
This is the same as saying that each component map $\Spc((\Res^G_H)^c)\colon \Spc(\Cell{H}^c)\to \Spc(\Cell{G}^c)$ has discrete fibers,
which is true by \cite[Thm.\,1.5]{Balmer16} since 
\[
\Res^G_H\colon \bigCell{G}\to \bigCell{H}\simeq \Mod_{\bigCell{G}}(\Cont(G/H))
\]
is a finite \'etale tt-functor (by \Cref{Lem:Cell-finite-etale}) whose tt-ring $\Cont(G/H) = \Ind^G_H(\unit)$ has finite degree (by \Cref{Prop:finite-degree}).

The adjunctions $\Ind^G_H \dashv \Res^G_H \dashv \Ind^G_H$ for each $H$ (\Cref{Hyp:C_p}(d)) can be easily assembled into two adjunctions $U\dashv F \dashv U$, where  $U := \bigoplus_H \Ind^G_H\colon \cat S\to \cat T$. 
In particular, we see that the right adjoint $U$ of $F$ preserves compact objects, since its right adjoint ($F$) itself admits a coproducts-preserving right adjoint~($U$).

By the assumption on~$G$, each $H$ occurring in the product is either trivial or of prime order. 
Therefore each factor $\bigCell{H}$ of $\cat S$ is stratified by \Cref{Prop:strat-boot} (if $H$ is trivial) or \Cref{Lem:stratif-for-Cp} (if $H$ has prime order), and it follows easily that~$\cat S$, as a finite product of stratified big tt-categories, is also stratified.

All in all, we have verified that $F$ satisfies all the hypotheses of \Cref{Thm:quasi-finite} and we conclude that $\cat T= \bigCell{G}$ is stratified.
This completes the proof of \Cref{Thm:strat-big-one} and therefore also of \Cref{Thm:strat-Cell(C_p)}.
\end{proof}

\begin{Rem}
The fact that the map $\varphi$ (as in the proof) has discrete fibers can also be proved by using \Cref{Thm:Spc-prime-order} and the naturality of~$\rho$ to translate the question in terms of the analog map for the Zariski spectra of the representation rings, where it should follow from Segal's more precise structural results in~\cite{Segal68a}.

\end{Rem}

\begin{Rem}
If we could somehow prove  for finite cyclic groups of \emph{any} order that $\rho$ is a homeomorphism for $\Cell{H}^c$ and that $\bigCell{H}$ is stratified, the precise same arguments employed in this section would imply both statements for arbitrary finite groups~$G$, and we could then deduce that Conjectures~\ref{Conj:spc} and~\ref{Conj:strat} hold for all~$G$.
\end{Rem}

\appendix
\section{A genuinely big version of equivariant KK-theory}
\label{sec:app-big}%

In this Appendix we use the toolkit of Lurie's $\infty$-categories \cite{LurieHTT} \cite{LurieHA} in order to enlarge the equivariant Kasparov categories $\KK^G$ and their subcategories of cell algebras $\Cell{G}$ so that they contain arbitrary small coproducts. 
For our purposes, the enlargement must be done in a way which preserves all the structure we already have (the triangulation, the tensor product, the countable coproducts) and which accommodates extensions of the usual basic change-of-group functors (restriction and induction) and the adjunctions and formulas between them.

We do not claim any originality here, indeed our approach is adapted from that of recent work by Bunke--Engel--Land~\cite{BEL23pp}. 
Specifically, we use the same $\infty$-categorical enhancement of $\KK^G$ they use, but then we enlarge it by applying the relative $\Ind_{\aleph_1}$ construction, rather than the usual $\Ind =\Ind_{\aleph_0}$ construction. 
This change is to ensure that the countable coproducts of $\KK^G$ survive the enlargement.
It also causes us to lay out some rather technical preliminaries, because a few results which are well-known in the compactly generated case are more complicated, or are plain wrong, for the $\aleph_1$-relative situation. 
We will collect most of this material straight from Lurie's work, with the exception of the cardinal-accounting comparison result in \Cref{Prop:Ind-as-localization} which appears to be new.

\begin{Rem}
In this appendix we restrict attention to \emph{countable discrete} groups~$G$, because (out of convenience) we want  to cite~\cite{BEL23pp} where this hypothesis is made. But we have no doubts that the general constructions and results presented here extend to second countable locally compact groups.
\end{Rem}

\subsection*{The language of $\infty$-categories}
We will freely use the basic terminology involving $\infty$-categories and their objects, morphisms, equivalences, (full) subcategories, functors, natural transformations, adjunctions, localizations etc., as used in \cite{LurieHTT}.
Limits and colimits in this context always mean homotopy (co)limits, and commutative diagrams are understood to commute up to (given) homotopy. 
Recall that every $\infty$-category $\cat C$ has a homotopy category $\Ho (\cat C)$, which is an ordinary category, and conversely every ordinary category can be viewed as a (`discrete') $\infty$-category via the nerve functor (which we omit from notations). 
Recall also that products and coproducts in an $\infty$-category are the same as products and coproducts in its homotopy category (in the ordinary sense), but that this is false for other kinds of (co)limits.
To avoid confusion later on, we write $\Top$ (rather than~$\cat S$ as Lurie) for the $\infty$-category of spaces (a.k.a.\ Kan complexes, $\infty$-groupoids, anima, homotopy types,\,\ldots).
We say an $\infty$-category is \emph{essentially small} if it is equivalent to a small one (\ie to one whose underlying simplicial set is small).

We also recall the following cardinal-related terminology from~\cite{LurieHTT}, because it is partially at odds with the similar-sounding terminology for triangulated categories that we also use:

\begin{Rec} Let $\kappa$ be an infinite regular cardinal. A simplicial set is \emph{$\kappa$-small} if its set of non-degenerate simplices is $\kappa$-small, \ie of cardinality less than~$\kappa$.
A (homotopy) limit or colimit in an $\infty$-category $\cat E$ is \emph{$\kappa$-small} if it is the (co)limit of some functor $p\colon K\to \cat E$ with $K$ a $\kappa$-small simplicial set, and it is \emph{small} if it is $\kappa$-small for some~$\kappa$. 
For $\kappa$ a regular cardinal, a colimit is \emph{$\kappa$-filtered} if it is the colimit of a diagram $p\colon K\to \cat E$ where $K$ is a (small) \emph{$\kappa$-filtered} simplicial set (\cite[Def.\,5.3.1.7 and Rem.\,5.3.1.11]{LurieHTT}). 
A functor $F\colon \cat E\to \cat E'$ is \emph{$\kappa$-continuous} if it preserves $\kappa$-filtered colimits, and \emph{$\kappa$-right exact} if it preserves $\kappa$-small colimits.
We say $\cat E$ is \emph{(co-)complete} if it admits arbitrary small (co-)limits. 
If  $\cat E$ is cocomplete, an object $C\in \cat E$ is \emph{$\kappa$-compact} if the representable functor $\cat E(C,-)\colon \cat E \to \Top$ is $\kappa$-continuous.
One says $\cat E$ is \emph{$\kappa$-compactly generated} if it is cocomplete, if its full subcategory 
\[
\cat E^\kappa
\]
of $\kappa$-compact objects is essentially small, and if every object $E\in \cat E$ is a $\kappa$-filtered colimit of $\kappa$-compact objects of~$\cat E$ (this is not the definition, but it is equivalent to it by \cite[Thm.\,5.5.1.1]{LurieHTT}). 
An $\infty$-category $\cat E$ is \emph{presentable} if it is $\kappa$-compactly generated for some regular~$\kappa$, and it is \emph{compactly generated} if we can take $\kappa = \aleph_0$ (denoted~$\omega$ by Lurie), the countable infinite cardinal.
Similarly, an object is \emph{compact} if it is $\aleph_0$-compact.
\end{Rec}

\subsection*{Stable $\infty$-categories}

Recall now from \cite[\S1]{LurieHA} that an $\infty$-category is \emph{stable} if it is pointed (\ie final and initial objects coincide), it admits all finite limits and colimits, and pushout and pullback squares in $\cat C$ coincide. It follows that finite limits and finite colimits of any shape are canonically equivalent. 
A functor of stable $\infty$-categories is \emph{exact} if it preserves all finite limits or, equivalently, finite colimits.
 If $\cat C$ is stable, its homotopy category $\Ho (\cat C)$ inherits a canonical structure of triangulated category, and exact functors between stable $\infty$-categories induce exact functors of triangulated categories.

In a stable $\infty$-category the existence of arbitrary ($\kappa$-small) colimits is equivalent to the existence of ($\kappa$-small) coproducts; and similarly, an exact functor between stable $\infty$-categories preserves the former iff it preserves the latter
 (see \cite[Prop.1.4.4.1]{LurieHA}). This is very useful in practice because, as mentioned, a diagram is a coproduct in an $\infty$-category $\cat C$ iff it is a coproduct in its homotopy category~$\Ho(\cat C)$. 
 
 In this vein, by a harmless slight abuse of language, we will be talking about `the thick/ localizing/ ideal subcategory of $\cat C$, when we mean the full ($\infty$-)subcategory of the $\infty$-category~$\cat C$ which is spanned by the same objects as the homonymous full subcategory of the triangulated category~$\Ho(\cat C)$. 
 We will also freely use the same notations $\Loc{-}$, $\tensLoc{-}$, etc.\ as introduced in the main body of the article. 

\subsection*{Ind-completion relative to a cardinal}

For any small $\infty$-category $\cat C$, let
\[ 
\PSh(\cat C) := \Fun(\cat C^\op, \Top)
\] 
be its presheaf $\infty$-category, that is the $\infty$-category of contravariant functors with values in the $\infty$-category of spaces.
It is a complete and cocomplete $\infty$-category.

Fix an infinite regular cardinal~$\kappa$. 
If $\cat C$ admits all $\kappa$-small colimits, we can consider as in \cite[\S5.3.5]{LurieHA} the full subcategory 
\[ \Indk(\cat C)\subset \PSh(\cat C) \]
of those presheaves which send $\kappa$-small colimits of $\cat C$ to (homotopy) limits in~$\Top$ (\cf the characterization in \cite[Cor.\,5.3.5.4]{LurieHTT}). 
The Yoneda embedding $\cat C\to \PSh(\cat C)$ restricts to a fully faithful functor 
\[
j=j_\cat C=j_{\cat C,\kappa}\colon \cat C\to \Indk(\cat C) .
\]
It turns out that an $\infty$-category is $\kappa$-compactly generated if and only if it is equivalent to one of the form $\Indk(\cat C)$ for some essentially small $\cat C$ with $\kappa$-small colimits. We now look more closely at how this works.

\begin{Prop}
\label{Prop:Indk}
Let $\cat C$ be a small (stable) $\infty$-category admitting $\kappa$-small colimits. 
Then $\Indk(\cat C)$ is a $\kappa$-compactly generated (stable) $\infty$-category, and the Yoneda embedding $j_\cat C\colon \cat C\to \Indk(\cat C)$ preserves $\kappa$-small colimits and displays the subcategory of $\kappa$-compact objects $(\Indk(\cat C))^\kappa$ as the idempotent completion of~$\cat C$.
In particular if $\kappa$ is uncountable (i.e.\ $\kappa \neq \aleph_0$) we get an equivalence $\cat C \overset{\sim}{\to}(\Indk(\cat C))^\kappa$.
\end{Prop}

\begin{proof}
If $\cat C$ is stable, then $\Indk(\cat C)$ is stable by \cite[Prop.\,1.1.3.6]{LurieHA}.

The Yoneda embedding $j_\cat C\colon \cat C \to \Indk(\cat C)$ preserves the $\kappa$-small colimits of $\cat C$ by \cite[Prop.\,5.3.5.14]{LurieHTT}, so in particular it is exact.
As $\cat C$ admits $\kappa$-small colimits, $\Indk(\cat C)$ is $\kappa$-compactly generated (see \cite[Thm.\,5.5.1.1]{LurieHTT} and its proof). 

By \cite[Prop.\,5.3.5.5]{LurieHTT} and \cite[Lem.\,5.4.2.4]{LurieHTT}, the fully faithful embedding $j_\cat C\colon C\to \Indk(\cat C)$ restricts to $\kappa$-compact objects, $j_\cat C\colon \cat C \to (\Indk(\cat C))^\kappa$, and the resulting functor is an idempotent completion in the sense of $\infty$-categories \cite[\S4.4.5]{LurieHTT} (which lifts the usual notion, at least in the stable case, in the sense that a stable $\infty$-category $\cat D$ is idempotent complete iff the additive category $\Ho \cat D$ is; see \cite[Lem.\,1.2.4.6]{LurieHA}).

For the last part, note that if $\kappa$ is uncountable then $\cat C$ admits arbitrary countable colimits, hence is already idempotent complete by \cite[Prop.\,4.4.5.15]{LurieHTT} (this argument already works at the level of triangulated categories by \cite[Prop.\,3.2]{BoekstedtNeeman93}).
\end{proof}

\begin{Prop}
\label{Prop:Indk-f}
Let $\cat C$ and $\cat D$ be two small $\infty$-categories admitting $\kappa$-small colimits, and let $f\colon \cat C\to \cat D$ be a $\kappa$-right exact functor (it preserves $\kappa$-small colimits).
Then there exists a functor $\Indk(f)$, unique up to equivalence, making the square
\[
\xymatrix{
\cat C \ar[d]_f \ar[r]^-{j_\cat C} & \Indk(\cat C) \ar[d]^{\Indk(f)} \\
\cat D \ar[r]^-{j_\cat D} & \Indk(\cat D)
}
\]
 (homotopy) commute.
This functor preserves arbitrary colimits, in fact it has a right adjoint given by precomposition with~$f$.
\end{Prop}

\begin{proof}
By \cite[Prop.\,5.3.5.10]{LurieHTT} precomposition with $j_\cat C$ induces, 
for every $\infty$-category $\cat E$ with arbitrary $\kappa$-filtered colimits,
 an equivalence  of $\infty$-categories
\begin{equation} \label{eq:UP-Indk}
j_\cat C^*\colon 
\Fun_\kappa (\Indk(\cat C), \cat E ) 
\overset{\sim}{\longrightarrow} 
\Fun(\cat C, \cat E)
\end{equation}
where $\Fun_\kappa$ on the left hand side denotes the full subcategory of $\Fun$ of all $\kappa$-continuous functors.
Setting $\cat E := \Indk(\cat D)$, we define $\Indk(f)$ to be a $\kappa$-continuous functor $\Indk(\cat C)\to \cat E$ corresponding to $j_\cat D\circ f\colon \cat C\to \cat E$ via the above equivalence.
Since $f\colon \cat C\to \cat D$ preserves $\kappa$-small colimits by hypothesis, $\Indk(f)$ has the claimed right adjoint, and in particular preserves colimits, by \cite[Prop.\,5.3.5.13]{LurieHTT}.
\end{proof}

\begin{Prop}
\label{Prop:Indk-adj}
In the situation of the previous proposition, suppose that $f$ admits a right adjoint $g\colon \cat D\to \cat C$ which also preserves $\kappa$-small colimits.
Then the adjunction $f\dashv g $ extends to an adjunction $\Indk(f)\dashv \Indk(g)$.
\end{Prop}

\begin{proof}
By \Cref{Prop:Indk-f} we get an adjunction $\Indk(f)\dashv f^*$, where $f^*$ denotes the precomposition functor $\Indk(\cat D)\to \Indk(\cat C)$, $P\mapsto P\circ f$.
Similarly, we can apply the proposition to $g$ in order to obtain an adjunction $\Indk(g)\dashv g^*$.
Moreover, the given adjunction $f\dashv g$ induces an adjunction $g^\op\dashv f^\op$ between the opposite categories and therefore, by whiskering its unit and counit maps, an adjunction $f^*\dashv g^*$ between functor categories. 
By uniqueness of the left adjoint of~$g^*$, we must have an equivalence $f^*\simeq \Indk(g)$. 
From the latter and the adjunction $\Indk(f)\dashv f^*$, we deduce an adjunction $\Indk(f)\dashv \Indk(g)$.
\end{proof}

\begin{Rem} \label{Rem:Indk-is-functorial}
The Ind-construction is functorial, in that it can be upgraded to a functor $\Indk$ from the $\infty$-category of small $\infty$-categories to that of large $\infty$-categories. 
In fact, if $\kappa$ is any uncountable regular cardinal, we can restrict both the functor's source and target so it becomes an equivalence of $\infty$-categories 
\[
\Indk\colon  \mathrm{Cat}_\infty^{Rex(\kappa)} \overset{\sim}{\longrightarrow} \mathrm{Pr}^L_\kappa
\]
from the subcategory $\mathrm{Cat}_\infty^{Rex(\kappa)}$ of small $\infty$-category with $\kappa$-small colimits and functors preserving them, to  $\mathrm{Pr}^L_\kappa$, the subcategory whose objects are $\kappa$-compactly generated $\infty$-categories and whose arrows are functors which preserve small colimits and $\kappa$-compact objects (see \cite[Prop.\,5.5.7.10]{LurieHTT}).
Note that each $\cat C \in \mathrm{Cat}_\infty^{Rex(\kappa)}$ is stable if and only if $\Indk(\cat C)$ is stable, by \Cref{Prop:Indk} and because the former is a full subcategory of the latter closed under finite limits; hence the above equivalence restricts to an equivalence
\[
\Indk\colon  \mathrm{Cat}_\infty^{Rex(\kappa), st} \overset{\sim}{\longrightarrow} \mathrm{Pr}^{L, st}_\kappa
\]
 between the full subcategories of stable $\infty$-categories on both sides.
\end{Rem}

\subsection*{Compactness in the homotopy category}

Before considering tensor structures, we wish to link $\kappa$-compactness with a similar notion in the realm of well generated triangulated categories in the sense of \cite{Neeman01} (see also~\cite{Krause10})). 
We record these somewhat technical results in view of future use, as they apply to the case of \emph{infinite} groups~$G$ (with $\kappa=\aleph_1$) and also because they may be of more general interest.

\begin{Lem}
\label{Lem:infty-vs-tri}
Let $\cat E$ be a stable $\infty$-category with arbitrary small colimits. Then:
\begin{enumerate}
\item An object in $\cat E$ is compact if and only if it is compact in the triangulated category $\Ho(\cat E)$, 
and $\cat E$ is compactly generated as an $\infty$-category if and only if $\Ho(\cat E)$ is compactly generated as a triangulated category.
\item 
Suppose $\cat E$ compactly generated. Then for any regular cardinal~$\kappa>\aleph_0$, an object of $\cat E$ is $\kappa$-compact in $\cat E$ in the sense of Lurie if and only if it is $\kappa$-compact in $\Ho (\cat E)$ in the sense of well generated categories.
Moreover $\cat E^\kappa = \Lock{\cat E^{\aleph_0}}$, that is, the $\kappa$-compact objects  of~$\cat E$ form the $\kappa$-localizing (\ie thick and closed under $\kappa$-small coproducts) subcategory of $\cat E$ generated by the compact objects.
\end{enumerate}
\end{Lem}

\begin{proof} 
Part (a) is contained in \cite[Prop.\,1.4.4.1]{LurieHTT}. 
Unfortunately this proof does not appear to admit a straightforward generalization to arbitrary cardinals.

Let us prove part~(b).
Suppose $\cat E$ is a compactly generated stable $\infty$-category. 
We can take the (essentially) small subcategory $\cat C:=\cat E^{\aleph_0}$ of its compact object as a generating set of compacts, so that $\cat E\simeq \Ind(\cat C)$.

We claim that~$\cat E^\kappa$, the full subcategory of all $\kappa$-compact objects in~$\cat E$, is equal to the closure of $\cat C$ in $\cat E$ under $\kappa$-small colimits. (This does not use stability and should be well-known; we provide a proof as we could not find a precise reference).

To prove the claim, denote the latter closure by $\overline{\cat  C}^\kappa\subset \cat E$.
Since $\kappa$-compact objects are closed under $\kappa$-small colimits by \cite[Cor.\,5.3.4.15]{LurieHTT} and since compact objects are in particular $\kappa$-compact, we immediately obtain the inclusion $\overline{\cat  C}^\kappa\subseteq \cat E^\kappa$. 
To prove the reverse inclusion $\cat E^\kappa \subseteq \overline{\cat  C}^\kappa$, we can proceed as follows (there may be more efficient ways).
Consider the inclusion functor $f \colon \overline{\cat  C}^\kappa\to \cat E$.
We want to use \cite[Prop.\,5.3.5.11(2)]{LurieHTT} to show that $f$ extends to an equivalence $F\colon \Indk(\overline{\cat  C}^\kappa)\overset{\sim}{\to} \cat E$. For this, it suffices to check that $f$ satisfies the following three hypotheses: 
\begin{enumerate}
\item $f$ is fully faithful (true by definition),
\item $f$ factors through $\cat E^\kappa$ (we already proved that $\overline{\cat  C}^\kappa\subseteq \cat E^\kappa$),
\item and finally, that the objects of $\overline{\cat  C}^\kappa$ generate $\cat E$ under $\kappa$-filtered colimits.
\end{enumerate}
Item (c) can be shown precisely as in the proof of \cite[5.3.5.12]{LurieHTT}, with $\PSh(\cat C)$ replaced by $\cat E = \Ind(\cat C)$ and $\PSh(\cat C)^\kappa$ replaced by~$\overline{\cat  C}^\kappa$.
Namely, every object $X\in \cat E$ is a small colimit (even a filtering colimit) of objects of~$\cat C$.
By the proof of \cite[4.2.3.11]{LurieHTT}, $X$ can be rewritten as a $\kappa$-filtered colimit in~$\cat E$ of a diagram whose values are objects which are $\kappa$-small colimits of objects in~$\overline{\cat  C}^\kappa$, that is, of a diagram with values in~$\overline{\cat  C}^\kappa$.
We therefore obtain the announced equivalence $F\colon \Indk(\overline{\cat  C}^\kappa)\overset{\sim}{\to} \cat E$, which restricts to an equivalence $\overline{\cat C}^\kappa\overset{\sim}{\to} \cat E^\kappa$ by \Cref{Prop:Indk} (since $\kappa>\aleph_0$), which in fact must then be an equality of full subcategories of~$\cat E$.

(Alternatively and slightly more informally, once we know that every object $X\in \cat E$ is the colimit of a $\kappa$-filtering diagram of $\kappa$-small colimits~$Y_\alpha$ of objects in~$\cat C$, we may consider the case of a $\kappa$-compact $X$ and the resulting equivalence $\cat E(X,X)\simeq \cat E(X,\colim_\alpha Y_\alpha)\simeq \colim_\alpha \cat E(X, Y_\alpha)$. Applied to the identity map of~$X$, this shows that $X$ is a retract of one of the $Y_\alpha$, \ie of a $\kappa$-small colimit of objects of~$\cat C$. As retracts are countable colimits and $\kappa>\aleph_0$, we can rewrite $X$ as a $\kappa$-small colimit of objects of~$\cat C$.)

Recall now that $\cat E$ is stable, which implies that the closure $\overline{\cat  C}^\kappa$ under $\kappa$-small colimits agrees with $\Lock{\cat C}$, the $\kappa$-localizing subcategory generated by $\cat C$ (where we close under $\kappa$-small coproducts and mapping cofibers), which can be constructed already in the triangulated category~$\Ho (\cat E)$.
By \cite[Cor.\,7.2.2]{Krause10}, the subcategory $\Lock{\cat C}\subseteq \Ho (\cat E)$ is also precisely the full subcategory of $\kappa$-compact objects in the compactly generated (\ie $\aleph_0$-well generated) triangulated category~$\Ho (\cat E)$.
This proves the moreover part, and in particular that the notions of $\kappa$-compact objects in the sense of $\infty$-categories and of well-generated categories agree at least inside of a \emph{compactly} generated category, as claimed.
 \end{proof}

\begin{Rem}
A triangulated category $\cat T$ with small coproducts is well generated if and only if it is generated by some object as a localizing subcategory. 
By combining this with \cite[1.4.4.2]{LurieHA}, one can deduce that a stable $\infty$-category $\cat E$ is presentable iff $\Ho(\cat E)$ is well generated.
By definition though, `presentable' means $\kappa$-compactly generated for some regular~$\kappa$; so it would be nice to show that $\cat E$ is $\kappa$-compactly generated iff $\Ho (\cat E)$ is $\kappa$-well generated (for the \emph{same}~$\kappa$!).\footnote{The equivalence with the same $\kappa$ is claimed without proof in \cite[\S1.3]{Efimov24pp}, but the conclusion there appears simply to follow from a non-standard use of the terminology for well generated triangulated categories. 
} 
This boils down to showing that part~(b) of the lemma holds in a general presentable~$\cat E$. 
Unfortunately, we can only prove one implication. Fortunately, it is the only one we will use:
\end{Rem}

\begin{Prop} \label{Prop:Ind-as-localization}
Let $\cat E$ be a stable $\infty$-category with small colimits. 
If $\cat E$ is $\kappa$-compactly generated for an infinite regular cardinal~$\kappa$, then $\Ho(\cat E)$ is a $\kappa$-well generated triangulated category.
Or equivalently: if $\cat C$ is any small stable $\infty$-category with $\kappa$-small colimits, then $\Ho(\Indk(\cat C)) $ is $\kappa$-well generated.
\end{Prop}

\begin{proof}
The equivalence of the two statements is because any $\kappa$-compactly generated $\cat E$ is equivalent to $\Indk(\cat E^\kappa)$ (as explained right after \cite[Def.\,5.5.7.1]{LurieHTT}), and moreover $\cat E\simeq \Indk(\cat E^\kappa)$ is stable iff $\cat E^\kappa$ is stable (one implication was already mentioned; for the other one, recall that $\cat E^\kappa$ is closed under $\kappa$-small colimits hence in particular under finite limits, hence is stable whenever $\cat E$ is). So we can take $\cat C=\cat E^\kappa$.

Accordingly, suppose that $\cat E= \Indk(\cat C)$ for $\cat C$ a small stable $\infty$-category with $\kappa$-small colimits.
Since finite colimits in $\cat C$ are $\kappa$-small, we have an inclusion functor 
\[
i\colon \Indk(\cat C)\to \Ind (\cat C)=:\cat D
\]
which by construction is fully faithful and also $\kappa$-continuous (the latter because $\Indk(\cat C)$, resp.\ $\Ind(\cat C)$, is the closure of $\cat C$ in the ambient $\infty$-category $\PSh(\cat C)$ under $\kappa$-filtered colimits, resp.\ all filtered colimits).
Recall also that $\cat D$ is compactly generated by \Cref{Prop:Indk}, hence $\Ho(\cat D)$ is compactly generated (that is $\aleph_0$-well generated) as a triangulated category by part~(a) of \Cref{Lem:infty-vs-tri}  (applied to~$\cat D$).

We claim that $\cat E$ is a Bousfield localization of~$\cat D$ by the localizing subcategory $\cat L= \Loc{\cat L_0}$ generated by an essentially small set $\cat L_0$ of $\kappa$-compact objects of~$\cat 
D$. 

Note that, by part~(b) of \Cref{Lem:infty-vs-tri}, the objects of $\cat L_0$ would then also be $\kappa$-compact in~$\Ho(\cat D)$ in the sense of well-generated triangulated categories.
Since $\Ho (\cat E) = \Ho(\cat D/\cat L) \simeq \Ho(\cat D)/ \Ho(\cat L)$ as triangulated categories, we can deduce from this together with \cite[Thm.\,7.2.1(2)]{Krause10} (choosing $\cat T:=\Ho(\cat D)$, $\cat S:=\Ho(\cat D)$, and $\alpha:=\kappa$ in \emph{loc.\,cit.}) that the quotient triangulated category $\Ho(\cat E)$ is also $\kappa$-well generated, as wished.

Hence it only remains to prove the above claim. 

For this, write $y\colon \cat C \to \Ind (\cat C) = \cat D$ and $j\colon \cat C\to \Indk(\cat C)$ for the two Yoneda embeddings, so that $i\circ j=y$, and let $\cat L_0:= \{ \cone(\varphi_p)\}_p$ be the collection of the mapping cones in $\cat D$ of the comparison maps 
\[
\varphi_p \colon \colim(y \circ p)\to  y (\colim p ) ,
\]
with $p\colon K\to \cat C$ running through all diagrams in $\cat C$ indexed by a $\kappa$-small simplicial set~$K$.
Note that $\cat L_0$ is essentially small as there is essentially only a set of the latter diagrams, and that its objects are $\kappa$-compact since $\cat D^\kappa$ contains the image of~$y$ (which consists of $\aleph_0$-compacts) and is closed under $\kappa$-small colimits in $\cat D$ (which also includes mapping cones). 
We still need to find an equivalence $\cat D/\cat L\simeq \cat E$ with $\cat L:=\Loc{\cat L_0}\subseteq \cat D$.

To this end, consider the following diagram
\[
\xymatrix@R=12pt{
& \cat D = \Ind(\cat C) \ar[rd]^-q \ar@<-1.5ex>@{..>}[dd]_-\ell & \\
\cat C \ar[ur]^-y \ar[dr]_-j  & & \cat D/ \cat L \ar@<-1.5ex>@{..>}[dl]_-{\overline \ell} \\
& \cat E = \Indk(\cat C) \ar[uu]_-i \ar@{..>}[ur]_-{\overline{q}} & 
}
\]
where $q$ is the Bousfield localization functor and $i\circ j=y$ by definition.
First note that $i$ admits a left adjoint~$\ell\colon \cat D\to \cat E$. 
This follows from \cite[Cor.\,5.5.2.9]{LurieHTT} and the $\kappa$-continuity of $i$ if we can show that $i$ preserves limits. Indeed, $i$ actually creates limits: if $P= \lim_\alpha P_\alpha$ is a limit in $\Ind(\cat C)$ of objects of $\Indk(\cat C)$, that is of $\kappa$-small limits-preserving functors $P_\alpha\colon \cat C^\op\to \Top$, then $P$ also preserves $\kappa$-small limits since limits commute with limits (the latter is a routine argument). 
Now, since $i$ is fully faithful the adjunction $\ell\dashv i$ is a localization and in particular the counit of adjunction is a natural equivalence 
\[ 
\ell \circ i\simeq \Id_\cat E.
\]
We also deduce an equivalence
\begin{equation} \label{eq:j=ly}
j\simeq \ell \circ y
\end{equation}
given by the composite $\ell \circ y = \ell \circ (i \circ j) \simeq (\ell \circ i )\circ j\simeq \Id_\cat E \circ j \simeq j$.

Next, note that $\ell$ descends along $q$ to an (exact and) colimits preserving functor $\overline \ell\colon \cat D/\cat L\to \cat E$.
Indeed, as $\ell$ is a left adjoint it is colimits preserving and therefore it suffices to show that $\ell$ sends the above-defined generating  maps~$\varphi_p$ to equivalences.
To see why the latter holds for each~$p$, consider the commutative diagram
\[
\xymatrix{
\colim (j \circ p) \ar[rr]^-\simeq \ar[d]_\simeq && j\big( \colim p \big) \ar[d]^\simeq \\
\ell \big( \colim y \circ p \big) \ar[rr]^-{\ell(\varphi_p)} && \ell \circ y \big(\colim p \big)
}
\]
of canonical maps in~$\cat D$; the right vertical map is by \eqref{eq:j=ly}, the left vertical equivalence is \eqref{eq:j=ly} combined with  the fact that $\ell$ preserves arbitrary colimits (it is a left adjoint), and the top equivalence because $j$ preserves $\kappa$-small colimits. It follows that $\ell(\varphi_p)$ is an equivalence too, as required.

We claim that, conversely, $q$ factors as $\overline{q}\circ \ell$ for a functor $\overline{q}\colon \cat D/\cat L\to \cat E$.
It then follows easily that $\overline{\ell}$ and $\overline{q}$ are mutually inverse equivalences, for instance by the universal property of (Dwyer--Kan) localizations recalled later on (see \eqref{eq:UP-DKloc} and \Cref{Rem:Bousfield-is-DK}).
Thus in order to conclude the proof of the proposition we only need to verify the latter claim.

To this end, note that as $\cat D/\cat L$ is cocomplete it has in particular all $\kappa$-filtering colimits, hence by the universal property \eqref{eq:UP-Indk} the composite $q\circ y$ factors as 
\[
q\circ y \simeq \overline{q} \circ j
\]
for a unique $\kappa$-continuous functor $\overline{q}\colon \cat E= \Indk(\cat C)\to \cat D/\cat L$.
In fact $\overline{q}$ preserves arbitrary small colimits; this follows from \cite[Prop.\,5.5.19]{LurieHTT} (applied to $f:=\overline{q}$) as soon as the following three conditions are verified:
\begin{enumerate}
\item the target of $\overline{q}$ is presentable and its source $\kappa$-accessible (true by hypothesis);
\item $\overline{q}$ is $\kappa$-continuous (true by construction);
\item the restriction of $\overline{q}$ to $\cat E^\kappa$ is fully faithful. 
\end{enumerate}
Only the verification of~(c) needs some work. 
We know that the restriciton of $\overline{q}$ to $\cat E^\kappa$ can be identified with $\overline{q} \circ j$, which is equivalent to $q\circ y$. 
To see that the latter is fully faithful, it suffices (from the general properties of Bousfield localizations) to show that the Yoneda functor $y$ takes values among $\cat L$-local objects; in other words, it suffices to show that $\varphi_p^*:=\Hom_\cat D(\varphi_p, y(C))$ is an equivalence for every generating map $\varphi_p$ and every $C\in \cat C$.
This is a consequence of the following commutative diagram of spaces
\[
\xymatrix{
\Hom_\cat D \big( y(\colim p) , y(C) \big)
  \ar[r]^-{\varphi_p^*} &
 \Hom_\cat D \big( \colim (y \circ p), y(C) \big) 
  \ar[d]^\simeq \\
& \lim \Hom_\cat D (y\circ p , y (C)) \\
\Hom_\cat C (\colim p, C) \ar[uu]^{y}_\simeq \ar[r]^-\simeq & 
\lim \Hom_\cat C (p, C)
 \ar[u]^{\lim \,y}_\simeq
}
\]
where all other maps are evident equivalences. Thus (c) holds and we can deduce that~$\overline{q}$ preserves small colimits and in particular all filtering ones.

It follows from the above that $q$ and $\overline{q}\circ \ell$ are two parallel continuous functors defined on $\Ind(\cat C)$, which moreover agree on $\cat C$ by the following chain of equivalences: 
\[
q\circ y \simeq \overline{q} \circ j \simeq \overline{q} \circ (\ell \circ y) \simeq (\overline{q} \circ \ell) \circ y.
\]
We derive from this and the universal property of the $\Ind$ construction (that is, \eqref{eq:UP-Indk} with $\kappa=\aleph_0$)  that $q \simeq \overline{q}\circ \ell$.
This provides the claimed factorization of~$q$ through~$\ell$, concluding the proof of the proposition.
\end{proof}

We will need the following consequence of the proof, which also uses a purely triangulated-categorical fact proved below (once again, there may be easier ways):

\begin{Cor}
\label{Cor:pres-cpts}
Let $j\colon \cat C\to \Indk(\cat C)$ be the Yoneda embedding for $\cat C$ a
small stable $\infty$-category with all $\kappa$-small colimits. 
Then a representable object $j(C)$ is compact in $\Indk(\cat C)$, provided that $\Hom_\cat C(C,-)\colon \cat C\to \Top$ preserves $\kappa$-small colimits.
\end{Cor}

\begin{proof}
Recall from the proof of \Cref{Prop:Ind-as-localization} the commutative diagram
\[
\xymatrix@R=12pt{
& \Ind(\cat C)  \ar[dd]_-\ell  \\
\cat C \ar[ur]^-y \ar[dr]_-j  &   \\
& \Indk(\cat C)   
}
\]
where $j$ and $y$ are the Yoneda embeddings, and where $\ell$ is a localization functor with fully faithful right adjoint~$i$ such that $i\circ j= y$, and whose full object-kernel is the localizing subcategory $\cat L=\Loc{\cone(\varphi_p)\mid p}$ generated by the cones of the canonical maps $\varphi_p \colon \colim(y \circ p)\to  y (\colim p )$, with $p\colon K\to \cat C$ ranging over $\kappa$-small diagrams in~$\cat C$.

We claim that, if $C\in \cat C$ is such that that $\Hom_\cat C(C,-)$ preserves $\kappa$-small colimits, then $j(C)$ lies in the left orthogonal ${}^\perp\cat L$ in~$\Ind (\cat C)$. 
Since the image of $y$ consists of compact objects of $\Ind(\cat C)$, the corollary would then immediately follow from \Cref{Prop:abstract-cpt-loc} below applied to the functor $L=\Ho(\ell)$.

To show that $y(C)\in {}^\perp\cat L$, it suffices to verify that the functor $\Hom_{\Ind (\cat C)} (yC,-)$ kills the generating objects $\cone(\varphi_p)$, or equivalently, that it inverts the maps~$\varphi_p$. 
Consider the following commutative diagram in~$\Top$   (using $i\circ j=y$):
\[
\xymatrix{
\Hom_{\Ind (\cat C)} \big( yC , \colim (y \circ p) \big)
  \ar[r]^-{(\varphi_p)_*} &
  \Hom_{\Ind (\cat C)} \big( yC , y( \colim p ) \big)   \\
{ \colim \Hom_{\Ind (\cat C)} ( yC , y \circ p ) }
  \ar[u]^\simeq & 
 \Hom_{\Ind_{\aleph_1}(\cat C)} \big( jC , j ( \colim p ) \big)
  \ar[u]^\simeq_{i} \\
{ \colim \Hom_{\Ind_{\aleph_1}(\cat C)} \big( j C , j\circ p \big) }
   \ar[u]^\simeq_{\colim i} 
    \ar@/_5ex/[ru]_-\simeq &
}
\]
The top left map is an equivalence because $y(C)$ is compact in $\Ind(\cat C)$, as already noted, hence the exact functor $\Hom(yC,-)$ preserves arbitrary colimits (by stability). 
The top right and bottom left maps are equivalences because $i$ is fully faithful.
The bottom curved arrow is an equivalence by the hypothesis on $C\in \cat C$. 
It follows that $(\varphi_p)_*$ is also an equivalence as claimed.
\end{proof}

We record the following purely triangular fact separately, as we could not find it in writing and it may be of independent interest (in case the right adjoint $J$ preserves coproducts, the same conclusion follows by \cite[Thm.\,5.1]{Neeman96}; but we use this precisely in a situation when $J$ does not preserve coproducts, by construction).

\begin{Prop}
\label{Prop:abstract-cpt-loc}
Let $L \colon \cat T \rightleftarrows \cat S \! :\!J$ be a Bousfield localization of triangulated categories, \ie an adjoint pair of exact functors with fully faithful right adjoint~$J$, and suppose that $\cat T$ (and therefore also~$\cat S$) admits arbitrary small coproducts.
If $C$ is a compact object of $\cat T$ belonging to the left orthogonal of the localizing subcategory $\Ker(L)$, then $L(C)$ is a compact object of~$\cat S$.
\end{Prop}

\begin{proof}
Recall that $J$ being fully faithful is equivalent to the counit of adjunction being an isomorphism $LJ\cong \Id_\cat S$. 
Also, the coproducts of $\cat S$ are created from those of $\cat T$ by the formula $\coprod^\cat S_i X_i = L(\coprod^\cat T_i JX_i)$. 
(To see the latter, let $\{X_i\}_i$ be a family of objects of~$\cat S$ and $Y\in \cat S$, and use $L\dashv J$ and the fully faithfulness of~$J$ to compute 
\[
\cat S \bigg( L\big (\coprod_i^\cat T JX_i \big) , Y \bigg)
\cong \cat T \big( \coprod_i^\cat T JX_i  , JY \big)
\cong \prod_i \cat T( JX_i, JY )
\cong \prod_i \cat S (X_i , Y),
\]
which shows that $L\big (\coprod_i^\cat T J(X_i) \big)$ is indeed a coproduct in $\cat S$ of the~$X_i$.)

Now consider the homological functor $\cat T(C,-)\colon \cat T\to \Ab$. 
By hypothesis it annihilates $\Ker(L)$, hence it admits a unique factorization through a homological functor on $\cat S$, say~$F$:
\[
\xymatrix{
\cat T \ar[d]_{\cat T(C,-)} \ar[r]^-L & \cat S \ar@{-->}[dl]^F \\
\Ab &
}
\]
As $C$ is also assumed to be compact, the functor $\cat T(C,-)$ preserves coproducts. 
Now for any small family $\{X_i\}_i$ of objects of~$\cat S$ we can compute
\begin{align*}
F\Big( \coprod_i X_i \Big) 
& = FL \Big( \coprod_i J X_i \Big) 
\cong \cat T \Big(C , \coprod_i J X_i \Big)
\cong \coprod_i \cat T(C, JX_i)   \\
& \cong \coprod_i FLJX_i 
 \cong \coprod_i FX_i 
\end{align*}
which shows that $F$ also preserves coproducts.
Hence to prove the proposition it suffices to show that $F \cong \cat S(LC, -)$ as functors on~$\cat S$.

Since localization functors are epic in the category of categories, it suffices to show that the two functors
$F\circ L= \cat T(C,-) $ and $\cat S(LC, - ) \circ L \cong \cat T (C, JL -)$ on $\cat T$ are isomorphic.
To show the latter, we claim that $\cat T(C, - )$ inverts the unit of adjunction  $\eta_X\colon X\to JLX$ for all $X\in \cat T$.

Indeed, recall that $L(\eta)$ is invertible (as for any Bousfield localization), so that $\Cone(\eta)\in \Ker(L)$. 
Hence $\cat T(C, \Cone(\eta))=0$, since $C\in {}^\perp \Ker(L)$ by hypothesis. 
This shows that $\cat T(C,-)$ inverts $\eta$ as claimed, concluding the proof.
\end{proof}

\subsection*{Tensor structures}

We now introduce tensor structures into the picture, that is symmetric monoidal structures  on $\infty$-categories, and on functors between them, in the sense of \cite[\S2]{LurieHA}. 
Most of this is well-known and widely used for the special case $\kappa=\aleph_0$, and the general case works in the same way (whence our brevity).

\begin{Prop} 
\label{Prop:Indk-tensor}
As before, let~$\kappa$ be an infinite regular cardinal.
\begin{enumerate}
\item
As in \Cref{Prop:Indk}, let $\cat C$ be a stable $\infty$-category with all $\kappa$-small colimits.
Suppose moreover that $\cat C$ is symmetric monoidal and that its tensor functor $\otimes$ preserves $\kappa$-small colimits in each variable separately.
Then the stable presentable $\Indk(\cat C)$ admits a symmetric monoidal structure, unique up to equivalence, characterized by the properties that it preserves arbitrary small colimits in both variables and that the Yoneda embedding $\cat C\to \Indk(\cat C)$ refines to a symmetric monoidal functor.
\item Let $\cat C$ and $\cat D$ have $\kappa$-small colimits and let $f\colon \cat C\to \cat D$ be a functor preserving them.
Suppose moreover that $\cat C$ and $\cat D$ are symmetric monoidal with $\otimes$-functor preserving $\kappa$-small colimits in both variables, and that the functor $f\colon \cat C\to \cat D$ is symmetric monoidal. 
Then the extension $\Indk(f)$ of \Cref{Prop:Indk-f} admits a compatible refinement to a symmetric monoidal functor $\Indk(\cat C)\to \Indk(\cat D)$, unique up to homotopy.
\end{enumerate}
\end{Prop}

\begin{proof}
This is a all a direct consequence of \cite[Prop.\,4.8.1.10 and Variant 4.8.1.11]{LurieHA}.
Indeed, the most well-known special case $\kappa = \aleph_0$ is explicitly stated as \cite[Cor.\,4.8.1.14]{LurieHA}.
For the more general case, it suffices to apply the result to the situation of \cite[Example 5.3.6.8]{LurieHTT}), that is, in the cited proposition we must choose $\cat K$ and~$\cat K'$ to be the collections of all $\kappa$-small simplicial sets and of all small simplicial sets, respectively.
\end{proof}

\subsection*{Ordinary equivariant KK-theory as an $\infty$-category}

Now we define the tensor $\infty$-categories and functors on which we wish to apply the above $\aleph_1$-relative Ind construction.
Conveniently, it will suffice to quote a few results from~\cite{BEL23pp}.

Given any (essentially) small $\infty$-category $\cat C$ and any set of arrows $W$ in~$\cat C$, one can construct the \emph{Dwyer--Kan localization} of $\cat C$ at~$W$ (\cite[Def.\,1.3.4.1]{LurieHA} \cite{Hinich16}), which is an $\infty$-category $\cat C[W^{-1}]$ equipped with a functor
\[
\cat C \to \cat C[W^{-1}]
\]
sending the arrows of $W$ to equivalences in the target $\infty$-category $\cat C[W^{-1}]$, and such that it induces an equivalence of functor categories 
\begin{equation} \label{eq:UP-DKloc}
\Fun(\cat C[W^{-1}], \cat D)\overset{\sim}{\to} \Fun_W (\cat C, \cat D) 
\end{equation}
for any $\infty$-category~$\cat D$; here $\Fun_W\subset \Fun$ denotes the full subcategory of those functors which map the elements of $W$ to equivalences. 
(Note that even when $\cat C$ is discrete---\ie an ordinary category---there is no reason for $\cat C[W^{-1}]$ to be discrete, in fact this is typically not the case.)

\begin{Rem} \label{Rem:Bousfield-is-DK}
Recall that a \emph{localization} (or \emph{Bousfield localization}) in the more common sense of Lurie~\cite{LurieHTT} is a functor $L\colon \cat C \to \cat D$ of $\infty$-categories which has a fully faithful right adjoint. Any such localization is also a Dwyer--Kan localization, where $W$ can be taken to be the set of all arrows sent by~$L$ to equivalences of~$\cat D$. 
\end{Rem}

\begin{Def} \label{Def:KKsep}
Let $G$ be a countable discrete group.
Let $\Cstarsep{G}$ be the  essentially small (ordinary) category of all separable $G$-C*-algebras and all equivariant *-homomophisms between them. Equip it with the usual symmetric monoidal structure given by the minimal tensor product of C*-algebras with diagonal group action.
For the following, we can view $\Cstarsep{G}$ as a symmetric monoidal discrete $\infty$-category via the nerve functor (omitted from notations). 
Let $W_G$ be the class of all $\KK^G$-equivalences, that is all arrows in $\Cstarsep{G}$ which are sent to isomorphisms in the ordinary Kasparov category~$\KK^G$.
We denote by 
\[
\kksep{G}\colon  \Cstarsep{G} \to \Cstarsep{G}[ W_G^{-1}]  =: \inftyKKsep{G}
\]
the Dwyer--Kan localization (as recalled above) of $\Cstarsep{G}$ by~$W_G$.
\end{Def}

\begin{Prop}
\label{Prop:KKsep}
The $\infty$-category $\inftyKKsep{G}$ is stable, and there is a canonical equivalence of triangulated categories 
\[
\Ho (\inftyKKsep{G}) \simeq \KK^G
\]
identifying the functors $\Cstarsep{G} \overset{\kksep{G}}{\longrightarrow} \inftyKKsep{G} \to \Ho (\inftyKKsep{G})$ and $\Cstarsep{G}\to \KK^G$.
Moreover, $\inftyKKsep{G}$ admits all countable colimits, and it admits an (essentially unique) tensor structure which preserves countable colimits in both variables and which allows $\kksep{G}$ to refine to a (strong) tensor functor.
\end{Prop}

\begin{proof}
The first sentence is the content of \cite[Thm.\,1.3]{BEL23pp}.
To see why $\inftyKKsep{G}$ has countable colimits, recall that $\KK^G$ admits all countable coproducts, given by the image under $\kksep{G}$ of countable direct sums of $G$-C*-algebras. 
It follows from the equivalence $\Ho (\inftyKKsep{G}) \simeq \KK^G$ that direct sums become countable coproducts also in the $\infty$-category~$\inftyKKsep{G}$.
 As the latter is stable it admits all finite colimits as well, and by combining the two, we see that it admits all countable colimits.
 
The claims on the tensor structure are precisely \cite[Prop.\,2.20]{BEL23pp}.
\end{proof}

Restriction functors between Kasparov categories can also be enhanced in a similar way:

\begin{Prop} 
\label{Prop:Fsep}
Let $\varphi\colon H\to G$ be any morphism between two countable discrete groups, and let $\Res_\varphi\colon \Cstarsep{G}\to \Cstarsep{H}$ be the tensor functor restricting the group actions along~$\varphi$.
Then there exists a (strong) tensor functor of $\infty$-categories
\[
\Res_\varphi \colon \inftyKKsep{G} \to \inftyKKsep{H}
\]
extending the above $\Res_\varphi$ along $\kksep{G}$ and~$\kksep{H}$ and preserving countable colimits.
On homotopy categories and under the identifications of~\Cref{Prop:KKsep}, it agrees with the usual restriction functor $\KK^G\to \KK^H$ of Kasparov theory.
\end{Prop}

\begin{proof}
The existence of an exact functor extending $\Res_\varphi\colon \Cstarsep{G}\to \Cstarsep{H}$ along the canonical functors is part of \cite[Lemma~4.6]{BEL23pp}; indeed this follows from $\Res_\varphi(W_G)\subseteq W_H$, and the extension is the essentially unique functor $\Res_\varphi$ which makes the square
\[
\xymatrix{
{\Cstarsep{G}} 
 \ar[r]^-{\kksep{G}}
 \ar[d]_{\Res_\varphi} &
  \inftyKKsep{G} \ar[d]^{\Res_\varphi} \\
\Cstarsep{H} \ar[r]^-{\kksep{H}} &
 \inftyKKsep{G}
}
\]
commute, obtained by the universal property of Dwyer--Kan localization (exactness essentially follows from the fact that the usual restriction functor on $\KK$-theory is exact).
The tensor structure on the extension is obtained similarly, using that the tensor functor $\kksep{G}$ has the analogous universal property with respect to tensor functors on $\Cstarsep{G}$ by \cite[Prop.\,3.2.2]{Hinich16} 
(\cf the proof of \cite[Cor.\,4.8]{BEL23pp}).

Since $\Res_\varphi$ preserves countable direct sums of separable $G$-C*-algebras, the extension preserves countable coproducts.
Exactness means that it preserves finite (co)limits as well, and we deduce that it preserves arbitrary countable colimits.
The agreement with the restriction functor between ordinary Kasparov categories is automatic, since both functors are extensions of the same functor $\Cstarsep{G}\to \Cstarsep{H}\to \KK^H$ along the canonical functor $\Cstarsep{G}\to \KK^G$, which enjoys a universal property among functors with target an additive category (see \cite[Thm.\,6.6]{Meyer00} or \cite[Prop.\,2.2]{BEL23pp}).
\end{proof}

\subsection*{Adjoining small colimits to KK-theory}
Now we combine all the above construction and compile the resulting properties, adding some extra results on induction and restriction functors.
(For future reference, we also record in \Cref{Thm:main-inftyKK}(c) the extension of the maximum crossed product functors, although they are not used in this article and hence the reader may ignore that part.)

\begin{Def} \label{Def:inftyKK}
For every countable groups~$G$, we write 
\[
\inftyKK{G}:=\Ind_{\aleph_1}(\inftyKKsep{G})
\]
for the $\aleph_1$-relative (or `countable') Ind-completion of the Bunke--Engel--Land enhancement of the Kasparov category $\KK^G$ recalled in \Cref{Def:KKsep}, and we write
\[
\xymatrix{
\kk{G} \colon \Cstarsep{G} \ar[r]^-{\kksep{G}} &
 \inftyKKsep{G} \ar[r]^-{j} &
   \inftyKK{G}
}
\]
for the resulting composite canonical functor.
Informally, $\inftyKK{G}$ is simply the universal $\infty$-category obtained from $\Cstarsep{G}$ by inverting the $\KK^G$-equivalences and adding all colimits in a way which preserves the existing countable colimits.
\end{Def}

\begin{Rem} \label{Rem:Ho-generation}
It follows immediately from \Cref{Prop:Ind-as-localization} that the homotopy category $\Ho(\inftyKK{G})$ is an $\aleph_1$-well generated triangulated category.
\end{Rem}

\begin{Thm} 
\label{Thm:main-inftyKK}
Retain \Cref{Def:inftyKK}. We have:
\begin{enumerate}
\item For $G$ any countable discrete group, $\inftyKK{G}$ is an $\aleph_1$-compactly generated stable $\infty$-category equipped with a symmetric monoidal structure which preserves small colimits in both variables.
The canonical functors induce a fully faithful and (countable) coproducts-preserving tt-functor on~$\KK^G$
\[
\xymatrix{
\Cstarsep{G}
  \ar[r]^-{\kk{G}}
   \ar[d] &
  \inftyKK{G}
   \ar[d]
      \\
\KK^G \ar@{..>}[r]^-{\exists !}& \Ho(\inftyKK{G}) 
}
\]
whose essential image is the homotopy category of $\aleph_1$-compact objects:
\begin{equation} \label{eq:agreement-Ho}
\KK^G \simeq \Ho(\inftyKKsep{G}) \simeq \Ho\!\big( (\inftyKK{G})^{\aleph_1} \big).
\end{equation}

\item If $\varphi\colon H\to G$ is any morphism between countable discrete groups, there is an essentially unique tensor functor 
$
\Res_\varphi\colon \inftyKK{G} \to \inftyKK{H}
$
making the following square of tensor functors commute:
\[
\xymatrix{
{\Cstarsep{G}} 
 \ar[r]^-{\kk{G}}
 \ar[d]_{\Res_\varphi} &
  \inftyKK{G} \ar[d]_{\Res_\varphi} \\
\Cstar{H} \ar[r]^-{\kk{H}} &
 \inftyKK{H} \ar@<-2ex>@{..>}[u]_{\varphi_*}
}
\]
The functor $\Res_\varphi$ admits a right adjoint which we denote~$\varphi_*$, so in particular it is exact and preserves all colimits. 
Under the identification \eqref{eq:agreement-Ho},  $\Res_\varphi$ agrees with the usual KK-theory restriction functor along~$\varphi$.

\item 
Suppose moreover that $\varphi$ is either 
\begin{enumerate}
\item the inclusion $H\hookrightarrow G$ of a subgroup, or
\item a projection morphism $H\twoheadrightarrow 1$ to the trivial group.
\end{enumerate}
Then the restriction functor $\Res_\varphi$ admits also a \emph{left} adjoint, which we may denote by~$\varphi_!$, which under the identification \eqref{eq:agreement-Ho} agrees with the usual induction functor $\Ind^G_H$ of KK-theory (in the first case) and with the usual maximal crossed product $H\ltimes-$ (in the second case).
Moreover, there is an equivalence of functors
\[ \varphi_!\simeq \varphi_*\]
between left and right adjoints when $G/H$ is finite in the first case and when $H$ is finite in the second case.

\item
Let $\Res^G_H:= \Res_\varphi$ be as in part~(c) for a subgroup inclusion $\varphi\colon H\hookrightarrow G$ of finite index. 
Then the adjunctions $\varphi_!\dashv  \Res^G_H\dashv \varphi_*$ and the isomorphism $\varphi_! \simeq \varphi_*$ can be chosen so that the resulting composite natural transformation
\[
\xymatrix{
\Id_{\inftyKK{H}} \ar[r]^-{\mathrm{unit}} & \Res^G_H\circ  \varphi_!  \simeq  \Res^G_H\circ \varphi_* \ar[r]^-{\mathrm{counit}} & \Id_{\inftyKK{H}}
}
\]
is equivalent to the identity.
\end{enumerate}
\end{Thm}

\begin{proof}
All claims of part~(a) follow immediately from  \Cref{Prop:Indk} and \Cref{Prop:Indk-tensor}~(a), applied with $\kappa = \aleph_1$ to the tensor $\infty$-category $\cat C:= \inftyKKsep{G}$; note that the latter satisfies all necessary hypotheses thanks to \Cref{Prop:KKsep}.

Similarly, all claims of part~(b) are directly obtained by applying \Cref{Prop:Indk-f} and \Cref{Prop:Indk-tensor}~(b) to the functor $\Res_\varphi\colon \inftyKKsep{G}\to \inftyKKsep{H}$ of \Cref{Prop:Fsep}.

Let us prove part~(c) and~(d) which, unlike the rest, are not completely formal and will need some more algebra-level observations.

In the two mentioned special cases of $\varphi\colon H\to G$, recall that the usual induction and (maximal) crossed product functors for KK-theory are left adjoint to $\Res_\varphi$ and are induced by algebra-level direct sums-preserving functors $\Ind_H^G\colon \Cstarsep{H}\to \Cstarsep{G}$ and $H\ltimes-\colon \Cstarsep{G}\to \Cstarsep{1}$ such that $\Ind^G_H(W_H)\subseteq W_G$ and $H\ltimes(W_H)\subseteq W_1$, respectively (indeed, note that in the second case $\Res_\varphi=\tau^H$ is the functor adjoining the trivial $H$-action to a plain C*-algebra).

Therefore, just as we did for~$\Res_\varphi$, the universal properties of Dyer--Kan localization and that of $\Ind_{\aleph_1}$ let us extend the latter functors to two (homonymous) colimit-preserving functors fitting into the following commutative diagrams:
\[
\xymatrix{
\Cstarsep{G} \ar[r]^-{\kksep{G}} &
 \inftyKKsep{G} \ar[r]^-j &
  \inftyKK{G} \\
\Cstarsep{H}
 \ar[r]^-{\kksep{H}}
  \ar[u]^{\Ind_H^G} &
 \inftyKKsep{H}
  \ar[r]^-{j}
   \ar[u]_{\Ind_H^G} &
 \inftyKK{H}
  \ar@{..>}[u]_{\Ind_H^G}
}
\quad\quad
\xymatrix{
\Cstarsep{1} \ar[r]^-{\kksep{1}} &
 \inftyKKsep{1} \ar[r]^-j &
  \inftyKK{1} \\
\Cstarsep{H}
 \ar[r]^-{\kksep{H}}
  \ar[u]^{H\ltimes-} &
 \inftyKKsep{H}
  \ar[r]^-{j}
   \ar[u]_{H\ltimes-} &
 \inftyKK{H}
  \ar@{..>}[u]_{H\ltimes-}
}
\]
We wish to lift the adjunctions $\Ind^G_H\dashv \Res^G_H$ and $H\ltimes(-)\dashv \tau^H$, as well as $\Res^G_H\dashv \Ind^G_H$ when $G/H$ is finite and $\tau^H\dashv H\ltimes(-)$ when $H$ is finite  (see \cite[\S4.1]{Meyer08}), which a priori are only defined at the level of homotopy categories (\ie of usual Kasparov categories), to their enhancements~$\inftyKKsep{(\ldots)}$. 
This is not automatic, but it suffices to show that either the unit or the counit of each KK-theoretic adjunction lifts to a natural transformation between the corresponding functors of $\infty$-categories; and for the latter, it suffices that this unit or counit is already defined as a natural transformation between the algebra-level functors (because then it extends to the $\infty$-categories by the universal property of Dwyer--Kan localization). 
And indeed, it is well-known that the units
\[ 
\Id_{\Cstarsep{H}}\to \Res^G_H\circ \Ind^G_H 
\quad
\textrm{ and }
\quad
\Id_{\Cstarsep{1}}\to H\ltimes \tau^H 
\]
and the counits
 \[
 H\ltimes \tau^H\to \Id_{\Cstarsep{1}}
\quad  \textrm{ and }
\quad
\Res^G_H \circ   \Ind^G_H \to \Id_{\Cstarsep{H}}
 \]
already exist at the algebra level in all relevant cases. This is mostly explained in \cite{MeyerNest06} and \cite{Meyer08}; for a complete reference, see respectively the maps (1.16), (1.18), (1.19) and (4.20) in~\cite{BEL23pp} (whose definitions are incidentally also shown to extend to non-necessarily separable C*-algebras), observing for the last map~(4.20) that when $G/H$ is finite the functor $\Coind^G_H$ of \cite[Constr.\,4.18]{BEL23pp} which underlies the right adjoint $\varphi_*$ evidently preserves separable C*-algebras, and is isomorphic already at the level of algebras to the functor~$\Ind^G_H$ underlying the left adjoint~$\varphi_!$ (indeed, the constructions of $\Ind^G_H$ and $\Coind^G_H$ are actually equal for $G/H$ finite).

The last observation allows us also to lift the equivalence between the left and the right adjoints of $\Res^G_H$ from KK-theory to the enrichments;
similarly, if $H$ is finite the left and right adjoints of $\tau^H$ agree on the enrichments because they are induced by the same algebra-level functor~$H\ltimes-$.

Finally, we wish to extend the previous adjunctions and functor isomorphisms from the level of the categories $\inftyKKsep{(\ldots)}$ to that of their $\Indk$-categories~$\inftyKK{(\ldots)}$. 
For this it suffices to recall from \Cref{Rem:Indk-is-functorial} that $\Indk(-)$ is functorial between the relevant $\infty$-categories of $\infty$-categories and thus preserves adjunctions and equivalences of functors.
This concludes the proof of part~(c).

For part~(d), recall that if $G/H$ is finite we already have an equality $\Ind^G_H= \Coind^G_H$ between the algebra-level functors constructed in \cite{BEL23pp}, and one can verify immediately that the composite of the (algebra level!) unit \cite[(1.16)]{BEL23pp} followed by the counit \cite[(4.20)]{BEL23pp} is the identity map (\cf also \cite[Lem.\,4.2 and Rem.\,2.10]{BDS15} for what amounts to the same observation). 
This relation still holds between the extended natural transformations between the extended functors at the level of the categories $\inftyKK{(\ldots)}$, by the universal properties of Dwyer--Kan localization and of the $\Indk$-construction. 
\end{proof}

\subsection*{Cell algebras and finite groups}

We now restrict the above constructions and results to cell C*-algebras and then further to finite groups, as used in the main body of this article.

\begin{Def}
For $G$ any countable discrete group, let 
\[
\cat G_G := \{ \Cont_0(G/H) : H\leq G \}
\]
be the (finite or countable) set of standard orbit algebras $\Cont_0(G/H) \in \Cstarsep{G}$ for all subgroups $H\leq G$, and let 
\[
\inftyCellsep{G}:= \Locc{\kksep{G} (\cat G_G) } \subset \inftyKKsep{G} 
\]
and
\[
\inftyCell{G}:= \Loc{\kk{G} (\cat G_G)} \subset \inftyKK{G} 
\]
be the $\aleph_1$-localizing subcategory of $\inftyKKsep{G}$, resp.\ the localizing subcategory of $\inftyKK{G}$, generated by its canonical image. 
We collect a few properties of these $\infty$-categories:
\end{Def}

\begin{Cor} \label{Cor:inftyCell-generators}
With the above notations, we have for any countable group~$G$:
\begin{enumerate}
\item
$\inftyCellsep{G}$ and $\inftyCell{G}$ are stable tensor subcategories of $\inftyKKsep{G}$ and $\inftyKK{G}$, respectively. 
Moreover $\inftyCell{G}$ is $\aleph_1$-compactly generated and $\inftyCellsep{G}$ identifies with its full tensor subcategory of $\aleph_1$-compact objects: 
\[ 
\inftyCellsep{G}\simeq \inftyCell{G}^{\aleph_1} .
\]
\item There is a canonical equivalence of tt-categories
\[
\Cell{G} \simeq \Ho(\inftyCellsep{G}),
\]
where $\Cell{G} = \Locc{\Cont_0(G/H): H\leq G}\subset \KK^G$ as in \Cref{def.CellCat}.
\item 
If $H$ is finite, the generator $\kk{G}(\Cont_0(G/H))$ is a compact object in~$\inftyKK{G}$.
\item 
If $G$ is a finite group (so that $\Cont_0(G/H)=\Cont(G/H)$ for all subgroups), $\inftyCell{G}$ is a compactly generated stable $\infty$-category, with
$
\cat G_G 
$
as a finite set of compact generators.
\item  
For any subgroup $H\leq G$, the restriction and induction functors $\Res^G_H$ and $\Ind^G_H$ restrict to functors between $\inftyCell{G}$ and $\inftyCell{H}$, satisfying the same adjunctions and isomorphisms as in \Cref{Thm:main-inftyKK}~(c) and~(d).
\end{enumerate}
\end{Cor}

\begin{proof}
In the following, for the sake of simplicity we identify $\inftyKKsep{G}$ with the full subcategory of $\aleph_1$-compacts in $\inftyKK{G}$, as per \Cref{Thm:main-inftyKK}(a), suppressing the Yoneda functor as well as the functors $\kksep{G}$ and $\kk{G}$ from notations. 

Part~(a): 
We clearly have $\inftyCell{G}=\Loc{\cat G_G} = \Loc{\inftyCellsep{G}}$ inside $\inftyKK{G}$, and write $\cat C:= \inftyCellsep{G}$ for short. 
Note that $\cat C$ admits all $\aleph_1$-small colimits, so in particular it is a stable subcategory, 
and similarly $\inftyCell{G}$ is by construction closed under small colimits and is stable; moreover, the latter is $\aleph_1$-presentable as it is generated under colimits by the set $\cat G_G$ of $\aleph_1$-compact objects of~$\inftyKK{G}$.
We claim that the inclusion $\cat C\to \Loc{\cat C}$ extends to an equivalence $\Ind_{\aleph_1}(\cat C)\overset{\sim}{\to} \Loc{\cat C}$. 
This follows from \cite[Prop.\,5.3.5.11(2)]{LurieHTT} (as in the proof of \Cref{Lem:infty-vs-tri}), because the inclusion is fully faithful and its image consists of $\aleph_1$-compact objects in $\Loc{\cat C}$ (since they are $\aleph_1$-compact in the ambient $\infty$-category~$\inftyKK{G}$) which collectively generate $\Loc{\cat C}$ under $\aleph_1$-filtering colimits (since they are already closed under $\aleph_1$-small colimits). 
It follows in particular that $\cat C=\inftyCellsep{G} $ identifies with the $\aleph_1$-compact objects of $\Ind_{\aleph_1}(\cat C) \simeq \Loc{\cat C} = \inftyCell{G}$, as claimed in part~(a).
To conclude this part, we easily deduce from the Mackey formula 
$
\Cont_0(G/H)\otimes \Cont_0(G/K)\simeq \bigoplus_{[x]\in H\backslash G/K} \Cont_0(G/H\cap {}^xK)
$
that $\inftyCellsep{G}$ and $\inftyCell{G}$ are symmetric monoidal full subcategories.

Part~(b) is immediate from the identification \eqref{eq:agreement-Ho}.

For part~(c), we begin by showing that the trivial $H$-algebra $\mathbb C$ is compact in $\inftyKK{H}$ when $H$ is finite. 
Recall that for $H$ finite $\mathbb C$ is compact${}_{\aleph_1}$ in $\KK^H \simeq \Ho(\inftyKKsep{H})$, since it corepresents the ordinary $H$-equivariant topological K-theory functor. 
In particular, the functor $\Hom(\mathbb C,-)\colon \inftyKKsep{H}\to \Top$ preserves countable coproducts.
By stability, this is the same as preserving arbitrary countable colimits. 
It now follows by  \Cref{Cor:pres-cpts}  (with $\kappa=\aleph_1$) that the canonical embedding $\inftyKKsep{G} \to \Ind_{\aleph_1}(\inftyKKsep{G})=\inftyKK{G}$ sends $\mathbb C$ to a compact object of~$\inftyKK{H}$. 
Since $\mathbb C= \Cont_0(G/G)$ belongs to $\inftyCell{H}$, it is \emph{a~fortiori} compact in the latter localizing category, as claimed.
We deduce from this that $\Cont_0(G/H)\simeq \Ind^G_H(\mathbb C)$ is compact in $\inftyCell{G}$, since by \cite[Lem.\,5.5.1.4]{LurieHTT} the induction functor $\Ind^G_H$ preserves compact objects because its right adjoint $\Res^G_H$ is continuous (since it admits a right adjoint by \Cref{Thm:main-inftyKK}(b)).
This proves part~(c).

Part~(d) is now immediate: If $G$ is finite so are all its subgroups, hence $\cat G_G$ is a set of compact generators of $\inftyCell{G}$ by part~(c).

Part~(e) is proved precisely as the analogue statement for ordinary KK-theory, see \cite[Prop.\,2.9]{DellAmbrogio14}, exploiting the elementary fact that the (ordinary) $\Res^G_H$ and $\Ind^G_H$ functors send orbit algebras to orbit algebras or direct sums thereof.
\end{proof}

Finally, the homotopy category $\bigCell{G}$ of the $\infty$-category $\inftyCell{G}$ will provide us with the announced enlargement of $\Cell{G}$ with all the nice properties:

\begin{Thm}
\label{Thm:bigCell}
\begin{enumerate}[\rm(1)]
\item
For every countable discrete group~$G$, there exists a tensor triangulated category $\bigCell{G}$ equipped with a functor 
\[ \iota_G \colon \Cell{G} \to \bigCell{G} \]
satisfying the following properties:
\begin{enumerate}[\rm(a)]
\item $\bigCell{G}$ is an $\aleph_1$-well-generated triangulated category in the sense of Neeman (see \cite{Neeman01} \cite{Krause10}), with the standard orbit algebras $\cat G_G$ as a countable set of $\aleph_1$-compact generators. Moreover, its tensor product is exact and preserves small coproducts in both variables.
\item
The functor $\iota_G$ is symmetric monoidal, exact, fully faithful, and preserves all countable coproducts.
\item
When $G$ is finite, $\bigCell{G}$ is a rigidly-compactly generated tt-category and $\iota_G$ restricts to an equivalence of tt-categories
\[
\iota_G \colon \Cell{G}^c \overset{\sim}{\to} (\bigCell{G})^c
\]
between the rigid-compact${}_{\aleph_1}$ objects of $\Cell{G}$ and the rigid-compact objects of $\bigCell{G}$.
\end{enumerate}

\item
The standard functoriality in $G$ of $\Cell{G}$ (deduced from that of $G\mapsto \KK^G$ by restricting functors) extends as follows:
\begin{enumerate}[\rm(a)]
\item \emph{(Restriction.)} If $\varphi\colon G\to H$ is any group homomorphism of discrete countable groups, there exists an exact, coproduct-preserving and symmetric mon\-oidal restriction functor 
\[ \Res_\varphi\colon \bigCell{H}\to \bigCell{G} \]
extending, up to isomorphism, the usual functor $\Cell{H}\to \Cell{G}$ given by restriction along~$\varphi$. 
\item \emph{(Induction.)}
When $\varphi\colon G\to H$ is the inclusion of a subgroup, the functor $\Res_G^H:=\Res_\varphi$ of part (a) admits a left adjoint, the \emph{induction functor} 
\[
 \Ind_G^H \colon \bigCell{G} \to \bigCell{H}
\]
which extends, up to isomorphism, the usual induction functor of equivariant KK-theory. 
When $G/H$ is finite, $\Ind^G_H$ is also right adjoint to $\Res_G^H$, and we can choose the two adjunctions so that the composite
\[
\Id \longrightarrow \Res^G_H\Ind^G_H  \longrightarrow \Id
\]
of the unit of the adjunction $\Ind^G_H\dashv\Res^G_H$ followed by the counit of the adjunction $\Res^G_H\dashv \Ind^G_H$ is the identity natural transformation of $\Id_{\bigCell{H}}$.
\end{enumerate}
\end{enumerate}
\end{Thm}

\begin{proof}
For (1) we set $\bigCell{G}:= \Ho(\inftyCell{G})$ as planned.
Since the $\infty$-category $\inftyCell{G}$ is $\aleph_1$-compactly generated (\Cref{Cor:inftyCell-generators}(a)), the triangulated category $\bigCell{G}$ is $\aleph_1$-well generated by \Cref{Prop:Ind-as-localization}.
The properties of the tensor product induced on the homotopy category are immediate consequences of the tensor product on $\inftyCell{G}\subseteq \inftyKK{G}$ preserving small colimits in both variables (\Cref{Thm:main-inftyKK}(a)). 
The functor $\iota_G$ and its properties as in~(1)(b) are immediately obtained, via the identification in \Cref{Cor:inftyCell-generators}(b), by the tt-functor induced on homotopy categories from the inclusion $\inftyCellsep{G}\to \inftyCell{G}$.

By \Cref{Cor:inftyCell-generators}(d), when $G$ is finite $\inftyCell{G}$ is a compactly generated $\infty$-category, hence $\bigCell{G}$ is compactly generated as a triangulated category by \Cref{Lem:infty-vs-tri}(a); moreover the images under $\iota_G$ of the standard orbit algebras $\cat G_G$ provide a set of compact generators. 
Since $\cat G_G$ generates $\Cell{G}$ as a compactly${}_{\aleph_1}$ generated triangulated category, we conclude that $\iota_G$ restricts to an equivalence
$(\Cell{G})^c=\Thick{\cat G_G} \overset{\sim}{\to}\Thick{\iota_G (\cat G_G)}= (\bigCell{G})^c$.
Recall also that for $G$ finite each orbit algebra $\Cont(G/H)$ is rigid in the tt-category $\Cell{G}$, hence also in $\bigCell{G}$ since tensor functors preserve duals.
This proves part~(1)(c).

For part~(2), it suffices to appeal to \Cref{Cor:inftyCell-generators}(e).
\end{proof}

\section{Comparing big and countable stratifications}
\label{sec:app-comparison}%

In this appendix we show how to deduce countable stratification from the `usual' stratification in settings where both stratification statements  (via usual or countable Balmer--Favi supports) make sense. 
Note that our proof is purely tensor-triangulated, in particular it does not require any $\infty$-categorical models.

\begin{Prop}
\label{Prop:strat-restriction}
Let $\iota\colon \cat T \to \bigT$ be a functor of tensor-triangulated categories having the following properties:
\begin{enumerate}[\rm(1)]
\item $\bigT$ is rigidly-compactly generated (\Cref{Def:big-tt-cat}).
\item $\cat T$ is rigidly-compactly${}_{\aleph_1}$ generated (\Cref{Def:cbly-big-ttcat}). 
\item The functor $\iota$ is fully faithful, exact, symmetric monoidal and preserves all countable coproducts. 
\item The restriction of $\iota$ to rigid (\ie compact${}_{\aleph_1}$) objects of $\cat T$ and rigid (\ie compact) objects of $\bigT$ is an equivalence of tt-categories: $\iota\colon \cat T^c \overset{\sim}{\to} \bigT^c$.
\item The topological space $\Spc (\cat T^c) \cong \Spc (\bigT^c)$ is weakly noetherian.
\end{enumerate}
Then if the big tt-category $\bigT$ is stratified in the sense of \cite{BHS23}, it follows that $\cat T$ is stratified in the sense of \Cref{Cons:countable-strat}. 
More precisely, if $\bigT$ is stratified then we have a commutative diagram of bijections
\begin{equation} \label{eq:bijections-square}
\vcenter{
\xymatrix@C-5pt{
\{\textrm{localizing${}_{\aleph_1}$\! $\otimes$-ideals of }\cat T\}
 \ar[d]_\cong
 \ar@<2pt>[rr]^-{\tensLoc{\iota (-)}}&&  
 \{ \textrm{localizing $\otimes$-ideals of }\bigT \} 
  \ar[d]^\cong
 \ar@<2pt>[ll]^-{\iota^{-1}}  \\
\{\textrm{subsets of } \Spc(\cat T^c)\} \ar@{->}[rr]^{\Spc(\iota)^{-1}}_-\cong^-{} &&
  \{\textrm{subsets of } \Spc ( \bigT^c)\}
}
}
\end{equation}
where the vertical ones are the stratification bijections $\cat L \mapsto \supp \cat L$ via the Balmer--Favi supports and their countable version (\Cref{Cons:countable-strat}), and the bottom horizontal one is via the homeomorphism $\Spc(\iota)= \iota^{-1}\colon  \Spc ( \bigT^c) \overset{\sim}{\to} \Spc(\cat T^c)$.
\end{Prop}

\begin{Rem}
By replacing $\cat T$ with its image $\iota (\cat T)$ in~$\bigT$, we could assume that $\iota$ and the induced homeomorphism $\iota^{-1}$ are identity maps.
Nonetheless, we will preserve both notations $\iota$ and $\iota^{-1}$ throughout the proof in order to better keep track of the two ambient categories $\cat T$ and~$\bigT$ for our various constructions.
\end{Rem}

\begin{Exa}
Suppose $\cat S$ is a rigidly-compactly generated tt-category such that its subcategory $\cat S^c$ of rigid-compacts admits a countable skeleton (one with countably many objects and arrows). 
Let $\cat T$ be the localizing${}_{\aleph_1}$ subcategory of $\cat S$ generated by~$\cat S^c$.
Then the inclusion functor $\iota\colon \cat T\to \cat S=:\bigT $ satisfies all the hypotheses of the proposition. 
Indeed, the only non-evident part of the claim is whether $\cat T(C,X)$ is countable for every $C\in \cat T^c$ and $X\in \cat T$ (in part~(2)), but this can be verified by an inductive argument starting with the countability assumption on~$\cat S^c= \cat T^c$.
\end{Exa}

\begin{proof}[Proof of \Cref{Prop:strat-restriction}]
By hypothesis $\bigT$ is stratified, \ie we have the right vertical bijection in the diagram~\eqref{eq:bijections-square}.
We also have the bottom bijection, as already noted in the statement.
To prove the proposition, we are first going to show that the square~\eqref{eq:bijections-square} (with the right-going top arrow) is commutative and then we will show that the top two arrows are mutually inverses; this will provide the claimed left vertical bijection. 

We also assumed that $\cat T$ is rigidly-compactly${}_{\aleph_1}$ generated.
In particular, for every tt-prime $\cat P\in \Spec (\bigT^c)$ we may construct the countable version  $g(\iota^{-1}\cat P) \in \cat T$ of the Balmer--Favi idempotent for the corresponding tt-prime $\iota^{-1}\cat P$ of~$\cat T^c$. 

We need to compare the countable and the ordinary Balmer--Favi supports:

\begin{Lem}
\label{Lem:isomorphic-g's}
We have $\iota (g (\iota^{-1}\cat P)) \cong g (\cat P)$ for every prime $\cat P \in \Spc (\bigT^c)$. 
\end{Lem}

\begin{proof}
Consider the idempotent triangle $e_Y \to \unit \to f_Y\to \Sigma e_Y$ of $\cat T$ associated to a Thomason subset $Y \subseteq \Spc (\cat T^c)$. 
After applying the tt-functor $\iota$ we obtain an idempotent triangle 
\[
\iota e_Y \longrightarrow \unit \longrightarrow \iota f_Y \longrightarrow \Sigma \iota e_Y
\]
in~$\bigT$.
On the other hand, we also have the idempotent triangle 
\[
e_{\tilde Y} \longrightarrow \unit \longrightarrow f_{\tilde Y} \longrightarrow \Sigma e_{\tilde Y}
\]
of $\bigT$ associated with the Thomason subset $\tilde Y :=\Spc(\iota)^{-1}(Y)$ of $\Spc (\bigT^c)$.

We claim that these two triangles are isomorphic.
For this, it will suffice to show that the (full) images of the two left idempotents agree,
\[ 
\iota e_Y \otimes \bigT = e_{\tilde Y} \otimes \bigT,
\]
 because the latter localizing subcategories of $\bigT$ determine the idempotent triangles up to a (unique) isomorphism.
By construction we have $e_{Y}\otimes \cat T= \Locc{\cat C_Y}$ in $\cat T$ and $e_{\tilde Y} \otimes \bigT = \Loc{\cat C_{\tilde Y}}$ in~$\bigT$, where $\cat C_Y \subseteq \cat T^c$ and $\cat C_{\tilde Y} \subseteq \bigT^c$ are the thick tensor ideals of compacts supported on~$Y$ and~$\tilde Y$, respectively.
Thus we need to show that 
\[
\iota e_Y \otimes \bigT = \Loc{ \cat C_{\tilde Y}} \quad \textrm{ in }\bigT.
\]
Note that $\iota (\cat C_Y) = \cat C_{\tilde Y}$ since $\iota$ identifies compacts by hypothesis~(4).

To prove the inclusion~$\subseteq$, note that $\iota e_Y \in \iota (\Locc{\cat C_Y}) \subseteq \Loc {\iota \cat C_Y} = \Loc {\cat C_{\tilde Y}}$ and thus, since the latter is a localizing tensor ideal, also $\iota e_Y \otimes \bigT \subseteq \Loc {\cat C_{\tilde Y}}$.

For the reverse inclusion~$\supseteq$, use $\iota (\cat C_Y) = \cat C_{\tilde Y}$ again and then apply the coproducts preserving tt-functor $\iota$ to the equality $\Locc{\cat C_Y} = e_{Y} \otimes \cat T$ to get the inclusion
\[
\cat C_{\tilde Y}
\subseteq \iota (\Locc{\cat C_Y}) 
= \iota e_Y \otimes \iota (\cat T) 
\subseteq \iota e_Y \otimes \bigT 
\]
and therefore, since the latter is a localizing subcategory, also 
$
\Loc{\cat C_{\tilde Y}} 
\subseteq \iota e_Y \otimes \bigT^c .
$
This proves the claimed equality of localizing subcategories and therefore of the two idempotent triangles.

Now to prove the lemma, choose Thomason subsets $Y_1,Y_2$ in $\Spc( \cat T^c)$ such that $\{\iota^{-1} \cat P\} = Y_1 \cap Y_2^c$ and deduce from the above observation that 
\[ 
\iota ( g(\iota^{-1} \cat P) )
= \iota ( e_{Y_1} \otimes f_{Y_2} ) 
= \iota (e_{Y_1}) \otimes \iota (f_{Y_2}) 
\cong e_{\tilde Y_1} \otimes f_{\tilde Y_2}
= g(\cat P)
\]
as claimed, since $\{\cat P\} = \tilde Y_1 \cap \tilde Y_2^c$ (with notations $\tilde Y_i = \Spc(\iota)^{-1}(Y_i) $ as before).
\end{proof}

\begin{Cor}
\label{Cor:comparison-of-supports}
$\supp (\iota A)  = \Spc(\iota)^{-1} (\supp (A))$ for every object $A\in \cat T$.
\end{Cor}

\begin{proof}
For every prime $\cat P\in \Spc (\bigT^c)$ and object $A\in \cat T$, we have
$\cat P\in \supp (\iota A)$ 
$\Leftrightarrow$
$0\neq \iota(A)\otimes g(\cat P) \cong \iota (A) \otimes \iota g(\iota^{-1}\cat P) = \iota (A \otimes g(\iota^{-1} \cat P)) $ (by \Cref{Lem:isomorphic-g's}) 
$\Leftrightarrow$
$A \otimes g(\iota^{-1}\cat P) \neq 0$ (since $\iota$ is conservative)
$\Leftrightarrow$
$\iota^{-1}\cat P\in \supp (A)$.
\end{proof}

\begin{Lem}
\label{Lem:square-commutes}
The square 
$\vcenter{ \xymatrix@1@C=12pt@R=12pt{ \bullet \ar[r] \ar[d] & \bullet \ar[d] \\  \bullet \ar[r] & \bullet }}$ 
in \eqref{eq:bijections-square} is commutative.
\end{Lem}

\begin{proof}
Let $\cat L $ be an $\aleph_1$-localizing tensor ideal of~$\cat T$. 
Mapping it down-then-right in \eqref{eq:bijections-square} yields $\Spc(\iota)^{-1}(\supp (\cat L))$, the image in $\Spc (\bigT^c)$ of its countable support $\supp (\cat L) $.
Mapping it right-then-down yields $\supp (\tensLoc{\iota \cat L})$, the Balmer--Favi support of the localizing tensor-ideal generated by $\cat L$ in~$\bigT$. 
But \Cref{Cor:comparison-of-supports} implies that the two are equal:
\begin{align*}
\Spc(\iota)^{-1}  (\supp (\cat L))  
&= \bigcup_{A\in \cat L} \Spc(\iota)^{-1} (\supp (A))\overset{\eqref{Cor:comparison-of-supports}}{=} \bigcup_{A\in \cat L}  \supp (\iota A)\\
&= \supp (\iota \cat L)
= \supp (\tensLoc{\iota \cat L})
\end{align*}
where the last equality is an immediate consequence of the elementary properties of the Balmer--Favi support (see \cite[Rem.\,2.12]{BHS23}).
\end{proof}

\begin{Lem}
\label{Lem:technical-tensor-generation}
$\tensLoc{ \cat{F} } = \Loc{ \cat{F} \otimes \iota (\cat T^c) }$ for  any class of objects $\cat{F} \subseteq \bigT$.
\end{Lem}

\begin{proof}
We are going to verify the following chain of equalities:
\[
\tensLoc{ \cat{F} }
= \Loc{ \cat{F} \otimes \bigT }
= \Loc{ \cat{F} \otimes \bigT^c}
= \Loc{ \cat{F} \otimes \iota (\cat T^c)}.
\]
The last one is obvious since we have $\iota (\cat T^c)= \bigT^c$ by hypothesis~(4).

The middle one follows from the hypothesis that $\bigT$ is compactly generated, in part~(1), which in particular implies that $\bigT = \Loc{\bigT^c}$ and therefore $\cat{F} \otimes \bigT = \cat{F} \otimes \Loc{\bigT^c} \subseteq \Loc{ \cat{F} \otimes \bigT^c}$; whence we get $\Loc{\cat{F} \otimes \bigT} \subseteq \Loc{ \cat{F} \otimes \bigT^c}$, and the opposite inclusion is obvious.

For the first equality, the inclusion $\supseteq$ is obvious. 
To see~$\subseteq$, since $\Loc{ \cat{F} \otimes \bigT }$ is a localizing subcategory of $\bigT$ containing~$\cat{F}$, it will suffice to show that it is also a tensor ideal. 
To this end, fix an arbitrary object $X\in \bigT$ and consider the full subcategory $\cat{S}_X:= \{ Y  \in \bigT \mid Y \otimes X \in \Loc{ \cat{F} \otimes \bigT } \}$; we must show that $\cat{S}_X$ contains $\Loc{ \cat{F} \otimes \bigT }$. 
Since the functor $-\otimes X$ is exact and commutes with coproducts, $\cat{S}_X$ is localizing.
Since $\cat{F} \otimes \bigT \otimes X = \cat{F} \otimes \bigT \subseteq \Loc{\cat{F} \otimes \bigT}$, it contains $\cat{F} \otimes \bigT$. Therefore $\cat{S}_X$ contains $\Loc{ \cat{F} \otimes \bigT }$, as claimed.
\end{proof}

\begin{Lem}
\label{Lem:left-right=id}
$\cat U =  \tensLoc{\iota(\iota^{-1} \cat U)}$ for every localizing $\otimes$-ideal $\cat U \subseteq \bigT$. 
\end{Lem}

\begin{proof}
By stratification (\Cref{Thm:strat-minimality}\,(b)), we have the equality
\[
\cat U = \tensLoc{ g(\cat P) \mid \cat P \in \supp (\cat U) }
\]
for any localizing $\otimes$-ideal $\cat U$ of~$\bigT$, 
hence
\[
\cat U = \tensLoc{ \iota g(\iota^{-1} \cat P) \mid \cat P \in \supp (\cat U) }
\]
by \Cref{Lem:isomorphic-g's}. 
Since the objects $\iota g (\iota^{-1} \cat P)$ come from~$\cat T$ by construction and belong to~$\cat U$ by the above, this immediately implies that $\cat U \subseteq \tensLoc{\iota (\iota^{-1}) \cat U}$.
The reverse inclusion is clear since $\cat U$ is a localizing $\otimes$-ideal containing~$\iota (\iota^{-1} \cat U)$.
\end{proof}

\begin{Lem}
\label{Lem:right-left=id}
$\cat L = \iota^{-1} (\tensLoc{\iota \cat L})$ for every localizing tensor-ideal $\cat L \subseteq \cat T$. 
\end{Lem}

\begin{proof}
The inclusion $\cat L \subseteq \iota^{-1}( \tensLoc{\iota \cat L})$ is immediate, since $\iota \cat L \subseteq \tensLoc{\iota \cat L}$.

It remains to show $\iota^{-1} (\tensLoc{\iota \cat L}) \subseteq \cat L$.
By \Cref{Lem:technical-tensor-generation}, we have $\tensLoc{\iota \cat L} = \Loc{ \iota \cat L \otimes \iota \cat T^c} = \Loc{ \iota ( \cat L \otimes \cat T^c)} = \Loc { \iota \cat L }$, the latter equality because $\cat L$ is a tensor ideal of~$\cat T$.
We are thus reduced to proving that $\iota^{-1} (\Loc{\iota \cat L}) \subseteq \cat L$.

The latter claim is non obvious. In order to prove it, we will appeal to Neeman's theory of well-generated triangulated categories (\cite{Neeman10}), as presented in~\cite{Krause10}. For the remainder of the proof we will adopt the terminology of \emph{loc.\,cit.}

By hypothesis, the triangulated category $\bigT$ is compactly generated, that is $\aleph_0$-well generated. Since compact (\ie $\aleph_0$-compact) objects are also $\aleph_1$-compact, $\bigT$ is also $\aleph_1$-well generated with $\iota \cat T^c = \bigT^c$ as an (essentially) small set of $\aleph_1$-compact generators.
Thus by \cite[Cor.\,7.2.2]{Krause10} its full subcategory $(\bigT)^{\aleph_1}$ of all $\aleph_1$-compact objects in~$\bigT$ is precisely the $\aleph_1$-localizing ($=$~localizing${}_{\aleph_1}$) subcategory generated by~$\iota \cat T^c$, that is, $\iota \cat T$.
Incidentally, it also follows that the subcategory $\iota \cat L \subseteq \iota \cat T$ consists of $\aleph_1$-compact objects. 
Therefore by \cite[Thm.\,7.2.1]{Krause10}, the localizing subcategory $\Loc{\iota \cat L} \subseteq \bigT$ is itself an $\aleph_1$-well generated triangulated category, and its full subcategory $(\Loc{\iota \cat L})^{\aleph_1}$ of $\aleph_1$-compact objects equals $\Loc{\iota \cat L} \cap (\bigT)^{\aleph_1}$, that is, by the above discussion, $\Loc{\iota \cat L} \cap \iota \cat T$.

But another application of \cite[Cor.\,7.2.2]{Krause10} implies that $(\Loc{\iota \cat L})^{\aleph_1}$ is also the $\aleph_1$-localizing subcategory generated by~$\iota \cat L$, which is $\iota \cat L$ itself.
We conclude that $\Loc{\iota \cat L} \cap \iota \cat T = \iota \cat L$, as claimed.
\end{proof}

We can now conclude the proof of \Cref{Prop:strat-restriction}.
By Lemmas \ref{Lem:right-left=id} and~\ref{Lem:left-right=id}, the top horizontal arrows in the square \eqref{eq:bijections-square} are mutually inverse bijections. 
By \Cref{Lem:square-commutes} the square is commutative, and we deduce that the left vertical arrow is a bijection, since the other three are. In other words, $\cat T$ is countably stratified.
\end{proof}

\bibliographystyle{alpha}

\begin{thebibliography}{BCH{\etalchar{+}}23}

\bibitem[AK18]{AranoKubota18}
Yuki Arano and Yosuke Kubota.
\newblock A categorical perspective on the {A}tiyah-{S}egal completion theorem
  in {KK}-theory.
\newblock {\em J. Noncommut. Geom.}, 12(2):779--821, 2018.

\bibitem[Bal05]{Balmer05a}
Paul Balmer.
\newblock The spectrum of prime ideals in tensor triangulated categories.
\newblock {\em J. Reine Angew. Math.}, 588:149--168, 2005.

\bibitem[Bal10a]{Balmer10b}
Paul Balmer.
\newblock Spectra, spectra, spectra -- tensor triangular spectra versus
  {Z}ariski spectra of endomorphism rings.
\newblock {\em Algebr. Geom. Topol.}, 10(3):1521--1563, 2010.

\bibitem[Bal10b]{BalmerICM}
Paul Balmer.
\newblock Tensor triangular geometry.
\newblock In {\em International {C}ongress of {M}athematicians, Hyderabad
  (2010), {V}ol. {II}}, pages 85--112. Hindustan Book Agency, 2010.

\bibitem[Bal14]{Balmer14}
Paul Balmer.
\newblock Splitting tower and degree of tt-rings.
\newblock {\em Algebra Number Theory}, 8(3):767--779, 2014.

\bibitem[Bal16]{Balmer16}
Paul Balmer.
\newblock Separable extensions in tensor-triangular geometry and generalized
  {Q}uillen stratification.
\newblock {\em Ann. Sci. \'Ec. Norm. Sup\'er. (4)}, 49(4):907--925, 2016.

\bibitem[Bal18]{Balmer18}
Paul Balmer.
\newblock On the surjectivity of the map of spectra associated to a
  tensor-triangulated functor.
\newblock {\em Bull. Lond. Math. Soc.}, 50(3):487--495, 2018.

\bibitem[Bal20]{Balmer20}
Paul Balmer.
\newblock A guide to tensor-triangular classification.
\newblock In {\em Handbook of homotopy theory}, CRC Press/Chapman Hall Handb.
  Math. Ser., pages 145--162. CRC Press, Boca Raton, FL, [2020] \copyright
  2020.

\bibitem[BBB24]{BBB24pp}
Scott Balchin, David Barnes, and Tobias Barthel.
\newblock Profinite equivariant spectra and their tensor-triangular geometry.
\newblock Preprint, \url{https://arxiv.org/abs/2401.01878}, 2024.

\bibitem[BCH{\etalchar{+}}23]{BCHNP23pp}
Tobias Barthel, Natalia Castellana, Drew Heard, Niko Naumann, and Luca Pol.
\newblock Quillen stratification in equivariant homotopy theory.
\newblock \url{https://arxiv.org/abs/2301.02212}, 2023.

\bibitem[BCHS24]{BCHS24pp}
Tobias Barthel, Natalia Castellana, Drew Heard, and Beren Sanders.
\newblock Cosupport in tensor triangular geometry.
\newblock Preprint, \url{https://arxiv.org/abs/2303.13480}, 2024.

\bibitem[BDM24]{BDM24pp}
Serge Bouc, Ivo Dell'Ambrogio, and Rub\'en Martos.
\newblock A generalized {G}reenlees--{M}ay splitting principle.
\newblock Preprint, 2024.

\bibitem[BDS15]{BDS15}
Paul Balmer, Ivo Dell'Ambrogio, and Beren Sanders.
\newblock Restriction to finite-index subgroups as \'etale extensions in
  topology, {KK}-theory and geometry.
\newblock {\em Algebr. Geom. Topol.}, 15(5):3025--3047, 2015.

\bibitem[BDS16]{BDS16}
Paul Balmer, Ivo Dell'Ambrogio, and Beren Sanders.
\newblock Grothendieck-{N}eeman duality and the {W}irthm\"{u}ller isomorphism.
\newblock {\em Compos. Math.}, 152(8):1740--1776, 2016.

\bibitem[BEL23]{BEL23pp}
Ulrich Bunke, Alexander Engel, and Markus Land.
\newblock A stable $\infty$-category for equivariant {KK}-theory.
\newblock Preprint, \url{https://arxiv.org/abs/2102.13372}, 2023.

\bibitem[BF11]{BalmerFavi11}
Paul Balmer and Giordano Favi.
\newblock Generalized tensor idempotents and the telescope conjecture.
\newblock {\em Proc. Lond. Math. Soc. (3)}, 102(6):1161--1185, 2011.

\bibitem[BHS23]{BHS23}
Tobias Barthel, Drew Heard, and Beren Sanders.
\newblock Stratification in tensor triangular geometry with applications to
  spectral {M}ackey functors.
\newblock {\em Camb. J. Math.}, 11(4):829--915, 2023.

\bibitem[BIK08]{BensonIyengarKrause08}
Dave~J. Benson, Srikanth~B. Iyengar, and Henning Krause.
\newblock Local cohomology and support for triangulated categories.
\newblock {\em Ann. Sci. \'Ec. Norm. Sup\'er. (4)}, 41(4):573--619, 2008.

\bibitem[BKS19]{BKS19}
Paul Balmer, Henning Krause, and Greg Stevenson.
\newblock Tensor-triangular fields: ruminations.
\newblock {\em Selecta Math. (N.S.)}, 25(1):Paper No. 13, 36, 2019.

\bibitem[Bla98]{Blackadar98}
Bruce Blackadar.
\newblock {\em K-theory for operator algebras}, volume~5.
\newblock Cambridge University Press, 1998.

\bibitem[BN93]{BoekstedtNeeman93}
Marcel B{\"o}kstedt and Amnon Neeman.
\newblock Homotopy limits in triangulated categories.
\newblock {\em Compositio Math.}, 86(2):209--234, 1993.

\bibitem[CDLS93]{CDLS93}
Kai~Nah Cheng, M.~Deaconescu, Mong-Lung Lang, and Wu~Jie Shi.
\newblock Corrigendum and addendum to: ``{C}lassification of finite groups with
  all elements of prime order'' [{P}roc.\ {A}mer.\ {M}ath.\ {S}oc.\ {\bf 106}
  (1989), no.\ 3, 625--629; {MR}0969518 (89k:20038)] by {D}eaconescu.
\newblock {\em Proc. Amer. Math. Soc.}, 117(4):1205--1207, 1993.

\bibitem[Dea89]{Deaconescu89}
Marian Deaconescu.
\newblock Classification of finite groups with all elements of prime order.
\newblock {\em Proc. Amer. Math. Soc.}, 106(3):625--629, 1989.

\bibitem[Del10]{DellAmbrogio10}
Ivo Dell'Ambrogio.
\newblock Tensor triangular geometry and {$KK$}-theory.
\newblock {\em J. Homotopy Relat. Struct.}, 5(1):319--358, 2010.

\bibitem[Del11]{DellAmbrogio11}
Ivo Dell'Ambrogio.
\newblock Localizing subcategories in the bootstrap category of separable
  {$C^\ast$}-algebras.
\newblock {\em J. K-Theory}, 8(3):493--505, 2011.

\bibitem[Del14]{DellAmbrogio14}
Ivo Dell'Ambrogio.
\newblock Equivariant {K}asparov theory of finite groups via {M}ackey functors.
\newblock {\em J. Noncommut. Geom.}, 8(3):837--871, 2014.

\bibitem[Del22]{DellAmbrogio22}
Ivo Dell'Ambrogio.
\newblock Green 2-functors.
\newblock {\em Trans. Am. Math. Soc.}, 375(11):7783--7829, 2022.

\bibitem[DEM14]{DEM14}
Ivo Dell'Ambrogio, Heath Emerson, and Ralf Meyer.
\newblock An equivariant {L}efschetz fixed-point formula for correspondences.
\newblock {\em Doc. Math.}, 19:141--194, 2014.

\bibitem[DM]{DellAmbrogioMartos24pp}
Ivo Dell'Ambrogio and Rub\'{e}n Martos.
\newblock The {B}aum--{C}onnes conjecture via support theory.
\newblock {I}n preparation.

\bibitem[DM21]{DellAmbrogioMeyer21}
Ivo Dell'Ambrogio and Ralf Meyer.
\newblock The spectrum of equivariant {Kasparov} theory for cyclic groups of
  prime order.
\newblock {\em Ann. \(K\)-Theory}, 6(3):543--558, 2021.

\bibitem[DS16]{DellAmbrogioStanley16}
Ivo Dell'Ambrogio and Donald Stanley.
\newblock Affine weakly regular tensor triangulated categories.
\newblock {\em Pacific J. Math.}, 285(1):93--109, 2016.

\bibitem[Efi24]{Efimov24pp}
Alexander~I. Efimov.
\newblock {K}-theory and localizing invariants of large categories.
\newblock \url{https://arxiv.org/abs/2405.12169}, 2024.

\bibitem[G{\'{o}}m23]{Gomez23pp}
Juan~Omar G{\'{o}}mez.
\newblock A family of infinite degree tt-rings.
\newblock Preprint, \url{https://arxiv.org/abs/2307.00768}, 2023.

\bibitem[Hin16]{Hinich16}
Vladimir Hinich.
\newblock Dwyer-{K}an localization revisited.
\newblock {\em Homology Homotopy Appl.}, 18(1):27--48, 2016.

\bibitem[HPS97]{HoveyPalmieriStrickland97}
Mark Hovey, John~H. Palmieri, and Neil~P. Strickland.
\newblock Axiomatic stable homotopy theory.
\newblock {\em Mem. Amer. Math. Soc.}, 128(610), 1997.

\bibitem[Kas88]{Kasparov88}
G.~G. Kasparov.
\newblock Equivariant {$KK$}-theory and the {N}ovikov conjecture.
\newblock {\em Invent. Math.}, 91(1):147--201, 1988.

\bibitem[Kra10]{Krause10}
Henning Krause.
\newblock Localization for triangulated categories.
\newblock In {\em Triangulated categories}, volume 375 of {\em London Math.
  Soc. Lecture Note Ser.}, pages 161--235. Cambridge Univ. Press, Cambridge,
  2010.

\bibitem[Lau22]{Eike21pp}
Eike Lau.
\newblock The {B}almer spectrum of certain {D}eligne--{M}umford stacks.
\newblock Preprint, 34 pages. \texttt{arXiv:2101.01446v2 [math.AG]}, 2022.

\bibitem[Lur09]{LurieHTT}
Jacob Lurie.
\newblock {\em Higher topos theory}, volume 170 of {\em Annals of Mathematics
  Studies}.
\newblock Princeton University Press, Princeton, NJ, 2009.

\bibitem[Lur17]{LurieHA}
Jacob Lurie.
\newblock Higher algebra.
\newblock \url{https://www.math.ias.edu/~lurie/papers/HA.pdf}, 18 September
  2017.

\bibitem[Mey00]{Meyer00}
Ralf Meyer.
\newblock Equivariant {K}asparov theory and generalized homomorphisms.
\newblock {\em $K$-Theory}, 21(3):201--228, 2000.

\bibitem[Mey08]{Meyer08}
Ralf Meyer.
\newblock Categorical aspects of bivariant {$K$}-theory.
\newblock In {\em {$K$}-theory and noncommutative geometry}, EMS Ser. Congr.
  Rep., pages 1--39. Eur. Math. Soc., Z\"urich, 2008.

\bibitem[MN06]{MeyerNest06}
Ralf Meyer and Ryszard Nest.
\newblock The {B}aum-{C}onnes conjecture via localisation of categories.
\newblock {\em Topology}, 45(2):209--259, 2006.

\bibitem[MN24]{MeyerNadareishvili24pp}
Ralf Meyer and George Nadareishvili.
\newblock A universal coefficient theorem for actions of finite groups on
  {C}*-algebras.
\newblock Preprint \url{https://arxiv.org/abs/2406.11787}, 2024.

\bibitem[Nee92a]{Neeman92a}
Amnon Neeman.
\newblock The chromatic tower for {$D(R)$}.
\newblock {\em Topology}, 31(3):519--532, 1992.

\bibitem[Nee92b]{Neeman92b}
Amnon Neeman.
\newblock The connection between the {$K$}-theory localization theorem of
  {T}homason, {T}robaugh and {Y}ao and the smashing subcategories of
  {B}ousfield and {R}avenel.
\newblock {\em Ann. Sci. \'Ecole Norm. Sup. (4)}, 25(5):547--566, 1992.

\bibitem[Nee96]{Neeman96}
Amnon Neeman.
\newblock The {G}rothendieck duality theorem via {B}ousfield's techniques and
  {B}rown representability.
\newblock {\em J. Amer. Math. Soc.}, 9(1):205--236, 1996.

\bibitem[Nee01]{Neeman01}
Amnon Neeman.
\newblock {\em Triangulated categories}, volume 148 of {\em Annals of
  Mathematics Studies}.
\newblock Princeton University Press, 2001.

\bibitem[Nee10]{Neeman10}
Amnon Neeman.
\newblock Derived categories and {G}rothendieck duality.
\newblock In {\em Triangulated categories}, volume 375 of {\em London Math.
  Soc. Lecture Note Ser.}, pages 290--350. Cambridge Univ. Press, Cambridge,
  2010.

\bibitem[RS86]{RosenbergSchochet86}
J.~Rosenberg and C.~Schochet.
\newblock {T}he {K}\"{u}nneth theorem and the universal coefficient theorem for
  equivariant ${K}$-theory and ${K}{K}$-theory.
\newblock {\em {A}merican {M}athematical {S}oc.}, 348, 1986.

\bibitem[San22]{Sanders22}
Beren Sanders.
\newblock A characterization of finite \'{e}tale morphisms in tensor triangular
  geometry.
\newblock {\em \'{E}pijournal G\'{e}om. Alg\'{e}brique}, 6:Art. 18, 25, 2022.

\bibitem[Seg68]{Segal68a}
Graeme Segal.
\newblock The representation ring of a compact {L}ie group.
\newblock {\em Inst. Hautes \'Etudes Sci. Publ. Math.}, (34):113--128, 1968.

\end{thebibliography}
\newcommand{\etalchar}[1]{$^{#1}$}

\printindex
\end{document}